%% file: Unordered_flag_varieties.tex
\documentclass[a4paper]{amsart}
\usepackage[english]{babel}
\usepackage[T1]{fontenc}
\usepackage[latin1]{inputenc}
\usepackage{amssymb}
\usepackage{amsmath}
\usepackage{amsthm}
\usepackage{amscd}
\usepackage{mathrsfs}
\usepackage{bm}
\usepackage{tikz-cd}
\usepackage{pgf,tikz}
\usepackage{mathrsfs}
\usetikzlibrary{arrows}
\usepackage{pgfplots}
\pgfplotsset{compat=1.15}
\usepackage{float}
\usepackage{verbatim}
\usepackage[all]{xy}
\usepackage[cal=boondoxo]{mathalfa}
\usepackage[numbers]{natbib}         % bibliography style
\usepackage[colorlinks]{hyperref}    % better urls in bibliography
\usepackage{color}                   %May be necessary if you want to color links
\usepackage{hyperref}
\hypersetup{
    colorlinks=true,          %set true if you want colored links
    linktoc=all,              %set to all if you want both sections and subsections linked
    linkcolor=blue,           %choose some color if you want links to stand out
}

\DeclareMathOperator{\rk}{rk}

\DeclareMathOperator{\sgn}{sgn}

\DeclareMathOperator{\gr}{gr}

\DeclareMathOperator{\Prim}{Prim}

\DeclareMathOperator{\Fl}{Fl}
\DeclareMathOperator{\Bgroup}{B}

\DeclareMathOperator{\characteristic}{char}
\DeclareMathOperator{\res}{res}

\DeclareMathOperator{\Ind}{Ind}

\newtheorem{theorem}{Theorem}[section]
\newtheorem{proposition}[theorem]{Proposition}
\newtheorem{corollary}[theorem]{Corollary}
\newtheorem{lemma}[theorem]{Lemma}
\newtheorem*{theorem*}{Theorem}

\theoremstyle{definition}
\newtheorem{definition}[theorem]{Definition}

\theoremstyle{remark}
\newtheorem{remark}[theorem]{Remark}

\newtheorem{example}[theorem]{Example}

\newcommand{\Ftwo}{\ensuremath{\mathbb{F}_2}}
\newcommand{\Fp}{\ensuremath{\mathbb{F}_p}}
\newcommand{\Bgrouppos}[1]{\ensuremath \Bgroup_{#1}^{+}}

\newcommand{\N}{\mathbb{N}}
\newcommand{\Z}{\mathbb{Z}}

\newcommand{\R}{\mathbb{R}}
\newcommand{\F}{\mathbb{F}}
\newcommand{\C}{\mathbb{C}}

\newcommand{\UFl}{\overline{\mathrm{Fl}}}
\newcommand{\AB}{A'_{\Bgrouppos{}}}

\begin{document}
	
\author[L. Guerra]{Lorenzo Guerra }
\address{Universit\`a di Roma Tor Vergata}
\email{guerra@mat.uniroma2.it}
\author[S. Jana]{Santanil Jana}
\address{Mathematics Department, 
	University of British Columbia}
\email{santanil@math.ubc.ca}
%\author[P. Salvatore]{Paolo Salvatore}
%\address{Universit\`a di Roma Tor Vergata}
%\email{salvator@mat.uniroma2.it}

\begin{abstract}
We consider quotients of complete flag manifolds in $ \mathbb{C}^n $ and $ \mathbb{R}^n $ by an action of the symmetric group on $ n $ objects. We compute their cohomology with field coefficients of any characteristic. Specifically, we show that these topological spaces exhibit homological stability and we provide a closed-form description of their stable cohomology rings. We also describe a simple algorithmic procedure to determine their unstable cohomology additively.
\end{abstract}
\thanks{The first author acknowledges the 
	MIUR Excellence Department Project awarded to the Department of Mathematics, University of Rome Tor Vergata, CUP E83C18000100006. 
} 

\subjclass[2020]{Primary 14M15; Secondary 55N10, 55N91, 55P47, 55R20, 55R45, 55S12, 55T10, 20J06.}
\title{Cohomology of complete unordered flag manifolds}
\maketitle

\input{Section1} % Introduction
\tableofcontents
\input{Section2} % Recollection of the cohomology of extended symmetric powers
\input{Section3} % The mod 2 cohomology of certain subgroups of the hyperoctahedral groups
\input{Section4} % Stable cohomology of unordered flag varieties      (we might want to have two subsections, one for complex and one for real case)
\input{Section5} % The pullback formula for characteristic classes and the Serre spectral sequence for unordered flag manifolds
\input{Section6} % Unstable cohomology of unordered flag varieties    (again, we might want to have two subsections)

\bibliographystyle{plain}
\bibliography{bibliografia}
	
\end{document}

%% file: Section1.tex
% !TEX root = Unordered flag varieties.tex

\section{Introduction}

A \textit{complete flag} over a field $\F$ is a nested sequence of  $\F$-linear subspaces of $\F^n$: \begin{equation} \label{eq:flag}
    \{0\} = V_0 \subset V_1 \subset \cdots \subset V_n = \F^n 
    \end{equation} 
such that $\mathrm{dim}_{\F} (V_j) = j$, for all $j=1,2,\dots,n$. The space of all complete flags of order $n$ over $\F$ forms an algebraic variety called the \textit{full flag variety} of order $n$ over $\F$, denoted as $\Fl_n (\F)$. In this paper, our focus is exclusively on cases where $\F$ corresponds to the fields $\R$ and $\C$. Over these fields, $ \Fl_n(\F) $ endowed with the analytic topology is a topological manifold, that we call \textit{real and complex flag manifold}, respectively. The unitary group $U(n)$ (or the orthogonal group $O(n)$) acts transitively on $\Fl_n (\C)$ (or $\Fl_n (\R)$). In other words, given any flag $\{ V_j\}_{j=0}^n$ as in (\ref{eq:flag}), there exists a $g\in U(n)$ (or $O(n)$) which maps the flag (\ref{eq:flag}) to the standard flag of $\C^n$ (or $\R^n$): \[  \{0\} \subset \C \{e_1\}  \subset \cdots \subset \C \{e_1, \dots, e_n\} = \C^n,  \]
where $\{ e_1, \dots, e_n\}$ is the standard basis of $\C^n$. Therefore, complex (or real) flags can be identified with the elements of $U(n)$ (or $O(n)$), modulo the subgroup that leaves the standard flag fixed. Therefore, the flag manifold $\Fl_n (\C)$ (or $\Fl_n (\R)$) can be identified with $U(n)/T(n)$ (or $O(n)/T(n)$), where $T(n)$ is the maximal torus in $U(n)$ (or $O(n)$). 

Given a flag $\{ V_j\}_{j=0}^n$ as in (\ref{eq:flag}), there is an orthonormal basis $\{ v_1, \dots,v_n \}$ of $\C^n$ (or $\R^n$) such that \[ V_j = \C \{ v_1,\dots,v_j\} \text{ }(\text{or } V_j = \R \{ v_1,\dots,v_j\} ) . \] An ordered orthonormal basis of $\C^n$ (or $\R^n$) describes a complete flag. Any element in $\Fl_n (\C)$ (or $\Fl_n (\R)$) can be described as a set of $n$ ordered mutually orthogonal lines in $\C^n$ (or $\R^n$). The symmetric group $\Sigma_n$ acts freely on $\Fl_n (\C)$ (or $\Fl_n (\R)$) by permuting the ordered basis elements. Note that the action of $\Sigma_n$ on $\Fl_n (\C)$ (or $\Fl_n (\R)$) does not permute the subspaces $V_j$, rather it permutes the basis elements $v_j$.

\begin{definition}
    The complex (or real) \textit{unordered flag manifold} of order $n$ is defined as the quotient $\Fl_n (\C)/\Sigma_n$ (or $\Fl_n (\R)/\Sigma_n$) and denoted as $\UFl_n (\C)$ (or $\UFl_n (\R)$). 
\end{definition}

Let $N(n) := N_{U(n)} (T(n))$ be the normalizer of the maximal torus $T(n)$ in $U(n)$. The Weyl group of the maximal torus is $N(n)/T(n) \cong \Sigma_n$. The unordered flag manifold of order $n$ can also be described as \begin{equation} \label{eq:flagidentity}
    \UFl_n (\C)= \Fl_n (\C)/\Sigma_n \cong U(n)/N(n) . 
\end{equation}
Similarly $\Bgroup_n := N_{O(n)} (T(n))$ be the normalizer of the maximal torus in $O(n)$. Then we can identify $O(n) / \Bgroup_n$ with $\Fl_n (\R)$. In \S \ref{section:4.2} we use the alternating subgroup $\Bgrouppos{n} := N_{SO(n)} (T(n) \cap SO(n))$ (so that $ \UFl_n (\mathbb{R}) \cong SO(n)/\Bgrouppos{n}$) for our computations.

In this paper, we delve into the study of the stable and unstable cohomology of complex unordered flag manifolds. While complete flag manifolds have been extensively studied in the fields of algebraic topology and geometry due to their significance in Lie theory, unordered complete flag manifolds have received less attention from algebraic topologists. Understanding the topology of these unordered flag manifolds is not only intriguing in its own right, but also carries important implications for emerging problems in algebraic topology and convex geometry.

The cohomology of the unordered flag manifolds bears relevance to the computation of the cohomology of the classifying spaces for commutativity. The classifying space for commutativity $B_{com} G$ and the total space of the associated principal $G$-bundle, $E_{com} G$, was described as a homotopy colimit over a poset by Adem--G{\'o}mez \cite{Adem-Gomez}. This homotopy colimit diagram involves the unordered flag manifold $\UFl_n (\C)$ (or $\UFl_n (\R)$) when $G$ is $U(n)$ (or $O(n)$). The cohomology of these spaces, particularly the $p$-torsion structure, has implications for the study of spaces of homomorphism. The study of the cohomology of unordered flag manifolds is also connected to the resolution of a conjecture proposed by Atiyah \cite{Atiyah2002}. However, providing a comprehensive discussion of this conjecture is beyond the scope of the paper. For further information and in-depth analysis, we refer readers to \cite[\S 4]{Atiyah2002}. 

The cohomology of unordered flag manifolds also has implications in estimating the number of Auerbach bases of finite dimensional Banach spaces. Let $\mathbb{X}$ be a $n$ dimensional complex (or real) Banach space and $S_{\mathbb{X}}$ denote its unit sphere. A basis $\mathcal{B}=\{v_1,\dots,v_n\}$ of $\mathbb{X}$ is called an \textit{Auerbach basis} if $v_i \in S_{\mathbb{X}}$ and there is a basis ${v^1,\dots,v^n}$ of the dual space $\mathbb{X}^{*}$ satisfying
\begin{equation*}
 v^i(v_j) = \delta_{ij},\quad \text{and} \quad v^i \in S_{\mathbb{X}^{*}} \text{ for} \ i,j = 1,2,\dots,n.
\end{equation*}
Weber--Wojciechowski \cite{Weber2016} provided an estimate of the number of Auerbach bases of a finite-dimensional Banach space using topological methods. Here, one identifies bases that differ only by permutation or multiplication by scalars of absolute value one. In other words, two bases are said to be equivalent if they lie in the same orbit of the action of $N (n)$ (or $\Bgroup_n$) on $U(n)$ (or $O(n)$). In \cite{Weber2016}, the estimates of the Lusternick--Schnirelmann category \cite{Fox1939} $\mathrm{cat} ({\UFl}_n (\mathbb{R}))$ and $\mathrm{rank} \big ( H^*({\UFl}_n (\mathbb{R})) \big )$ along with their complex counterparts were obtained using known results about the cohomology of ordered flag manifolds. The estimates for $\mathrm{rank} \big ( H^*({\UFl}_n (\mathbb{R})) \big )$ and $\mathrm{cat}({\UFl}_n (\mathbb{C}))$ has been much-improved by the computational results about the cohomology of the unordered flag varieties for lower orders in \cite{G-J-M}. 

The paper is organized as follows: \begin{itemize}
    \item In section 2, we present an overview of the recent developments on the mod $p$ cohomology of extended powers following the work of Guerra--Salvatore--Sinha \cite{Guerra-Salvatore-Sinha}. The cohomology of all extended powers of a space together has the structure of a component Hopf ring with divided powers. This structure allows for a more efficient computation of cup products compared to doing so individually. This is an extension of previous works \cite{Sinha:12, Guerra:17} on the cohomology of symmetric groups, which are the cohomology of extended powers of a point. 
    \item In Section 3, we provide an overview of the cohomology of the alternating subgroups $\Bgrouppos{n}$ of the Hyperoctahedral groups. We review the mod $2$ cohomology of $\Bgrouppos{n}$ previously discussed in \cite{Guerra-Santanil}. Describing the cohomology of $\Bgrouppos{n}$ at odd primes is relatively straightforward. We show in Theorem~\ref{thm:cohomology alternating group mod p} that it is isomorphic to the cohomology of the symmetric groups as Hopf rings.  
    \item In Section 4, we focus on studying the stable behavior of cohomology of the unordered flag manifolds, namely $\{ \UFl_n (\C) \}_n$ and $\{\UFl_n (\R) \}_n$. Proposition~\ref{prop:hom-stab-UN} establishes that unordered flag manifolds exhibit homological stability, which enables us to discuss stable cohomology. We examine the complex and real cases separately. In the complex case, we prove a pullback formula in Theorem~\ref{pullbackformula} for the Chern classes. This formula serves as the foundation for our main result concerning the stable cohomology of complex unordered manifolds in Corollary~\ref{cor:stable-cohomology}. In the real case, we deal with the mod $2$ and mod $p$ cohomologies separately. To determine the stable cohomology $H^* (\UFl_{\infty} (\R); \F_2)$, it is essential to determine the stable cohomology of $\Bgrouppos{n}$. We describe this in Corollary~\ref{cor:homological stability alternating subgroups}, which is a consequence of the main results in Section 3. In both the mod $2$ and the mod $p$ cases, the pullback formulas for the Stiefel-Whitney classes in Theorem~\ref{pullbackformula real} and Pontrjagin classes in Theorem~\ref{pullback pontrjagin}, respectively, are critical for computing the cohomology of real unordered flag manifolds. We also describe the Poincar\'e series of the stable cohomology of unordered flag manifolds.
    \item In Section 5, we lay the groundwork for the spectral sequence argument to compute the unstable cohomology of the unordered flag manifolds. Specifically, we study the Serre spectral sequence that computes $H^* (\UFl_n (\C); \F)$ and $H^* (\UFl_n (\R); \F)$, and describe the differentials in terms of the fullback formulas that were derived in Section 4. 
    \item Finally, in Section 6, we present the mod $2$ unstable cohomology computations of the unordered flag manifolds of low ($n =3,4,5$) orders. Additionally, we also compute the mod $p$ cohomology of $\UFl_p (\C)$ and $\UFl_p (\R)$ for all prime $p>2$. 
\end{itemize}

\textbf{Convention.} In the rest of this paper, $\F$ will denote a field and $p$ will denote a prime. The cohomological degree of a cohomology class $\gamma$ will be denoted by $|\gamma|$.\\ 

\textbf{Acknowledgements.} The first author would like to thank Prof. Paolo Salvatore for the helpful conversation and for suggesting the use of regular sequences, and Prof. Ben Knudsen for an email exchange about representation stability. He also acknowledges the MIUR Excellence Department Project awarded to the Department of Mathematics, University of Rome Tor Vergata, CUP E83C18000100006.

The second author expresses their gratitude to Prof. Alejandro Adem for engaging discussions that significantly contributed to the research presented in this paper.

%%% Local Variables:
%%% TeX-master: "Unordered flag varieties.tex"
%%% End:

%% file: Section2.tex
% !TEX root = Unordered flag varieties.tex

\section{Recollection of the cohomology of extended powers}
\label{sec:cohomology extended powers}

Let $ X $ be a Hausdorff compactly generated topological space. Define its $ n $-fold extended (symmetric) power $ D_n(X_+) $ as the homotopy quotient of $ X^n $ with respect to the action of the symmetric group $ \Sigma_n $ that permutes the cartesian factors.
In particular, if $ \{*\} $ is a space with one point, $ D_n(\{*\}_+) = B \Sigma_n $, the classifying space of the symmetric group itself. %We also let $ D(X_+) := \bigsqcup_{ n \geq 0} D_n(X_+) $.

We are mostly interested in the cases where $ \F = \mathbb{F}_p $ for some prime number $ p $ or $ \F = \mathbb{Q} $. There will be some results that are valid only for these particular choices; we will stress explicitly this fact where it happens.
Let $ A_X = \bigoplus_{n,d} A_X^{n,d} $ be the bigraded $ \F $-vector space defined by $ A_X^{n,d}= H^d(D_n(X_+)) $. We refer to the indices $ n $ and $ d $ as the component and (cohomological) dimension, respectively.
On $ A_X $ there are some structural morphisms, all natural in $ X $, that provide it with a rich algebraic structure:
\begin{itemize}
	\item the usual cup product, component by component, $ \cdot \colon A_X \otimes A_X \to A_X $.
	\item a coproduct $ \Delta \colon A_X \to A_X \otimes A_X $ induced by the maps $ p_{n,m} \colon D_n(X_+) \times D_m(X_+) \to D_{n+m}(X_+) $ defined by passing to quotients the following map:
	\begin{align*}
		&(p,(x_1,\dots,x_n)) \times (q,(x_{n+1},\dots,x_{n+m})) \in (E \Sigma_n \times X^n) \times (E \Sigma_m \times X^m) \\
		&\longmapsto (i_{n,m}(p,q),(x_1,\dots,x_n,x_{n+1},\dots,x_{n+m})) \in E \Sigma_{n+m} \times X^{n+m},
	\end{align*}
	where $ i_{n,m} \colon E \Sigma_n \times E \Sigma_m \to E \Sigma_{n+m} $ is the $ \Sigma_n \times \Sigma_m $-equivariant map induced by the standard inclusion $ \Sigma_n \times \Sigma_m \hookrightarrow \Sigma_{n+m} $.
	\item since $ p_{n,m} $ is homotopy equivalent to a finite covering, there is a product $ \odot \colon A_X \otimes A_X \to A_X $ corresponding to the cohomological transfer maps of $ p_{n,m} $.
\end{itemize}

\begin{definition} \label{def:Hopf ring}
A commutative component bigraded Hopf ring is a bigraded vector space $ A = \bigoplus_{n,d} A^{n,d} $ with a coproduct $ \Delta \colon A \to A \otimes A $, two products $ \odot, \cdot \colon A \otimes A \to A $, a unit $ \eta \colon \F \to A $, a counit $ \varepsilon \colon A \to \F $, and an antipode $ S \colon A \to A $ such that
\begin{itemize}
\item $ \Delta, \odot, \eta, \varepsilon $ are bigraded, $ \cdot $ is graded with respect to $ d $, and the products and the coproduct are graded commutative with respect to $ d $,
\item $ (A,\Delta,\odot,\eta,\varepsilon) $ is a Hopf algebra,
\item $ (A, \cdot,\Delta) $ is a bialgebra,
\item $ \forall x \in A^{n,d}, x' \in A^{n',d'}\colon x \cdot x' = 0 $ if $ n \not= n' $,
\item and the following Hopf ring distributivity axiom holds for all $ x,y,z \in A $, where we use Sweedler's notation $ \Delta(x) = \sum x_{(1)} \otimes x_{(2)} $:
\[
x \cdot (y \odot z) = \sum (-1)^{d(y)d(x_{(2)})} (x_{(1)} \cdot y) \odot (x_{(2)} \cdot z).
\] 
\end{itemize}
\end{definition}

\begin{theorem}[{\cite[Corollary 2.31]{Guerra-Salvatore-Sinha}}] \label{thm:Hopf ring structure}
$ A_X $, with the morphisms defined above, the unit and counit that identify $ A_X^{(0,0)} $ with $ \F $, and an antipode given by multiplication by $ \pm 1 $ depending on the component, is a commutative bigraded component Hopf ring.
\end{theorem}

In the case where $ \F = \mathbb{F}_p $ or $ \F = \mathbb{Q} $, there are distinguished classes in $ A_X $:
\begin{itemize}
	\item The first distinguished set of classes appears only if $ \F = \mathbb{F}_p $ for some prime $ p $. The unique map $ \pi \colon X \to \{*\} $ induces an injective homomorphism $ \pi^* \colon \bigoplus_{n \geq 0} H^*(\Sigma_n; \mathbb{F}_p) = A_{\{*\}} \to A_X $. By a slight abuse of notation, we identify the classes $ \gamma_{k,l} $ of Giusti--Salvatore--Sinha \cite{Sinha:12} (if $ p = 2 $) or the classes $ \gamma_{k,l} $, $ \alpha_{i,k} $ and $ \beta_{i,j,l} $ of Guerra \cite{Guerra:17} (if $ p > 2 $) with their image $ \pi^*(\gamma_{k,l}) $ under $ \pi^* $, thus realizing them as elements of $ A_X $.% In other words, the cohomology of $ \Sigma_n $ is mapped into $ H^*(D_n(X);\mathbb{F}_p) $ for all $ X $, and this allows to interpret a class in the cohomology of the symmetric group $ \Sigma_n $ as an element of the cohomology of $ D_n(X) $ for every topological space $ X $.
	\item The second distinguished set of classes appears for all choices of $ \F $. For every even-dimensional cohomology class $ \alpha \in H^*(X) $ and for all $ n > 0 $, the twisted cross product $ \alpha_{[n]} = 1 \times_{\Sigma_n} \alpha^{\times n} \in H^*(E \Sigma_n \times_{\Sigma_n} X^n) = H^*(D_n(X_+)) $ is well-defined. explicitly, the complex of singular chains of $ E \Sigma_n \times_{\Sigma_n} X^n $ is quasi-isomorphic to $ W_* \otimes_{\Sigma_n} C_*(X)^{\otimes n} $, where $ W_* $ is a free resolution of $ \F $ as a $ \F[\Sigma_n] $-module. Let $ \varepsilon \colon W_0 \to \F $ be the augmentation (representing the unit in $ H^0(\Sigma_n) $) and $ a $ be a cocycle representative of $ \alpha $. Then $ \alpha_{[n]} $ is represented by the $ \Sigma_n $-invariant cocycle $ \varepsilon \otimes a^{\otimes n} \colon W_* \otimes C_*(X)^{\otimes n} \to \F $, which does not depend on the chosen representative.
	We also define, by convention, $ \alpha_{[0]} = 1_0 $, the unit of the $ 0$-th component of $ A_X $.
	If $ \F = \mathbb{F}_2 $, then $ \varepsilon \otimes a^{\otimes n} $ is $ \Sigma_n $-invariant even if $ a $ is an odd-dimensional cochain; consequently, $ \alpha_{[n]} $ is defined for all $ \alpha \in H^*(X; \mathbb{F}_2) $, odd- and even-dimensional.
\end{itemize}

Guerra--Salvatore--Sinha proved that these distinguished classes suffice to fully determine $ A_X $ as a Hopf ring. We recall their results below.

\begin{theorem}[{\cite[Theorem 2.37]{Guerra-Salvatore-Sinha}}] \label{thm:cohomology DX mod 2}
	Assume that $ \F = \Ftwo $. As a commutative bigraded component Hopf ring, $ A_X $ is generated by the classes $ \gamma_{k,l} $ (for $ k,l \geq 1 $) and $ x_{n} $ (for $ x \in H^*(X) $ and $ n \geq 1 $), with the following relations:
	\begin{itemize}
		\item the relations of \cite[Theorem 1.2]{Sinha:12},
            \item $ x_{[n]}\cdot {x'}_{[n]} = (x \cdot x')_{[n]}$ for $x,x' \in H^*(X)$,
		\item $ \Delta(x_{[n]}) = \sum_{i=0}^n x_{[i]} \otimes x_{[n-i]}$ for $x \in H^*(X)$,
            \item $x_{[m]} \odot x_{[n]} = \binom{m+n}{n} x_{[m+n]}$ for $x\in H^* (X)$,
            \item $(\lambda x)_{[n]} = \lambda^n x_{[n]}$ and $(x+y)_{[n]} = \sum_{k=0}^n x_{[k]} \odot y_{[n-k]}$ for $x,y \in H^* (X)$ and $\lambda \in \F$.
	\end{itemize}
\end{theorem}

\begin{theorem}[{\cite[Theorem 2.38]{Guerra-Salvatore-Sinha}}, particular case] \label{thm:cohomology DX mod p}
	Assume that $ \F = \Fp $ with $ p $ odd prime and that $ H^*(X) $ is $ 0 $ in odd degrees. Then, as a commutative bigraded component Hopf ring, $ A_X $ is generated by the classes $ \gamma_{k,l} $ (for $ k,l \geq 1 $), $ \alpha_{i,k} $ (for $ 1 \leq i \leq k $), $ \beta_{i,j,k} $ (for $ 1 \leq i < j \leq k $) and $ x_{[n]} $ (for $ x \in H^*(X) $ and $ n \geq 1 $), with the following relations:\begin{itemize}
	\item the relations of \cite[Theorem 2.7]{Guerra:17},
	\item $ x_{[n]}\cdot {x'}_{[n]} = (x \cdot x')_{[n]}$ for $x,x' \in H^*(X)$,
	\item $ \Delta(x_{[n]}) = \sum_{i=0}^n x_{[i]} \otimes x_{[n-i]}$ for $x \in H^*(X)$,
        \item $x_{[m]} \odot x_{[n]} = \binom{m+n}{n} x_{[m+n]}$ for $x\in H^* (X)$,
        \item $(\lambda x)_{[n]} = \lambda^n x_{[n]}$ and $(x+y)_{[n]} = \sum_{k=0}^n x_{[k]} \odot y_{[n-k]}$ for $x,y \in H^* (X)$ and $\lambda \in \F$.
	\end{itemize}
\end{theorem}
From this presentation, one can extract an additive basis for the cohomology of $ D_n(X_+) $ under the hypotheses of Theorems \ref{thm:cohomology DX mod 2} and \ref{thm:cohomology DX mod p}.
\begin{definition}
	Let $ \mathcal{B} $ be a graded basis of $ H^*(X) $ as a $ \F $-vector space. A decorated gathered block in $ A_X $ is a couple $ (b,x) $, where $ b $ is a gathered block in $ \bigoplus_{n \geq 0} H^*(\Sigma_n) $, in the sense of Giusti--Salvatore--Sinha \cite[Definition 6.3]{Sinha:12} (if $ \F = \Ftwo $) or Guerra \cite[page 964]{Guerra:17} (if $ \F = \Fp $ with $ p > 2 $), and $ x \in \mathcal{B} $ is a basis element. We call $ x $ the decoration of the gathered block.
	
	A decorated Hopf monomial is a formal expression of the form $ b_1 \odot \dots \odot b_r $, where $ b_1,\dots,b_r $ are decorated gathered blocks, such that any two distinct elements among these $ r $ gathered blocks have not the same profile (in the sense of \cite[Definition 6.3]{Sinha:12} (if $ \F = \Ftwo $) or \cite[Definition 3.1]{Guerra:17} and the same decoration concurrently.
\end{definition}

\begin{corollary}[{\cite[Proposition 4.5]{Guerra-Salvatore-Sinha}}] \label{cor: basis DX char+}
    Assume that $ \F = \Fp $ with $ p \geq 2 $ and $H^d (X; \F_p) = 0$ for $d$ odd. Let $ \mathcal{B} $ be as in the previous definition. Realize a decorated gathered block $ (b,x) $, as an element of $ A_X $ by taking the cup product $ \pi^*(b) \cdot x_{[n]} $, where $n$ is the component of $b$. Realize a decorated Hopf monomial $ x = b_1 \odot \dots \odot b_r $ as an element of $ A_X $ by realizing the constituent decorated gathered blocks as elements of $ A_X $ and taking their transfer product.
    Then the set $ \mathcal{M} $ of decorated Hopf monomial is a bigraded basis for $ A_X $ as a $ \F $-vector space.
\end{corollary}
Hopf ring distributivity and the relations of Theorems \ref{thm:cohomology DX mod 2} and \ref{thm:cohomology DX mod p} are enough to explicitly compute the products and the coproduct on decorated Hopf monomials in $ A_X $. We exemplify below the general procedure in a specific case. The reader might also find the graphical algorithm described in \cite[\S 4]{Guerra-Salvatore-Sinha} in terms of skyline diagrams more accessible.

\begin{example}
    Let $\mathcal{B} = \{1, c, c^2, \dots \}$ be the graded basis of $H^* (BU(1); \F_2)$ as a $\F_2$-vector space as before. Consider the following decorated Hopf monomials $$x = (\gamma_{1,2}^2, c) \odot (\gamma_{1,1},c) \quad \text{and} \quad y = (\gamma_{2,1}, 1) \odot (\gamma_{1,1}, c).$$
    We will compute the cup product $x\cdot y$. First, we compute the coproduct of $x$. Using the fact that the cup product and coproduct from a bialgebra, we have 
    \begin{align*}
        &\Delta ((\gamma_{1,2}^2, c)) = \Delta (\gamma_{1,2}^2 \cdot c_{[4]}) = \Delta (\gamma_{1,2})^2 \cdot \Delta (c_{[4]}) \\
        & \hspace{1cm} = (\gamma_{1,2}^2 \otimes 1_0 + \gamma_{1,1}^2 \otimes \gamma_{1,1}^2 + 1_0 \otimes \gamma_{1,2}^2) \cdot (c_{[4]} \otimes 1_0 + c_{[3]} \otimes c_{[1]} + \cdots + 1_0 \otimes c_{[4]}) \\
        & \hspace{1cm} = \gamma_{1,2}^2 \cdot c_{[4]} \otimes 1_0 + \gamma_{1,1}^2 \cdot c_{[2]} \otimes \gamma_{1,1}^2 \cdot c_{[2]} + 1_0 \otimes \gamma_{1,2}^2 \cdot c_{[4]}
    \end{align*}
    and \begin{align*}
        \Delta ((\gamma_{1,1},c)) &= \Delta (\gamma_{1,1}) \cdot \Delta (c_{[2]}) \\ &= (\gamma_{1,1} \otimes 1_0 + 1_0 \otimes \gamma_{1,1}) \cdot (c_{[2]} \otimes 1_0 + c_{[1]} \otimes c_{[1]} + 1_0 \otimes c_{[2]}) \\
        &= \gamma_{1,1} \cdot c_{[2]} \otimes 1_0 + 1_0 \otimes \gamma_{1,1} \cdot c_{[2]}. 
    \end{align*}
    Using the fact that the transfer product and the coproduct form a bialgebra, we have
    \begin{align*}
        \Delta (x) &= \Delta ((\gamma_{1,2}^2, c)) \odot \Delta ((\gamma_{1,1},c)) \\
        &= x \otimes 1_0 + (\gamma_{1,2}^2 \cdot c_{[4]}) \otimes (\gamma_{1,1} \cdot c_{[2]})  +((\gamma_{1,1}^2 \cdot c_{[2]}) \odot (\gamma_{1,1} \cdot c_{[2]})) \otimes (\gamma_{1,1}^2 \cdot c_{[2]}) \\
        &+ (\gamma_{1,1}^2 \cdot c_{[2]}) \otimes ((\gamma_{1,1}^2 \cdot c_{[2]}) \odot (\gamma_{1,1} \cdot c_{[2]})) + (\gamma_{1,1} \cdot c_{[2]}) \otimes (\gamma_{1,2}^2 \cdot c_{[4]}) + 1_0 \otimes x.
    \end{align*}
    Using the Hopf ring distributivity axiom, we have
    \begin{align*}
        x \cdot y &= x\cdot ((\gamma_{2,1}, 1) \odot (\gamma_{1,1}, c)) \\
        &= \sum_{\Delta x = \sum x' \otimes x''} (x' \cdot \gamma_{2,1} ) \odot (x'' \cdot \gamma_{1,1} \cdot c_{[2]}).
    \end{align*}
    We observe that only addends $x' \otimes x''$ in $\Delta (x)$ that have the right components such that the cup products are non-zero are $(\gamma_{1,2}^2 \cdot c_{[4]}) \otimes (\gamma_{1,1} \cdot c_{[2]})$ and $((\gamma_{1,1}^2 \cdot c_{[2]}) \odot (\gamma_{1,1} \cdot c_{[2]})) \otimes (\gamma_{1,1}^2 \cdot c_{[2]})$. Note that by Hopf ring distributivity \[ \gamma_{2,1} \cdot ((\gamma_{1,1}^2 \cdot c_{[2]}) \odot (\gamma_{1,1} \cdot c_{[2]})) = 0 . \] Therefore, \begin{align*}
        x\cdot y &= (\gamma_{2,1} \cdot \gamma_{1,2}^2 \cdot c_{[4]}) \odot (\gamma_{1,1}^2 \cdot c_{[2]}^2)  
    \end{align*} 
    The corresponding skyline diagram is described in Figure~\ref{skylinexy}.
\end{example}
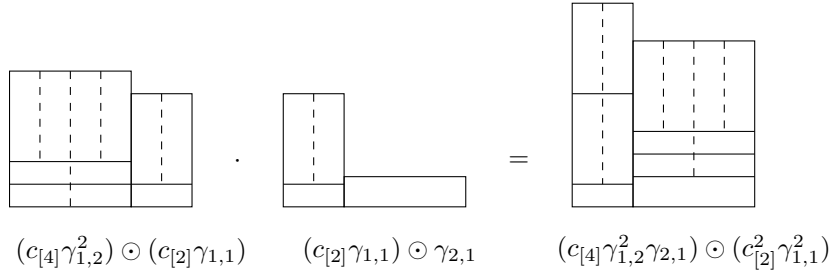
\begin{figure}[t]
    \centering
    \begin{tikzpicture}
        \draw[]{(-6.0,1.8) rectangle (-4.4,0.0)};
        \draw[dashed]{(-5.6,1.8) -- (-5.6,0.6)}; 
        \draw[dashed]{(-5.2,1.8) -- (-5.2,0.0)};
        \draw[dashed]{(-4.8,1.8) -- (-4.8,0.6)};
        \draw[]{(-6.0, 0.6) -- (-4.4, 0.6)};
        \draw[]{(-6.0, 0.3) -- (-4.4 , 0.3)};
        \draw[]{(-4.4, 1.5) rectangle (-3.6, 0.0)};
        \draw[dashed]{(-4.0, 1.5) -- (-4.0, 0.3)};
        \draw[]{(-4.4, 0.3) -- (-3.6, 0.3)};
        \node[color=black] (A) at (-3.0,0.6) {$\cdot$};
        \draw[]{(-2.4, 1.5) rectangle (-1.6, 0.0)};
        \draw[dashed]{(-2.0, 1.5) -- (-2.0, 0.3)};
        \draw[]{(-2.4, 0.3) -- (-1.6, 0.3)};
        \draw[]{(-1.6, 0.4) rectangle (0.0, 0.0)};
        \node[color=black] (B) at (0.7, 0.6) {$=$};
        \draw[]{(1.4, 2.7) rectangle (2.2, 0.0)};
        \draw[dashed]{(1.8, 2.7) -- (1.8, 0.3)};
        \draw[]{(1.4, 0.3) -- (2.2, 0.3)};
        \draw[]{(1.4, 1.5) -- (2.2, 1.5)};
        \draw[]{(2.2,2.2) rectangle (3.8,0.0)};
        \draw[dashed]{(2.6,2.2) -- (2.6,1.0)}; 
        \draw[dashed]{(3.0,2.2) -- (3.0,0.4)};
        \draw[dashed]{(3.4,2.2) -- (3.4,1.0)};
        \draw[]{(2.2, 0.4) -- (3.8, 0.4)};
        \draw[]{(2.2, 1.0) -- (3.8 , 1.0)};
        \draw[]{(2.2, 0.7) -- (3.8, 0.7)};
        \node[color=black] (D) at (-4.4, -0.6) {$(c_{[4]} \gamma_{1,2}^2) \odot (c_{[2]} \gamma_{1,1}) $};
        \node[color=black] (E) at (-1.0, -0.6) {$(c_{[2]} \gamma_{1,1}) \odot \gamma_{2,1}$};
        \node[color=black] (F) at (3.0, -0.6) {$(c_{[4]} \gamma_{1,2}^2 \gamma_{2,1} ) \odot (c_{[2]}^2 \gamma_{1,1}^2 )$};
    \end{tikzpicture}
    \caption{Skyline diagram for the dot product of Hopf monomials $x$ and $y$ using Hopf ring distributivity}
    \label{skylinexy}
\end{figure}

The rational cohomology of $ D(X_+) $ is well-known to experts. We slightly enhance its usual description by incorporating the Hopf ring structure.
\begin{proposition} \label{prop: basis DX char0}
	Let $ \F = \mathbb{Q} $. Let $ \mathcal{B} $ be a graded basis of the cohomology of $ X $.
	If $ H^*(X; \mathbb{Q}) $ is concentrated in even degrees, then $ A_X $ is the commutative bigraded component Hopf ring generated by $ x_{[n]} $ for $ x \in \mathcal{B} $ with the following relations:
	\begin{itemize}
		\item $ x_{[n]}\cdot {x'}_{[n]} = (x \cdot x')_{[n]}$ for $x,x' \in H^*(X)$,
		\item $ \Delta(x_{[n]}) = \sum_{i=0}^n x_{[i]} \otimes x_{[n-i]}$ for $x \in H^*(X)$,
            \item $x_{[m]} \odot x_{[n]} = \binom{m+n}{n} x_{[m+n]}$ for $x\in H^* (X)$,
            \item $(\lambda x)_{[n]} = \lambda^n x_{[n]}$ and $(x+y)_{[n]} = \sum_{k=0}^n x_{[k]} \odot y_{[n-k]}$ for $x,y \in H^* (X)$ and $\lambda \in \F$.
	\end{itemize}
	
	Moreover, the set $ \mathcal{M} $ consisting of elements $ (x_1)_{[n_1]} \odot \dots \odot (x_k)_{[n_k]} $ up to permutation of $ \odot $-factors, where $ x_i \not= x_j $ for $ i\not=j $ and $ x_i \in \mathcal{B} $, form a bigraded basis for $ A_X $.
\end{proposition}
\begin{proof}
	The proof that the relations hold and that $ \mathcal{M} $ is a basis for the commutative bigraded component Hopf ring with that presentation is the same as in \cite{Guerra-Salvatore-Sinha}.
	
	The rational cohomology of $ D_n(X_+) $ is isomorphic to the subspace of invariants of $ {H^*(X)}^{\otimes n} $ under the action of $ \Sigma_n $.
	Under this isomorphism, $ (x_1)_{[n_1]} \odot \dots \odot (x_k)_{[n_k]} $ corresponds to the symmetrization by means of $ (n_1,\dots,n_k) $-shuffles of 
	\[
	\underbrace{x_1 \otimes \dots \otimes x_1}_{n_1 \mbox{ times}} \otimes \underbrace{x_2 \otimes \dots \otimes x_2}_{n_2 \mbox{ times}} \otimes \dots \otimes x_k
	\]
	which constitute a basis for the subspace of invariants.
\end{proof}

We conclude this section by recalling some stable calculations. In the following definitions, we fix a graded basis $ \mathcal{B} $ for the cohomology of $ X $ containing the unit $ 1_X $.

\begin{definition} \label{def:pure}
	Let $ \mathcal{M} $ be as in Corollary \ref{cor: basis DX char+} or Proposition \ref{prop: basis DX char0}.
	An element $ x = b_1 \odot \dots \odot b_r $ is pure if none of the $ b_i $'s is equal to the unit class of a component $ 1_n \in H^*(D_n(X_+)) $.
\end{definition}

The stabilizer of $ n+1 $ in $ \Sigma_{n+1} $ is isomorphic to $ \Sigma_n $. This provides an inclusion $ i_n \colon \Sigma_n \hookrightarrow \Sigma_{n+1} $. If we fix a basepoint $ * \in X $, then the spaces $ D(X_+) $ exhibit homological stability with respect to the stabilization maps
\begin{align*}
	&j_n \colon (p,(x_1,\dots,x_n)) \in E \Sigma_n \times_{\Sigma_n} X^n = D_n(X_+) \\
	&\longmapsto (i_{n,1}(p),(x_1,\dots,x_n,*))\in E \Sigma_{n+1} \times_{\Sigma_{n+1}} X^{n+1} = D_{n+1}(X_+).
\end{align*}
$ j_n $ depends on the basepoint, but if $ X $ is path-connected, then any two choices yield homotopic maps. 
This homologically stability property is classically well-known and follows, for instance, from the calculations of \cite[Theorem 4.1]{May-Cohen}.
Thus, under this connectedness hypothesis, the stable cohomology of $ D_n(X_+) $ is given by $ A_\infty(X) = \varprojlim_n H^*(D_n(X_+)) $ is well-defined and coincides with $ H^*(D_n(X_+)) $ in low degrees.

The stabilization maps in cohomology have a section on the subspace generated by pure Hopf monomials given by the transfer product with the units:
\[
1_1 \odot \_ \colon H^*(D_n(X_+)) \longrightarrow H^*(D_{n+1}(X_+))
\]
Therefore, we can define the stabilization of a pure $ x \in \mathcal{M} $ as the unique class $ x \odot 1_{\infty} $ restricting to $ x \odot 1_{[n]} \in H^*(D_{n+n(x)}(X_+)) $ for all $ n \in \mathbb{N} $.

\begin{corollary}[{\cite[Lemma 6.5]{Guerra-Salvatore-Sinha}}] \label{cor:basis DX stable}
	If $ \characteristic(\F) \not= 2 $, assume that $ H^*(X) = 0 $ in odd degrees. Then, the set $ \{ x \odot 1_{\infty}: x \in \mathcal{M} \mbox{ pure} \} $ is a graded basis for $ A_\infty(X) $ as a $ \F $-vector space.
\end{corollary}

%When the cohomology algebra of $ X $ is polynomial on a class $ \alpha $, we can use the graded monomial basis $ \mathcal{B} = \{1_X,\alpha,\alpha^2,\dots,\alpha^n,\dots\} $ for $ H^*(X) $. Then, from Theorem \ref{main DX} we deduce a simple description of the ring structure of $ A_\infty(X) $. In particular, this additional hypotheses holds true if $ X = \mathbb{P}^\infty(\mathbb{R}) = B(O(1)) $ or $ X = \mathbb{P}^\infty( \mathbb{C}) = B(U(1)) $. The following corollary, which is a reformulation of Dung's results \cite{Dung}, is true under these assumptions.
\begin{corollary}[{\cite[Theorem 6.8]{Guerra-Salvatore-Sinha}}, particular case] \label{cor:polynomial generators}
	Assume that $ H^*(X) $ is a polynomial algebra, with even-dimensional generators if $ \characteristic(\F) \not= 2 $. Let $ \mathcal{B} $ be the monomial basis of $ H^*(X) $.
	If $ \characteristic(\F) = p \not= 0 $, then $ A_\infty(X) $ is the free graded commutative algebra generated by the classes $ b \odot 1_{\infty} $, where $ b $ is a decorated gathered block satisfying the following conditions:
	\begin{itemize}
		\item either the underlying (non-decorated) block of $ b $ or the decoration of $ b $ is not a $p$-th power,
		\item the width of $ b $ is a power of $ p $,
		\item $ b $ is different from the unit class.
	\end{itemize}
	If, instead, $ \characteristic(\F) = 0 $, then $ A_\infty(X) $ is a polynomial algebra generated by classes $ x_{[n]} \odot 1_\infty $ for all $ n \geq 1 $ and $ x $ ranging in a minimal set of algebra generators for $ H^*(X) $.
\end{corollary}

%%% Local Variables:
%%% TeX-master: "Unordered flag varieties.tex"
%%% End:

%% file: Section3.tex
% !TEX root = Unordered flag varieties.tex

\section{Review of the cohomology of alternating subgroups of the hyperoctahedral groups}

The normalizer $ N_{O(n)}(T(n)) $ of the ``torus'' of diagonal matrices $ T(n) = \{D \in O(n): D \mbox{ is diagonal}\} $ is isomorphic to the wreath product $ \Sigma_n \wr (\mathbb{Z}/2\mathbb{Z}) $. This can be realized as the isometry group of a hyperoctahedron in $ \mathbb{R}^n $. This makes $ N_{O(n)}(T(n)) $ a reflection group that, under the classification of finite Coxeter groups (see \cite{Humphreys}), corresponds to the Dynkin diagram $ \Bgroup_n $ having the following form:
\begin{center}
\begin{tikzpicture}[line cap=round,line join=round,>=triangle 45,x=1cm,y=1cm]
	\clip(-3.197777777777781,-1) rectangle (7.575555555555557,1);
	\draw [line width=1pt,color=black] (0,0)-- (1,0);
	\draw [line width=1pt,color=black] (1,0)-- (2,0);
	\draw [line width=1pt,color=black] (2,0)-- (2.62,0);
	\draw [line width=1pt,color=black] (3.24,0)-- (4,0);
	\begin{scriptsize}
		\draw [fill=black] (0,0) circle (2.5pt);
		\draw[color=black] (0.07333333333333134,0.19) node {$s_0$};
		\draw [fill=black] (1,0) circle (2.5pt);
		\draw[color=black] (1.0688888888888872,0.19) node {$s_1$};
		\draw [fill=black] (2,0) circle (2.5pt);
		\draw[color=black] (2.0733333333333324,0.19) node {$s_1$};
		\draw [fill=black] (4,0) circle (2.5pt);
		\draw[color=black] (4.073333333333333,0.19) node {$s_{n-1}$};
		\draw[color=black] (0.5088888888888871,0.18111111111111095) node {$4$};
		\draw[color=black] (2.917777777777777,0.06555555555555548) node {$\dots$};
	\end{scriptsize}
\end{tikzpicture}
\end{center}

The intersection  $ N_{O(n)}(T(n)) \cap SO(n) $ is identified with the alternating subgroup $ \Bgrouppos{n} $ of the Coxeter group $ \Bgroup_n $, defined as the kernel of the sign homomorphism $ \sgn_{\Bgroup_n} \colon \Bgroup_n \to C_2 = \{-1,1\} $ whose value on every reflection is $ -1 $.

\subsection{Mod  \texorpdfstring{$2$}{2} cohomology of \texorpdfstring{$\Bgrouppos{n}$}{B+(n)}}

The mod $2$ cohomology of $\Bgroup_n$ has first been computed by Guerra in \cite{Guerra:21}, where the author shows that $ \bigoplus_n H^* (B \Bgroup_n; \mathbb{F}_2 ) $ is a Hopf ring. Here we observe that this Hopf ring is isomorphic to $ A_{\mathbb{P}^{\infty} (\mathbb{R})}$ and that, consequently, his description is a particular case of Theorem \ref{thm:cohomology DX mod 2}.

A similar structure exists on the mod $ 2 $ cohomology of $ \Bgrouppos{n} $. More precisely, on the direct sum $ \AB = \bigoplus_{n,d \geq 0} H^d (B\Bgrouppos{n}; \mathbb{F}_2) $ we consider structural morphisms analogous to those defined in \S \ref{sec:cohomology extended powers}:
\begin{itemize}
	\item the usual cup product, component by component, $ \cdot \colon \AB \otimes \AB \to \AB $.
	\item a coproduct $ \Delta \colon \AB \to \AB \otimes \AB $ whose components are restriction maps associated with the inclusions of groups $ \Bgrouppos{n} \times \Bgrouppos{m} \to \Bgrouppos{n+m} $.
	\item a transfer product $ \odot \colon \AB \otimes \AB \to \AB $ whose components are transfer maps associated with the same inclusions.
\end{itemize}

\begin{definition}[from \cite{Sinha:17}] \label{def:Hopf semiring}
	A commutative bigraded component almost-Hopf semiring is a bigraded vector space $ A = \bigoplus_{n,d}A^{n,d} $ with a coproduct $ \Delta $, two products $ \odot,\cdot $, a unit $ \eta \colon \F \to A $, and a counit $ \varepsilon \colon A \to \F $ satisfying all the properties of Definition \ref{def:Hopf ring}, except those involving the antipode and $ \Delta $ and $ \odot $ forming a bialgebra.
\end{definition}

Theorem 2.4 of \cite{Sinha:17} implies that $ \AB $, with the morphisms above, is an almost-Hopf ring over $ \mathbb{F}_2 $. There is an involution $ \iota \colon \AB \to \AB $ induced on the $n$-th component by conjugation by a reflection in $ \Bgroup_n $.
Using $ \iota $, $ \AB $ can be extended to a bigger almost-Hopf semiring $ \widetilde{\AB} $ that coincide with $ \AB $ in positive component, and such that $ \widetilde{\AB}^{0,*} = \mathbb{F}_2 \{ 1^+, 1^- \} $ concentrated in dimension $ 0 $. The unit $ 1_0 \in {\AB}^{0,0} $ is identified with $ 1^+ + 1^- $. $ 1^+ $ is the unit for the transfer product and $ 1^- \odot x = \iota(x) $ for all $ x \in \AB $. The cup product is extended on the $0$-th component by letting $ 1^+ \cdot 1^+ = 1^+ $, $ 1^- \cdot 1^- = 1^- $ and $ 1^+ \cdot 1^- = 1^- \cdot 1^+ = 0 $.
The coproduct in $ \widetilde{\AB} $ is defined by
\[
\Delta(x) = 1^+ \otimes x + 1^- \otimes x + x \otimes 1^+ + x \otimes 1^- + \overline{\Delta}(x),
\]
where $ \overline{\Delta} $ is the reduced coproduct in $ \AB $.
Proposition~3.11 of \cite{Guerra-Santanil} guarantees that the almost-Hopf semiring structure on $ \widetilde{\AB} $ is well-defined.
Clearly knowledge of $ \widetilde{\AB} $ implies knowledge of $ \AB $.

Moreover, in that same paper, a full presentation of $ \widetilde{\AB} $ is determined:
\begin{theorem}[{\cite[Theorem 6.8]{Guerra-Santanil}}] \label{thm:presentation B+}
Let $ \mathcal{Q}_1 $ be the quotient Hopf ring of $ A_{\mathbb{P}^\infty(\mathbb{R})} $ obtained by putting $ \gamma_{k,l} = 0 $ if $ k \geq 2 $. Consider the bigraded component bialgebra
\[
R = \bigoplus_{n \geq 0} \frac{\mathcal{Q}_1^{n,*}}{(w,\gamma_{1,1} + w \odot 1_1, \dots, \gamma_{1,1} \odot 1_{n-2} + w \odot 1_{n-1},\dots)}.
\]
Let $ R^0 $ be the Hopf ring obtained by adjoining a unit to $ R $ and letting $ \odot $ of elements of $ R $ be $ 0 $.

$ R^0 $ is identified as a sub-Hopf semiring of $ \widetilde{\AB} $. Moreover, as an almost-Hopf semiring over $ R^0 $, $ \widetilde{\AB} $ is generated by the two classes $ 1^{\pm} \in \widetilde{\AB}^{0,0} $ and a family of classes $ \{ \gamma_{k,l}^+ \in \widetilde{\AB}^{l2^k,l(2^k-1)} \}_{k \geq 2, l \geq 1} $.

A complete set of relations is the following, where we denote $ \gamma_{k,l}^+ \odot 1^- $ as $ \gamma_{k,l}^- $ and where we add the apex $ 0 $ to emphasize when elements belong to $ A $:
\begin{enumerate}
\item $ 1^+ + 1^-  = 1^0 $, the unit of the $0$-th component of $ R $,
\item $ 1^+ $ is the Hopf ring unit of $ \widetilde{\AB} $,
\item $ 1^- \odot 1^- = 1^+ $,
\item $ \gamma_{k,l}^+ \odot \gamma_{k,m}^+ = \left( \begin{array}{c} l+m \\ l \end{array} \right) \gamma_{k,l+m} $ for all $ k \geq 2 $, $ l,m \geq 1 $,
\item $ \gamma_{k,l}^+ \cdot \gamma_{k',l'}^- = 0 $ unless $ k = k' = 2 $,
\item $ \gamma_{2,m}^+ \cdot \gamma_{2,m}^- = \left\{ \begin{array}{ll}
(\gamma_{2,m}^+)^2 + (\gamma_{2,m}^-)^2 + (\gamma_{2,m-1}^+)^2 \odot (\gamma_{1,2}^3)^0 & \mbox{if } m \mbox{ is odd} \\
(\gamma_{2,m-1}^+)^2 \odot (\gamma_{1,2}^3)^0 & \mbox{if } m \mbox{ is even}
\end{array} \right. $ for all $ m \geq 1 $,
\item for all Hopf monomials $ x = b_1 \odot \dots \odot b_r \in \mathcal{Q}_{1} $, $$ \gamma_{k,l}^+ \cdot x^0 = \bigodot_{i=1}^r \left( \gamma_{k,\frac{n(b_i)}{2^k}}^+ \cdot (b_i)^0 \right) , $$ where $ n(b_i) $ is the component of $ b_i $, and $ \gamma_{k,l} $ is assumed to be $ 0 $ if $ l $ is not an integer,
\item $ \Delta(x \odot y) = \Delta(x) \odot \rho_+ \Delta(y) $ for all $ x,y \in \widetilde{\AB} $,
\item $ \Delta(\gamma_{k,l}^+) = \sum_{i=0}^l \left( \gamma_{k,i}^+ \otimes \gamma_{k,l-i}^+ + \gamma_{k,i}^- \otimes \gamma_{k,m-i}^- \right) $ for all $ k \geq 2 $, $ l \geq 0 $, with the convention that $ \gamma_{k,0}^\pm = 1^\pm $.
\end{enumerate}
\end{theorem}

There is also an explicit additive basis for $ \widetilde{\AB} $.
\begin{definition}[from \cite{Guerra-Santanil}] \label{def:basis B+}
Let $ \mathcal{B} = \{1, w, w^2, \dots \} $ be the monomial basis of $ H^*(\mathbb{P}^\infty(\mathbb{R}); \mathbb{F}_2) = \mathbb{F}_2[w] $. Let $ \mathcal{M} $ be the associated decorated Hopf monomial basis of $ A_{\mathbb{P}^\infty(\mathbb{R})} $.
We define $ \mathcal{G}_{ann} $ as the set of those $ x = b_1 \odot \dots \odot b_r \in \mathcal{M} $ whose constituent gathered blocks $ b_i $ all contain at least one factor $ \gamma_{k,l} $ with $ k \geq 2 $.

We also define $ \mathcal{G}_{quot} $ as the set of those $ x = b_1 \odot \dots \odot b_r \in \mathcal{M} $ satisfying one of the following two conditions:
\begin{itemize}
	\item at least one constituent gathered block $ b_i $ of $ x $ is of the form $ (1_n,w^k) $, and the one with $ k $ maximal (necessarily unique) satisfies $ n \geq 2 $
	\item no $ b_i $ is of the form $ (1_n,w^k) $ and, among the constituent gathered blocks of the form $ (\gamma_{1,n}^l, w^k) $ (if there are any), the one with the couple $ (k,l) $ maximal with respect to the lexicographic order $ \leq_{lex} $ on $ \mathbb{N} \times \mathbb{N} $ satisfies $ n \geq 2 $ or $ l = 1 $
\end{itemize}

For a gathered block $ b $ containing at least one factor $ \gamma_{k,l} $ with $ k \geq 2 $, $ b = (w^{a_0})_{[m2^n]} \cdot \prod_{i=1}^n \gamma_{i,2^{n-i}m}^{a_i} $ for some $ a_i \geq 0 $, with $ n \geq 2 $ and $ a_n \not= 0 $. We write
\[
b^+ = ((w_{[m2^n]})^{a_0} \gamma_{1,m2^{n-1}}^{a_1})^0 \sum_{\substack{0 \leq k_0 \leq a_0 \\
\dots \\
0 \leq k_{n-1} \leq a_{n-1} \\
0 \leq k_n \leq \lfloor \frac{a_n}{2} \rfloor}}^{a_{n-1}} \prod_{i=2}^n \left( \begin{array}{c}
a_i \\
k_i
\end{array} \right) (\gamma_{i,m2^{n-i}}^-)^{k_i} (\gamma_{i,m2^{n-i}}^+)^{a_i-k_i}.
\]
For an element $ x = b_1 \odot \dots \odot b_r \in \mathcal{G}_{ann} $, we let $ x^+ = \bigodot_{i=1}^r b_i^+ $. We also let $ x^- = \iota(x^+) $.

For an element $ x \in \mathcal{G}_{quot} \setminus \mathcal{G}_{ann} $, we let $ x^0 = \res(x) $, where $ \res \colon A_{\mathbb{P}^\infty(\mathbb{R})} \to \AB $ is the direct sum of the cohomological restriction maps $ \bigoplus_n \res^{\Bgroup_n}_{\Bgrouppos{n}} $.
\end{definition}

\begin{theorem}[{\cite[Theorem 3.7]{Guerra-Santanil}}] \label{thm:basis B+}
With reference to Definition \ref{def:basis B+}, let
\begin{itemize}
\item $ \mathcal{M}_0 = \{x^0 \}_{x \in \mathcal{G}_{quot}\setminus \mathcal{G}_{ann}} $,
\item $ \mathcal{M}_+ = \{x^+ \}_{x \in \mathcal{G}_{ann}} $,
\item $ \mathcal{M}_- = \{x^- \}_{x \in \mathcal{G}_{ann}} $.
\end{itemize}
The set $ \mathcal{M}_{charged} = \mathcal{M}_0 \sqcup \mathcal{M}_+ \sqcup \mathcal{M}_- $ is a bigraded basis of $ \widetilde{\AB} $ over $ \mathbb{F}_2 $. Moreover, $ \res(x) = x^+ + x^- $ for all $ x \in \mathcal{G}_{ann} $.
\end{theorem}

Although the description above is more complicated than its analogs Theorem~\ref{thm:cohomology DX mod 2} and Corollary~\ref{cor: basis DX char+}, we can still deduce the ring structure on the components $ H^*(B\Bgrouppos{n}; \mathbb{F}_2) $.
Since the component of all elements of $ \mathcal{G}_{ann} $ is a multiple of $ 4 $, then all classes belonging to the other components are restrictions of classes in $ A_{\mathbb{P}^\infty(\mathbb{R})} $. More precisely, we have the following result.

\begin{corollary}[of Theorem~\ref{thm:basis B+}]
If $ n \not\equiv 0 \mod 4 $, then the restriction map is surjective and induces an isomorphism
\[
\frac{H^*(B\Bgroup_n; \mathbb{F}_2)}{(\gamma_{1,1} \odot 1_{n-2} + w \odot 1_{n-1})} \cong H^*(B\Bgrouppos{n}; \mathbb{F}_2).
\]
\end{corollary}

In components that are multiples of $ 4 $, the calculation of the cup product is exemplified below.

\begin{example}
    We determine the cup product of the two charged Hopf monomials $ x = (\gamma_{2,2}^3 \odot \gamma_{3,1}^5 w_{[8]})^+ $ and $ y = (\gamma_{2,1}\gamma_{1,2} \odot \gamma_{2,1} \odot \gamma_{3,1})^+ $ in $ \AB $.
    The first step is to compute the coproduct of the charged gathered blocks appearing in $ x $. We use the fact that $ \cdot $ and $ \Delta $ form a bialgebra, the fact that cup products of classes in different components are zero, and the relations (5), (6), (7) and (8) of Theorem \ref{thm:presentation B+} to perform the computation.
    Explicitly
    \begin{align*}
        &\Delta((\gamma_{2,2}^3)^+) = \Delta((\gamma_{2,2}^+)^3 + (\gamma_{2,2}^+)^2 \cdot \gamma_{2,2}^-) = \Delta((\gamma_{2,2}^+))^3 + \Delta(\gamma_{2,2}^+) \cdot \Delta(\gamma_{2,2}^-)^2 \\
        %&= (1^+ \otimes \gamma_{2,2}^+ + \gamma_{2,1}^+ \otimes \gamma_{2,1}^+ + \gamma_{2,2}^+ \otimes 1^+ + 1^- \otimes \gamma_{2,2}^- + \gamma_{2,1}^- \otimes \gamma_{2,1}^- + \gamma_{2,2}^- \otimes 1^-)^3 \\
        %&+ (1^+ \otimes \gamma_{2,2}^+ + \gamma_{2,1}^+ \otimes \gamma_{2,1}^+ + \gamma_{2,2}^+ \otimes 1^+ + 1^- \otimes \gamma_{2,2}^- + \gamma_{2,1}^- \otimes \gamma_{2,1}^- + \gamma_{2,2}^- \otimes 1^-) \\
        %&\cdot ((1^+ \otimes \gamma_{2,2}^- + \gamma_{2,1}^+ \otimes \gamma_{2,1}^- + \gamma_{2,2}^+ \otimes 1^- + 1^- \otimes \gamma_{2,2}^+ + \gamma_{2,1}^- \otimes \gamma_{2,1}^+ + \gamma_{2,2}^- \otimes 1^+)^2 \\
        &= \Big(1^+ \otimes (\gamma_{2,2}^+)^3 + 1^- \otimes (\gamma_{2,2}^+)^3 + (\gamma_{2,1}^+)^3 \otimes (\gamma_{2,1}^+)^3 + (\gamma_{2,1}^+)^2 \gamma_{2,1}^- \otimes (\gamma_{2,1}^+)^2 \gamma_{2,1}^- \\
        &+ \gamma_{2,1}^+ (\gamma_{2,1}^-)^2 \otimes \gamma_{2,1}^+ (\gamma_{2,1}^-)^2 + 1^+ (\gamma_{2,2}^+)^3 \otimes 1^+ + (\gamma_{2,2}^-)^3 \otimes 1^- \Big) + \Big( 1^+ \otimes (\gamma_{2,2}^+)^2 \gamma_{2,2}^- \\
        &+ 1^- \otimes \gamma_{2,2}^+ (\gamma_{2,2}^-)^2 + (\gamma_{2,1}^+)^3 \otimes (\gamma_{2,1}^+)^2 \gamma_{2,1}^- + (\gamma_{2,1}^+)^2 \gamma_{2,1}^- \otimes (\gamma_{2,1}^+)^3 \\
        &+ \gamma_{2,1}^+(\gamma_{2,1}^-)^2 \otimes (\gamma_{2,1}^-)^3 + (\gamma_{2,1}^-)^3 \otimes \gamma_{2,1}^+(\gamma_{2,1}^-)^2 + (\gamma_{2,2}^+)^2 \gamma_{2,2}^- \otimes 1^+ + \gamma_{2,2}^+ (\gamma_{2,2}^-)^2 \otimes 1^- \Big) \\
        &= 1^+ \otimes (\gamma_{2,2}^3)^+ + 1^- \otimes (\gamma_{2,2}^3)^+ + (\gamma_{2,2}^3)^+ \otimes 1^+ + (\gamma_{2,2}^3)^- \otimes 1^- + (\gamma_{2,1}^3)^+ \otimes (\gamma_{2,1}^3)^+ \\
        &+ (\gamma_{2,1}^3)^- \otimes (\gamma_{2,1}^3)^- + (\gamma_{2,1}^3)^+ \otimes \gamma_{2,1}^+(\gamma_{2,1}^-)^2 + \gamma_{2,1}^+(\gamma_{2,1}^-)^2 \otimes (\gamma_{2,1}^3)^+ \\
        &+ (\gamma_{2,1}^3)^- \otimes (\gamma_{2,1}^+)^2 \gamma_{2,1}^- + (\gamma_{2,1}^+)^2 \gamma_{2,1}^- \otimes (\gamma_{2,1}^-)^3.
    \end{align*}
    By relation (6) of Theorem \ref{thm:presentation B+},
    \[
    (\gamma_{2,1}^+)^2\gamma_{2,1}^- = \left( (\gamma_{2,1}^-)^2 + (\gamma_{1,2}^3)^0 + \gamma_{1,2}^+ \gamma_{1,2}^- \right) \gamma_{2,1}^- = (\gamma_{2,1}^3)^- + (\gamma_{2,1}\gamma_{1,2}^3)^-
    \]
    and similarly for $ \gamma_{2,1}^+ (\gamma_{2,1}^-)^2 $. Substituting in the expression for $ \Delta((\gamma_{2,1}^3)^+) $ we obtain that
    \begin{align*}
    &\Delta((\gamma_{2,1}^3)^+) = 1^+ \otimes (\gamma_{2,2}^3)^+ + 1^- \otimes (\gamma_{2,2}^3)^+ + (\gamma_{2,2}^3)^+ \otimes 1^+ + (\gamma_{2,2}^3)^- \otimes 1^- \\
    &+ (\gamma_{2,1}^3)^+ \otimes (\gamma_{2,1}^3)^+ + (\gamma_{2,1}^3)^- \otimes \gamma_{2,1}^3)^- + (\gamma_{2,1}^3)^+ \otimes (\gamma_{2,1}\gamma_{1,2}^3)^- \\
    &+ (\gamma_{2,1}^3)^- \otimes (\gamma_{2,1}\gamma_{1,2}^3)^+ + (\gamma_{2,1}\gamma_{1,2}^3)^+ \otimes (\gamma_{1,2}^3)^- + (\gamma_{2,1}\gamma_{1,2}^3)^- \otimes (\gamma_{2,1}^3)^+.
    \end{align*}
The reader is encouraged to read Remark 5.6 of \cite{Sinha:17} for the explanation of some computational difficulties.
A similar calculation shows that
\begin{align*}
\Delta((\gamma_{3,1}^5 w_{[8]})^+) &= 1^+ \otimes (\gamma_{3,1}^5 w_{[8]})^+ + 1^- \otimes (\gamma_{3,1}^5 w_{[8]})^- \\
&+ (\gamma_{3,1}^5 w_{[8]})^+ \otimes 1^+ + (\gamma_{3,1}^5 w_{[8]})^- \otimes 1^-.
\end{align*}

As a second step, we exploit relations (3), (4) and (8) of Theorem \ref{thm:presentation B+} to determine the coproduct of $ x $:
\begin{align*}
&\Delta(x) = (\gamma_{3,1}^5 w_{[8]})^+ \otimes (\gamma_{2,2}^3)^+ + (\gamma_{3,1}^5 w_{[8]})^- \otimes (\gamma_{2,2}^3)^+ + (\gamma_{2,2}^3 \odot \gamma_{3,1}^5 w_{[8]})^+ \otimes 1^+ \\
&+ (\gamma_{2,2}^3 \odot \gamma_{3,1}^5 w_{[8]})^- \otimes 1^- + (\gamma_{2,1}^3 \odot \gamma_{3,1}^5 w_{[8]})^+ \otimes (\gamma_{2,1}^3 \odot \gamma_{3,1}^5 w_{[8]})^+ \\
&+ (\gamma_{2,1}^3 \odot \gamma_{3,1}^5 w_{[8]})^- \otimes (\gamma_{2,1}^3)^- + (\gamma_{2,1}^3 \odot \gamma_{3,1}^5 w_{[8]})^+ \otimes (\gamma_{2,1}\gamma_{1,2}^3)^- \\
&+ (\gamma_{2,1}^3 \odot \gamma_{3,1}^5 w_{[8]})^- \otimes (\gamma_{2,1}\gamma_{1,2}^3)^+ + (\gamma_{2,1}\gamma_{1,2}^3 \odot \gamma_{3,1}^5 w_{[8]})^+ \otimes (\gamma_{1,2}^3)^- \\
&+ (\gamma_{2,1}\gamma_{1,2}^3 \odot \gamma_{3,1}^5 w_{[8]})^- \otimes (\gamma_{2,1}^3)^+ + 1^+ \otimes (\gamma_{2,2}^3 \odot \gamma_{3,1}^5 w_{[8]})^+ + 1^- \otimes (\gamma_{2,2}^3 \odot \gamma_{3,1}^5 w_{[8]})^+ \\
&+ (\gamma_{2,2}^3)^+ \otimes (\gamma_{3,1}^5 w_{[8]})^+ + (\gamma_{2,2}^3)^- \otimes (\gamma_{3,1}^5 w_{[8]})^- + (\gamma_{2,1}^3)^+ \otimes (\gamma_{2,1}^3 \odot \gamma_{3,1}^5 w_{[8]})^+ \\
&+ (\gamma_{2,1}^3)^- \otimes (\gamma_{2,1}^3 \odot \gamma_{3,1}^5 w_{[8]})^- + (\gamma_{2,1}^3)^+ \otimes (\gamma_{2,1}\gamma_{1,2}^3 \odot \gamma_{3,1}^5 w_{[8]})^- \\
&+ (\gamma_{2,1}^3)^- \otimes (\gamma_{2,1}\gamma_{1,2}^3 \odot \gamma_{3,1}^5 w_{[8]})^+ + (\gamma_{2,1}\gamma_{1,2}^3)^+ \otimes (\gamma_{1,2}^3 \odot \gamma_{3,1}^5 w_{[8]})^- \\
&+ (\gamma_{2,1}\gamma_{1,2}^3)^- \otimes (\gamma_{2,1}^3 \odot \gamma_{3,1}^5 w_{[8]})^+
\end{align*}

Thirdly, as $ y = (\gamma_{2,1}\gamma_{1,2})^+ \odot (\gamma_{2,1} \odot \gamma_{3,1})^+ $, we can compute $ x \cdot y $ from the previous coproduct calculations and Hopf ring distributivity
\[
x \cdot y = x \cdot \left((\gamma_{2,1}\gamma_{1,2})^+ \odot (\gamma_{2,1} \odot \gamma_{3,1})^+ \right) = \sum \left( x_{(1)} \cdot (\gamma_{2,1}\gamma_{1,2})^+ \right) \odot \left( x_{(2)} \cdot (\gamma_{2,1} \odot \gamma_{3,1})^+ \right).
\]
One observes that the component of $ (\gamma_{2,1}w_{2})^+ $ is $ 4 $ and the component of $ (\gamma_{2,1} \odot \gamma_{3,1})^+ $ is $ 12 $. The addends in $ \Delta(x) $ in the correct components are $ (\gamma_{2,1}^3)^+ \otimes (\gamma_{2,1}^3 \odot \gamma_{3,1}^5 w_{[8]})^+ $, $ (\gamma_{2,1}^3)^- \otimes (\gamma_{2,1}^3 \odot \gamma_{3,1}^5 w_{[8]})^- $, $ (\gamma_{2,1}^3)^+ \otimes (\gamma_{2,1}\gamma_{1,2}^3 \odot \gamma_{3,1}^5 w_{[8]})^- $, $ (\gamma_{2,1}^3)^- \otimes (\gamma_{2,1}\gamma_{1,2}^3 \odot \gamma_{3,1}^5 w_{[8]})^+ $, $ (\gamma_{2,1}\gamma_{1,2}^3)^+ \otimes (\gamma_{1,2}^3 \odot \gamma_{3,1}^5 w_{[8]})^- $ and $ (\gamma_{2,1}\gamma_{1,2}^3)^- \otimes (\gamma_{2,1}^3 \odot \gamma_{3,1}^5 w_{[8]})^+ $.
For all these terms, $ x_{(2)} $ is the transfer product of two primitive elements. Therefore, the cup product $ x_{(2)} \cdot (\gamma_{2,1} \odot \gamma_{3,1})^+ $ is easily determined by applying Hopf ring distributivity again.
\begin{align*}
x \cdot y &= (\gamma_{2,1}^3)^+ (\gamma_{2,1}\gamma_{1,2})^+ \odot (\gamma_{2,1}^3)^+ \gamma_{2,1}^+ \odot (\gamma_{3,1}^2 w_{[8]})^+ + (\gamma_{2,1}^3)^- (\gamma_{2,1}\gamma_{1,2})^+ \\
&\odot (\gamma_{2,1}^3)^- \gamma_{2,1}^+ \odot (\gamma_{3,1}^2 w_{[8]})^+ + (\gamma_{2,1}^3)^+ (\gamma_{2,1}\gamma_{1,2})^+ \odot (\gamma_{2,1}\gamma_{1,2}^3)^- \gamma_{2,1}^+ \odot (\gamma_{3,1}^2 w_{[8]})^+ \\
&+ (\gamma_{2,1}^3)^- (\gamma_{2,1}\gamma_{1,2})^+ \odot (\gamma_{2,1}\gamma_{1,2}^3)^+ \gamma_{2,1}^+ \odot (\gamma_{3,1}^2 w_{[8]})^+ + (\gamma_{2,1}\gamma_{1,2}^3)^- (\gamma_{2,1}\gamma_{1,2})^+ \\
&\odot (\gamma_{2,1}^3)^+ \gamma_{2,1}^+ \odot (\gamma_{3,1}^2 w_{[8]})^+ + (\gamma_{2,1}\gamma_{1,2}^3)^+ (\gamma_{2,1}\gamma_{1,2})^+ \odot (\gamma_{2,1}^3)^- \gamma_{2,1}^+ \odot (\gamma_{3,1}^2 w_{[8]})^+.
\end{align*}

Finally, in order to write all the addends of $ x \cdot y $ as linear combinations of charged decorated Hopf monomials, we use our cup product relations again. Explicitly, we observe that
\begin{gather*}
\gamma_{2,1}^+ \cdot \gamma_{2,1}^+ = (\gamma_{2,1}^2)^+, 
(\gamma_{2,1}^3)^+ \cdot \gamma_{2,1}^+ = (\gamma_{2,1}^4)^+ + (\gamma_{2,1}^4)^- + (\gamma_{2,1}^2\gamma_{1,3}^3)^+ + (\gamma_{1,2}^6)^0 \\
\tag*{and} (\gamma_{2,1}^3)^- \cdot \gamma_{2,1}^+ = (\gamma_{2,1}^4)^- + (\gamma_{2,1}^2 \gamma_{1,2}^3)^-.
\end{gather*}
Substituting in the espression above for $ x \cdot y $, many terms cancel out and we obtain that $ x \cdot y = (\gamma_{2,1}^4 \gamma_{1,2} \odot \gamma_{2,1}^4 \odot \gamma_{3,1}^2 w_{[8]})^+ $.
\end{example}

\subsection{Cohomology of \texorpdfstring{$\Bgrouppos{n}$}{B+(n)} at odd primes}

The cohomology of $ \Bgrouppos{n} $ modulo odd primes is much simpler than its mod $ 2 $ counterpart.

\begin{theorem} \label{thm:cohomology alternating group mod p}
Let $ n \in \mathbb{N} $ and let $ p > 2 $ be a prime number. Then the projection $ \Bgrouppos{n} \to \Sigma_n $ induces an isomorphism in mod $ p $ cohomology. In particular, $ \bigoplus_n H^*(\Bgrouppos{n}; \mathbb{F}_p) \cong A_{\{*\}} $ as Hopf rings.
\end{theorem}
\begin{proof}
By Shapiro's lemma $ H^*(\Bgrouppos{n}; \mathbb{F}_p) \cong H^*(\Bgroup_n; \Ind_{\Bgrouppos{n}}^{\Bgroup_n}(\mathbb{F}_p)) $, where $ \Ind_{\Bgrouppos{n}}^{\Bgroup_n}(\mathbb{F}_p) $ is the induced representation of the constant representation $ \mathbb{F}_p $. $ \Ind_{\Bgrouppos{n}}^{\Bgroup_n}(\mathbb{F}_p) \cong \mathbb{F}_p \oplus \sgn $, where $ \sgn $ is the mod $ p $ sign representation of $ \Bgroup_n $, defined by $ x . 1 = (-1)^{l(x)} $, where $ l(x) $ is the Coxeter length function.

By the Leray spectral sequence for the fibration $ B(\mathbb{Z}/2\mathbb{Z})^n \to B \Bgroup_n  \to B \Sigma_n $, in the formulation of Lyndon--Hochschild--Serre, there are spectral sequences
\begin{align*}
E_2^{*,*} &= H^*(\Sigma_n;{H^*(\mathbb{Z}/2\mathbb{Z}; \mathbb{F}_p)}^{\otimes n}) \Rightarrow H^*(\Bgroup_n; \mathbb{F}_p) \\
\tag*{and} E_2^{*,*} &= H^*(\Sigma_n; {H^*(\mathbb{Z}/2\mathbb{Z};\sgn)}^{\otimes n} ) \Rightarrow H^*(\Bgroup_n; \sgn).
\end{align*}
In the expression above, $ \Sigma_n $ acts on $ {H^*(\mathbb{Z}/2\mathbb{Z};\mathbb{F}_p)}^{\otimes n} $ and $ {H^*(\mathbb{Z}/2\mathbb{Z};\sgn})^{\otimes n} $ by permuting the factors.

By a standard group-cohomological application of transfer maps, since the order of $ \mathbb{Z}/2\mathbb{Z} $ is coprime with $ p $, $ H^*(\mathbb{Z}/2\mathbb{Z};\mathbb{F}_p) \cong \mathbb{F}_p $ concentrated in degree $ 0 $ and $ H^*(\mathbb{Z}/2\mathbb{Z};\sgn) = 0 $. Plugging these in the Lyndon--Hochschild--Serre spectral sequences above, we see that the $ E_2 $ page of one of them is concentrated in the first row, while the other is entirely zero. Therefore, they both collapse at the second page and provide the desired isomorphism.
\end{proof}

We also record the following simple remark.
\begin{corollary} \label{cor:res mod p}
Let $ p $ be an odd prime. Then $ \res_n \colon H^*(\Bgroup_n; \mathbb{F}_p) \to H^*(\Bgrouppos{n}; \mathbb{F}_p) $ is an isomorphism.
\end{corollary}
\begin{proof}
The same Lyndon--Hochschild--Serre spectral sequence used in the proof of Theorem \ref{thm:cohomology alternating group mod p} shows that $ H^*(\Bgroup_n; \mathbb{F}_p) \cong H^*(\Sigma_n; \mathbb{F}_p) $ via the projection map $ \Bgroup_n \to \Sigma_n $.
\end{proof}

%%% Local Variables:
%%% TeX-master: "Unordered flag varieties.tex"
%%% End:

%% file: Section4.tex
% !TEX root = Unordered flag varieties.tex

\section{Stable cohomology of complete unordered flag varieties}

We recall the following classical theorem about Borel spectral sequence (i.e., Serre spectral sequence of fiber bundles $ X \rightarrow X_{hG} \rightarrow BG $) associated with the Borel construction (or homotopy quotient) $ X_{hG} $. From now on we will be using this theorem in many of our proofs and computations without explicitly referring to it every time.

\begin{theorem}[\cite{mcc}] \label{generalSSS} 
Let $X \rightarrow X_{hG} \rightarrow BG$ be the Borel fibration associated
to a $G$-space $X$. There is a first quadrant spectral sequence of algebras $E_r^{*,*},d^r\}$, converging to $H^*(X_{hG};\mathbb{F}_2)$ as an algebra with
\[ E_2^{k,l}\cong H^k(BG;H^l(X)),\]
where $ H^l(X) $ is understood as a $ G $-representation via the monodromy action of $ \pi_1(BG) $ onto the fiber.
If $ G $ acts trivially on the cohomology $H^*(X)$ and field coefficients are used, then by K\"unneth's theorem
\[
E_2^{k,l} \cong H^k(BG) \otimes H^l(X).
\]
\end{theorem}
If $ G $ acts freely on $ X $, then $ X_{hG} $ is known to be homotopy equivalent to the strict quotient $ X/G $.

\begin{proposition} \label{prop:DnBU(1)-BN(n)-iso}
    There are homotopy equivalences $BN(n) \simeq D_n (BU(1)_+)$ and $B\Bgroup_n \simeq D_n (BO(1)_+)$ for all $n\ge 1$.
\end{proposition}
\begin{proof}
    We provide details for the complex case. The real case is similar. It is enough to check that $D_n (BU(1)_+)$ is the base space of a fibration with fiber $N(n)$ and a contractible total space. Note that \[ D_n (BU(1)_+) = E\Sigma_n \times_{\Sigma_n} BU(1)^n \cong \mathrm{Conf}_n (\R^{\infty}) \times_{\Sigma_n} (\mathbb{P}^{\infty} (\mathbb{C}))^n \] and consider the contractible space $E= \mathrm{Conf}_n (\R^{\infty}) \times (S^{\infty})^n$. There is an action of $N(n) = \Sigma_n \ltimes U(1)^n$ on $E$ defined as follows: \begin{itemize}
        \item $\Sigma_n$ acts diagonally on $\mathrm{Conf}_n (\R^{\infty}) \times (S^{\infty})^n$.
        \item $U(1)^n$ acts on $(S^{\infty})^n$ factor by factor by identifying $U(1) \cong S^1$.
    \end{itemize}
    This gives a well-defined action and a principal $N(n)$-bundle $E\rightarrow D_n (BU(1)_+)$.
\end{proof}
There are inclusions $U(n) \hookrightarrow U(n+1)$ and $O(n) \rightarrow O(n+1)$ given by \[ M \longmapsto \begin{bmatrix} M & 0 \\ 0 & 1\end{bmatrix} \] which restrict to inclusions $N(n) \hookrightarrow N(n+1)$ and $\Bgroup_{n} \hookrightarrow \Bgroup_{n+1}$. Through the isomorphism of Proposition~\ref{prop:DnBU(1)-BN(n)-iso}, these inclusions correspond to the stabilization maps in $D_n (BU(1)_+)$ and $D_n (BO(1)_+)$.
\begin{corollary}
     The sequence of spaces $\{ BN(n)\}_n$ and $\{ B\Bgroup_n \}_{n}$ exhibit homological stability and if $N(\infty) = \varinjlim N(n)$ and $\Bgroup_{\infty} = \varinjlim \Bgroup_n $, then $H^* (BN(\infty); \F_p) \cong A_{\infty} (BU(1))$ and $H^* (B\Bgroup_{\infty} ; \F_p) \cong A_{\infty} (BU(1))$.
\end{corollary}
The above morphisms also determine quotient maps \begin{align*}
    i_n^{\C} \colon \frac{U(n)}{N(n)} \longrightarrow \frac{U(n+1)}{N(n+1)} \\ i_n^{\R} \colon \frac{O(n)}{\Bgroup_{n}} \longrightarrow \frac{O(n+1)}{\Bgroup_{n+1}} 
\end{align*}  and in cohomology \begin{align*}
    (i_n^{\C})^* \colon H^* \Big ( \frac{U(n+1)}{N(n+1)} ; \F_p \Big ) \longrightarrow H^* \Big ( \frac{U(n)}{N(n)} ; \F_p \Big ) \\ (i_n^{\R})^* \colon H^* \Big ( \frac{O(n+1)}{\Bgroup_{n+1}}; \F_p \Big ) \longrightarrow H^* \Big ( \frac{O(n)}{\Bgroup_{n}} ; \F_p \Big ).
\end{align*}
\begin{proposition}
\label{prop:hom-stab-UN}
     The sequence of spaces $\{ U(n)/N(n)\}_n$ and $\{ O(n)/\Bgroup_{n} \}_n$ exhibit homological stability, i.e., for all $d\ge 0$ there exists $N \in \N$ such that for all $n\ge N$, \begin{align*}
         (i_n^{\C})^* \colon H^d \Big ( \frac{U(n+1)}{N(n+1)}; \F_p \Big ) &\longrightarrow H^d \Big ( \frac{U(n)}{N(n)}; \F_p \Big ) \\ 
         (i_n^{\R})^* \colon H^d \Big ( \frac{O(n+1)}{\Bgroup_{n+1}}; \F_p \Big ) &\longrightarrow H^d \Big ( \frac{O(n)}{\Bgroup_{n} }; \F_p \Big )
     \end{align*}  are isomorphisms.
\end{proposition}
\begin{proof}
    From the previous corollary, $\{ BN(n) \}_n$ exhibits homological stability. Also, it is known classically that $\{ U(n)\}_n$ exhibits homological stability. Consider the Serre spectral sequence associated with the fiber sequence \begin{equation} \label{eq:fib-Fln}
        U(n) \longrightarrow \frac{U(n)}{N(n)} \longrightarrow BN(n).
    \end{equation} 
    Since homological stability holds for both the fiber and the base, it must hold at the level of $E_2$-page. Since the spectral sequence converges, it must hold at the $E_{\infty}$-page. The proof for $\{ O(n)/\Bgroup_{n} \}$ is similar.
\end{proof}
Proposition~\ref{prop:hom-stab-UN} allows us to discuss the stable cohomology of $U(n)/N(n)$ and $O(n)/\Bgroup_{n}$. Let us denote the limit of $\UFl_n (\C)$ and $\UFl_n (\R)$ by \[ \UFl_{\infty} (\C) = \varinjlim_n \UFl_n (\C) \text{ and } \UFl_{\infty} (\R) = \varinjlim_n \UFl_n (\R)  \] respectively. Then we can write \begin{align*}
    H^* ( \UFl_{\infty} (\C); \F_p ) &=  \varprojlim_n H^* ( \UFl_n (\C) ; \F_p ) \\ 
    H^* ( \UFl_{\infty} (\R); \F_p ) &=  \varprojlim_n H^* ( \UFl_n (\R) ; \F_p ).
\end{align*}

Next, we will define a rank function on the decorated Hopf monomials. This rank function establishes a filtration on $A_X$, which holds significant importance in our examination of the pullback of the Chern classes under $f_n$.

\begin{definition} \label{dfn:rank}
    Let $ X $ be a topological space, whose cohomology is concentrated in even degrees if $ p > 2 $. The \textit{rank of a decorated Hopf monomial} in $A_{X}$ denoted as $\rk$ is defined via the following rules: \begin{itemize}
        \item The rank of a decorated gathered block $(b,a)$ is $0$ if $b$ contains at least one of $\gamma_{k,l}$ (respectively $\alpha_{j,k}$, $\beta_{i,j,l}$, or $\gamma_{k,l}$).
        \item The rank of a decorated gathered block of the form $(1_n, a)$ is $n |a|$.
        \item $\rk (b_1 \odot \cdots \odot b_r) = \sum_{i=1}^r \rk (b_i)$.  
    \end{itemize}
    We also define $\rk (x\otimes y) := \rk (x) + \rk (y)$ in $A_{X} \otimes A_{X}$.
\end{definition}
\begin{definition} \label{dfn: rank-filtation}
     We define the \textit{rank filtration} $\mathcal{F}_{*}$ for $A_{X}$ by setting $\mathcal{F}_n$ as the linear span of decorated Hopf monomials $x$ with $\rk (x) \le n$. Moreover, we define a rank filtration on $A_{\infty} (X) $, by defining $\mathcal{F}_n$ as the linear span of stabilized Hopf monomials $x\odot 1_{\infty}$ where $\rk (x)\leq n$. With a slight abuse of notation, we also denote it as $\mathcal{F}_{*}$.
\end{definition}

The rank filtration is exhaustive, increasing, and bounded from below. Moreover, due to the Hopf ring distributivity law, it is multiplicative with respect to the cup product: $\mathcal{F}_m \cdot \mathcal{F}_n \subseteq \mathcal{F}_{m+n}$. \\

\subsection{The complex case}

The inclusions $N(n) \hookrightarrow U(n)$ induce maps $f_n \colon BN(n) \rightarrow BU(n)$. We denote the limiting map by $f\colon BN(\infty) \rightarrow BU(\infty)$. So, 
 \[ f^* \colon H^* (BU(\infty); \F_p) \cong \F_p [c_1, c_2, \dots] \longrightarrow H^* (BN(\infty); \F_p) \cong A_{\infty} (BU(1)) . \]

Let us recall some classical results from \cite{Milnor-Stasheff}. The vector bundle corresponding to $\varphi \in [X, BU(n)]$ is given by the pullback $ \varphi^* (\eta_n)$ of the universal vector bundle $\eta_n $ over $BU(n)$. Taking $X = BU(k) \times BU(n-k)$, the bundle $\eta_k \oplus \eta_{n-k}$ is classified by the homotopy class of a map $\delta_{k,n-k} \colon BU(k) \times BU(n-k) \rightarrow BU(n)$, induced by the direct sum of matrices. The induced map in cohomology \[ \Delta_{k,n-k} \colon H^* (BU(n)) \longrightarrow H^* (BU(k)) \otimes H^* (B(n-k)). \] is associative and commutative up to isomorphism, and the direct sum of these provides an associative and commutative coproduct \[ \Delta \colon \bigoplus_{n\ge 0} H^* (BU(n)) \longrightarrow  \Big ( \bigoplus_{n\ge 0} H^* (BU(n)) \Big ) \otimes \Big ( \bigoplus_{n\ge 0} H^* (BU(n)) \Big ). \]
We can regard the morphism $\Delta$ as a coproduct and note that it commutes with the cup product $\cdot$ as it is induced by homotopy classes of topological maps. This makes $\bigoplus_{n\ge 0} H^* (BU(n))$ a bialgebra. The stabilization maps of $ BU(n) $ preserve $ \Delta $, hence the limit $ H^*(BU(\infty)) $ is also a bialgebra. In fact,
$ \bigoplus_n f_n^* \colon \bigoplus_{n} H^*(BU(n)) \to A_{BU(1)} $ and $ f^* \colon H^*(BU(\infty)) \to {A}_{\infty} (BU(1)) $ are bialgebra morphisms.

Chern classes of vector bundles over $X$ correspond to the pullbacks of the universal Chern classes along its classifying map. Since, $\delta_{k,n-k}^* (\eta_n)$ is isomorphic to $\eta_k \oplus \eta_{n-k}$, we have $\Delta (c_k) = \sum_{i=0}^k c_i \otimes c_{k-i}$. 

\begin{lemma} \label{lema:chernformula}
    For $k \le n$ let $c_k \in H^{2k} (BU(n); \F_p)$ be the $k$-th universal Chern class. Then $f_n^* (c_k) = c_{[k]} \odot 1_{n-k} \mod \mathcal{F}_{2k-1}$.  
\end{lemma}
\begin{proof}
    We prove the lemma by induction on $k$. Since the only two-dimensional classes in $H^* (BN(n); \F_p)$ are $c_{[1]} \odot 1_{n-1}$, $\gamma_{1,1}^2 \odot 1_{n-2}$, and $\gamma_{1,2} \odot 1_{n-4}$, we have \[ f_n^* (c_1) = \lambda c_{[1]} \odot 1_{n-1} + \mu \gamma_{1,1}^2 \odot 1_{n-2} + \nu \gamma_{1,2} \odot 1_{n-4} \] for some $\lambda, \mu, \nu \in \F_p$. Since the restriction of $c_1$ to $H^* (BN(1); \F_p)$ is $c_{[1]}$, we must have $\lambda =1 $. Hence, $f_n^* (c_1) = c_{[1]} \odot 1_{n-1} + x$, with $x\in \mathcal{F}_{0}$. \\
    We assume by the inductive hypothesis that the lemma is true for all $1<l<k$. Consider the reduced coproduct map \[ \overline{\Delta} \colon H^* (BN(n); \F_p) \longrightarrow \bigoplus_{l=1}^{n-1} H^* (BN(l); \F_p) \otimes H^* (BN(n-l); \F_p) .\]
    Recall that  $\Delta_{l, n-l} (c_k) = \sum_{i=0}^k c_i \otimes c_{k-i}$ and the maps $f_n$ induce coalgebra maps. Therefore we have \[ \overline{\Delta} (f_n^* (c_k)) = \sum_{l=1}^{n-1} \sum_{i=0}^k f_l^* (c_i) \otimes f_{n-l}^* (c_{k-i}). \]   
    By the inductive hypothesis, \[ \overline{\Delta} (f_n^* (c_k)) = \sum_{l=1}^{n-1} \sum_{i=0}^k (c_{[i]} \odot 1_{l-i}) \otimes (c_{[k-i]} \odot 1_{n-l-k+i}) + x = \overline{\Delta} (c_{[k]} \odot 1_{n-k}) +x \] where $x \in \mathcal{F}_{k-1}$. The reduced coproduct $\overline{\Delta}$ preserves rank, hence \[ f_n^* (c_k) = c_{[k]} \odot 1_{n-k} \mod \mathcal{F}_{2k-1} + y,\] where $y\in \Prim (A_{BU(1)})$. However, all the primitive elements have rank $0$ except for $c_{[1]}^j$ ($j\ge 1$), which are in component one. Hence, $f_n^* (c_k) = c_{[k]} \odot 1_{n-k} \mod \mathcal{F}_{2k-1}$.  
\end{proof}

To fully compute $f_n^* (c_k)$ we can use an inductive argument using the coproduct and Kochman's formulas.

\begin{theorem} (Kochman's formulae for $ BU(\infty) $, \cite{Kochman}) \label{kof}
    Let \[ Q_*^r \colon H^d (BU(\infty); \F_2) \rightarrow H^{d-r} (BU(\infty); \F_2)\] be the dual of the $r$-th order Dyer-Lashof operation in homology. Then, for all $r\le k \le n$, the following identity holds:
    \[  Q_*^{2r} (c_k) = \binom{r-1}{k-r-1} c_{k-r} . \] 
    Moreover if $[n] \in H_0 (BU(\infty) \times \Z; \F_2) \cong \Z$ is the class corresponding to $n\in \Z$, then \[ Q^{2r} (1 \otimes [1]) = (c_1^r)^{\vee} \otimes [2] .\]
    Moreover, if $ p > 2 $ is a prime number, the dual Dyer-Lashof operations \[  Q^r_* \colon H^d(BU(\infty); \mathbb{F}_p) \to H^{d-2r(p-1)}(BU(\infty); \mathbb{F}_p) \] satisfies
    \begin{gather*}
    Q_*^r(c_k) = (-1)^{r+k} \left( \begin{array}{c}
    r-1 \\
    pr-k
    \end{array} \right) c_{k-r(p-1)} \\
    \tag*{and} Q^r(1 \otimes [1]) = (c_{p-1}^r)^\vee \otimes [p].
    \end{gather*}
\end{theorem}

As $H_* ( BU(\infty) \times \Z; \F_p )$ is generated as an algebra by the homology of its component corresponding to $0 \in \Z$ and $[1]$, Theorem~\ref{kof} and the Cartan formula fully determine the Dyer-Lashof operation on this Hopf algebra. The above identities are valid in the stable cohomology ring $H^* (BU(\infty) \times \Z; \F_p)$. They also determine the $Q_*^{r}$ unstably because the universal map from $\coprod_{n\ge 0} BU(n)$ to its group completion $BU(\infty) \times \Z$ is injective in homology. It is worth noting that, if $ p = 2 $, $Q_*^{2r-1} = \beta Q_*^{2r}$, where $\beta$ is the mod $2$ Bockstein homomorphism. Therefore, in this case knowledge of $Q_*^{2r}$ is sufficient to determine $Q_*^{2r-1}$.\\

Let us denote the primitive elements in $H^* (D_n (BU(1)_+); \F_p)$ as \[ \mathrm{Prim}_n (A_{BU(1)} ) = \mathrm{Prim} ( A_{BU(1)} ) \cap A_{BU(1)}^{n,*}. \] 
The module $\mathrm{Prim}_n (A_{BU(1)} )$ injects into the $n$-th component of the $\odot $-indecomposables $\mathrm{Indec} ( A_{BU(1)} )$ of $A_{BU(1)}$. The following proposition can be deduced from \cite{May-Cohen}.
\begin{proposition} \label{indec}
    Let $ X $ be a topological space and let $ \mathcal{B} $ be a homogeneous vector space basis for $ H^*(X; \mathbb{F}_p) $. The module $\mathrm{Indec} (A_X )$ is the linear dual of the subspace $M$ with basis $\{ Q^I (x^{\vee}) \mid I \text{ admissible, } k\ge 0 , x \in \mathcal{B}\}$. Moreover, for $m\ge 0$, the pairing between $\mathrm{Prim}_{2^m} ( A_X )$ and the subspace $M'_m$ with basis $\{ Q^I (x^{\vee}) \mid I \text{ admissible, } l(I) = m, e(I)>0, k\ge 0 , x \in \mathcal{B} \}$ is perfect. Here $l(I)$ and $e(I)$ are as defined in \cite{May-Cohen}.  
\end{proposition}

Using Proposition~\ref{indec} with $ X = BU(1) $ and the monomial basis on $ c $, we will be able to compute $f_n^* (c_k)$ recursively as follows:
\begin{enumerate}
    \item From the proof of Lemma~\ref{lema:chernformula} we see that $f_1^* (c_1) = c$ and the reduced coproduct is given by \[ \overline{\Delta} (f_n^* (c_k)) = \sum_{l=1}^{n-1} \sum_{i=0}^k f_l^* (c_i) \otimes f_{n-l}^* (c_{k-i}). \] Hence, the knowledge of $f_l^* (c_i)$ for $l<n$ is enough to recursively compute $\overline{\Delta} (f_n^* (c_k))$. This gives us $f_n^* (c_k)$ up to primitives.
    \item In order to find the primitive part, we note that $\mathrm{Prim}_n  ( A_{BU(1)} )$ does not include any $\gamma_{k,l}$ unless $n$ is equal to $2^m$ for some $m \geq 0$. Then by applying Proposition~\ref{indec}, we can determine the primitive part.
\end{enumerate}

% The reduced coproduct formula and knowledge of $f_m^* (c_k)$ for $m < n$ completely determines $f_n^* (c_k)$ when $n \neq 2^k$. When $n=2^k$, we can deal with the primitive elements by repeatedly using Theorem~\ref{kof} to determine the pairing of the primitive element of right dimension with $f_{2^k}^* (c_k)$. The following lemma provides more insight into this pairing

\begin{lemma} \label{pairing}
    Let $k\ge 1$ and $\langle \cdot , \cdot \rangle \colon H^* (BN(p^k); \F_p) \otimes H_* (BN(p^k); \F_p) \rightarrow \F_p$ be evaluation pairing between cohomology and homology of the space $BN(p^k)$. Then, if $ p = 2 $, \[ \big \langle f_{2^k}^* (c_{2^k-1}) , Q^{2^{k}} Q^{2^{k-1}} \cdots Q^4 Q^2 (1) \big \rangle =1 , \] where $1 \in H_0 (BN(1); \F_2)$.
    Similarly, if $ p > 2 $, \[ \big \langle f_{p^k}^* (c_{p^k-1}) , Q^{p^{k-1}} Q^{p^{k-2}} \cdots Q^p Q^1 (1) \big \rangle =1 , \] where $ 1 \in H_0 (BN(1); \F_p)$.
\end{lemma}
\begin{proof}
    We prove the identity for $ p = 2 $, as the argument for $ p > 2 $ is entirely similar. In the proof, we will use the notation $\langle \cdot, \cdot \rangle$ for the evaluation pairing between cohomology and homology in any topological space, not just in $BN(2^k)$. We also use the naturality of this pairing with respect to continuous maps of spaces without further notice. since the map $\coprod_n BN(n) \rightarrow \coprod_n BU(n)$ is an $E_{\infty}$ mapping, it preserves the Dyer--Lashof operations in homology. Additionally, the map in cohomology $H^* (BU(\infty); \F_2) \rightarrow H^* (BU(n); \F_2)$ is surjective and the dual map $H_* (BU(n); \F_2) \rightarrow H_* (BU(\infty); \F_2)$ is injective. This allows us to perform computations in $BU(\infty) \times \Z$. So, we prove the equivalent formula \[ \big \langle c_{2^k-1} \otimes [2^k], Q^{2^{k}} Q^{2^{k-1}} \cdots Q^4 Q^2 (1\otimes [1]) \big \rangle =1 , \] where $[1] \in H_0 (BU(\infty) \times \Z ; \F_2)$.
    We proceed by induction on $k$. For $k=1$, \[ \big \langle c_{1} \otimes [2], Q^2 (1 \otimes [1])  \big \rangle = \big \langle c_1 \otimes [2], c_1^{\vee} \otimes [2] \big \rangle = 1. \] For $k>1$, \begin{align*}
        \big \langle & c_{2^k-1} \otimes [2^k], Q^{2^{k}} Q^{2^{k-1}} \cdots Q^4 Q^2 (1\otimes [1]) \big \rangle \\ 
        & \quad = \big \langle Q_*^{2^k} (c_{2^k-1} \otimes [2^{k}]), Q^{2^{k-1}} \cdots Q^4 Q^2 (1 \otimes [1]) \big \rangle \\
        & \quad = \big \langle \binom{2^{k-1}-1}{2^{k-1}-2} c_{2^k -1 -2^{k-1}} \otimes [2^{k-1}], Q^{2^{k-1}} \cdots Q^2 (1 \otimes [1]) \big \rangle \text{ (by Theorem~\ref{kof})} \\
        & \quad = \big \langle  c_{2^{k-1}-1} \otimes [2^{k-1}], Q^{2^{k-1}} \cdots Q^4 Q^2 (1\otimes [1]) \big \rangle 
    \end{align*}
    which is equal to $1$ by the inductive hypothesis.
\end{proof}
We will now derive the pullback formula for the Chern classes. We will state the formula in terms of the total Chern class \[ c_* = 1+c_1 +c_2 +\cdots \in H^* (BU(\infty); \F). \] To obtain the pullback formula for a specific Chern class $c_k$, we only need to consider the classes of cohomological dimension $2k$. We will elaborate on this topic later in the section.

\begin{theorem} \label{pullbackformula} (Pullback Formula)
    Let $f\colon BN(\infty) \rightarrow BU(\infty)$ be the limiting map induced by the inclusions $N(n) \hookrightarrow U(n)$ and $\mathcal{X} = \{ \underline{a} = (a_0, a_1, a_2, \dots ) \in \N^{\N^*} \mid a_n = 0 \text{ for } n \gg 0 \}$ be the set of infinite sequences of natural numbers that have finite support (i.e., that are eventually zero). Then, in cohomology with coefficients in $    \mathbb{F} $, \begin{equation} \label{pullback-chern}
        f^* (c_*) = \left\{ \begin{array}{ll}
        \sum_{\underline{a} \in \mathcal{X}} c_{[a_0]} \odot \big ( \bigodot_{i,a_i \neq 0} \gamma_{i, a_i}^2 \big ) \odot 1_{\infty} & \mbox{if } \characteristic(\F) = 2 \\
        \sum_{\underline{a} \in \mathcal{X}} c_{[a_0]} \odot \big ( \bigodot_{i,a_i \neq 0} \gamma_{i, a_i} \big ) \odot 1_{\infty} & \mbox{if } \characteristic(\F) > 2 \\
        \sum_{n=0}^\infty c_{[n]} \odot 1_{\infty} & \mbox{if } \characteristic(\F) = 0 \\
        \end{array} \right. .
    \end{equation} 
    Note that requiring that every sequence in $\mathcal{X}$ has finite support guarantees that the transfer product in the formula (\ref{pullback-chern}) is iterated a finite number of times and well-defined.
\end{theorem}
\begin{proof}
    We first assume that $ \characteristic(\F) = 2 $. Since $ \F $ is a free module over $ \F_2 $, we can assume without loss of generality that $ \F = \F_2 $.
    Let us denote by $x_* $ the class \[ x_* = \sum_{\underline{a} \in \mathcal{X}} c_{[a_0]} \odot \big ( \bigodot_{i,a_i \neq 0} \gamma_{i, a_i}^2 \big ) \odot 1_{\infty} . \]
    We can write $x_* = \sum_{n\ge 0} x_n$, where $|x_n| = 2n$, as $\gamma_{i,a_i}^2$ and $c_{[a_0]}$ are always even dimensional cohomology classes. We have that $\Delta x_* = x_* \otimes x_*$.\\
    % Recall that $\Delta \gamma_{k,l} = \sum_{i=0}^l \gamma_{k,i} \otimes \gamma_{k, l-i}$ and $\bigoplus_{n\ge 0} H^* (BN(n); \F_2)$ with $\odot$ and $\Delta$ forms a bialgebra. It follows 
    Extrapolating the $2n$-dimensional part, the coproduct of $x_n $ must be \[ \Delta x_n = \sum_{j=0}^n x_j \otimes x_{n-j} , \] similar to those satisfied by the pullbacks of the Chern classes \[ \Delta f^* (c_n) = \sum_{j=0}^n f^* (c_j) \otimes f^* (c_{n-j}) . \]
    When $ m, n $ are finite positive numbers, the restriction of $c_m$ to $H^* (BU(n); \mathbb{F}_2)$ is $0$ for all $m>n$. Therefore, the restriction of $f^* (c_m)$ to $H^* (BN(n); \mathbb{F}_2)$ is also $0$ for all $m>n$.\\
    \textbf{Claim 1.} The restriction of $x_m$ to $H^* (BN(n); \F_2)$ is $0$ if $m > n$. \\
    \textit{Proof of Claim 1.} Recall that $|\gamma_{k,l}| = l(2^k -1)$ and $|c_{[k]}| = 2k$. Therefore the cohomological dimension of  \[ x_{\underline{a}} := c_{[a_0]} \odot \big ( \bigodot_{i,a_i \neq 0} \gamma_{i, a_i}^2 \big ) \odot 1_{\infty} \] corresponding to a sequence $\underline{a} = (a_0, a_1, a_2, \dots ) \in \mathcal{X}$ is $2a_0 + \sum_{i=1}^{\infty} 2a_i (2^i -1)$. The class $x_m$ is of the form $\sum_{\underline{a}} x_{\underline{a}}$ for some sequences $\underline{a} \in \mathcal{X}$. If the cohomological dimension of $x_{m}$ is bigger that $2n$, i.e., $2m = 2 a_0 + \sum_{i=1}^{\infty} 2a_i (2^i -1)  > 2n$, then $a_0 + \sum_{i=1}^{\infty} a_i 2^i > n$. Also, note that $x_{\underline{a}}$ is in the $(a_0 + \sum_{i=1}^{\infty} a_i 2^i)$-th component as the width of the skyline diagram corresponding to $x_{\underline{a}}$ is $a_0 + \sum_{i=1}^{\infty} a_i 2^i$, which is bigger than $n$. Therefore the restriction of $x_{\underline{a}}$ to $H^* (BN(n); \F_2$ is $0$. \hfill $\blacksquare$ \\
    Assume that, $f^* (c_*) \neq x_*$. Then there exists a minimal $m$ such that $f^* (c_m) \neq x_m$. By minimality and the coproduct formulas, we have that $x_m -f^* (c_m)$ is primitive of dimension $m$. The primitives in $A_{\infty} (BU(1))$ are necessarily a linear combination of elements of the form $b\odot 1_{\infty}$, where $b$ is a primitive gathered block in a component equal to $2^k$ for some $k$. We will now determine the condition on the gathered block $b$ that can appear in this linear combination.\\
    \textbf{Claim 2.} For $x_m - f^* (c_m)$ to be non-zero, we must have $m=2^k -1$ for some $k\in \N$ and \[ x_{2^k - 1} - f^* (c_{2^k -1}) = \lambda_k (\gamma_{k,1}^2 \odot 1_{\infty}) . \]
    \textit{Proof of Claim 2.} By Claim 1, $x_m - f^* (c_m)$ is in the kernel of the restriction map $ H^* (BN(\infty); \F_2) \xrightarrow{res} H^* (BN(n); \F_2)$ for $n < m$. Note that $\mathrm{ker} \big ( H^* (BN(\infty); \F_2) \xrightarrow{res} H^* (BN(m); \F_2) \big )$ is generated by $y\odot 1_{\infty}$ for $y$ a full width Hopf monomial with width strictly bigger than $m$. This rules out all primitive gathered blocks in a component equal to $2^k$ for some $k < \lceil \log_2 (m) \rceil$. It follows from the calculation in \cite{Sinha:12} that only primitive gathered block in component $2^k$ whose cohomological dimension is a multiple of $2$ and does not exceed $2^{k+1}$ is $\gamma_{k,1}^2$ of dimension $ 2(2^k-1)$. Therefore $m$ must be equal to $2^{k}-1$ for some $k\in \N$ and \begin{equation} \label{pulpf}
        x_{2^k - 1} - f^* (c_{2^k - 1}) = \lambda_k (\gamma_{k,1}^2 \odot 1_{\infty}),
    \end{equation} for some $\lambda_k \in \F_2$.   \hfill $\blacksquare$ \\
    \textbf{Claim 3.} For all $k$, we have that $\lambda_k = 0$.\\
    \textit{Proof of Claim 3.} We prove this by restricting (\ref{pulpf}) to $H^* (BN(2^k); \F_2)$ and evaluate the restriction on the homology class $Q^{2^k} Q^{2^{k-1}} \cdots Q^{4} Q^{2} (1)$, where $1 \in H_0 (BN(1); \F_2)$. Also note that $\gamma_{k,1}^2 \odot 1_{\infty}$ restricts to $\gamma_{k,1}^2$ in $H^* (BN(2^k); \F_2)$. Using a similar argument used to prove Claim 1, we can show that the only addends in $x_*$ whose width is at most $2^k$ and cohomological dimension is $2(2^k-1)$ are $c_{[2^k-2^l]} \odot \gamma_{l,1}^2 \odot 1_{\infty}$, with $ 1 \leq l \leq k $, and $c_{[2^k -1]} \odot 1_{\infty}$. Hence, the restriction of $x_m$ to $H^* (BN(2^k); \F_2)$ is $\sum_{l=1}^k c_{[2^k-2^l]} \odot \gamma_{l,1}^2 + c_{[2^k -1]} \odot 1_1$ and \[ f_{2^k}^* (c_{2^k -1}) = (1-\lambda_k) \gamma_{k,1}^2 + \sum_{l=1}^{k-1} c_{[2^k-2^l]} \odot \gamma_{l,1}^2 + c_{[2^k -1]} \odot 1_1 .  \] 
    From the definition of $\gamma_{k,1}$, we have $\langle \gamma_{k,1}^2 , Q^{2^{k}} Q^{2^{k-1}} \cdots Q^4 Q^2 (1) \rangle = 1$ and from Lemma~\ref{pairing}, we have $\langle f_{2^k}^* (c_{2^k-1}) , Q^{2^{k}} Q^{2^{k-1}} \cdots Q^4 Q^2 (1) \rangle = 1$. All the other addends pair trivially with $ Q^{2^{k}} Q^{2^{k-1}} \cdots Q^4 Q^2 (1) $. This implies $\lambda_k = 0$.              \hfill $\blacksquare$ \\
    Claim 3 implies that, $x_m - f^* (c_m) = 0$ for any $m \in \N$. This contradicts our previous assumption that $x_* \neq f^* (c_*)$. Therefore, we conclude that $f^* (c_*)$ must be equal to $x_*$ and the theorem is proved.

    If $ \characteristic(\F) = p > 2 $, then the argument is essentially the same with $ x_* = \sum_{\underline{a} \in \mathcal{X}} c_{[a_0]} \odot \big ( \bigodot_{i,a_i \neq 0} \gamma_{i, a_i} \big ) \odot 1_{\infty} $, but the analog of Claim 2 is slightly more subtle. In order to prove that if $ m $ is the minimal index such that $ x_m - f^*(c_m) \not= 0 $ then $ m = p^k-1 $ and $ x_{p^k-1} - f^*(c_{p^k-1}) = \lambda_k (\gamma_{k,1} \odot 1_\infty) $, one is lead to consider primitive gathered blocks in component $ p^k $ whose cohomological dimension does not excees $ 2p^k $. There are classes satisfying this contraint, in dimension $ 2(p^k-p^i-1) $, $ 2(p^k-p^i)-1 $ for $ 1 \leq i \leq k $ and $ 2(p^k-1) $. The classes of dimension $ 2(p^k-p^i)-1 $ are ruled out because $ |x_m - f^*(c_m)| $ is even, while $ 2(p^k-p^i-1) $ are ruled out because their Bockstein is non-zero, while $ \beta(x_{p^k-1} - f^*(c_{p^k-1})) = 0 $. The only remaining primitive is $ \gamma_{k,1} $, of dimension $ 2(p^k-1) $.

    If $ \characteristic(\F) = 0 $, then Claims 2 and 3 are unnecessary, because the only non-zero primitives lie in component $ 1 $.
\end{proof}

Let us now recall the definition of a regular sequence. The following definition and the subsequent lemma is a classical result \cite{Eisenbud99}. 

\begin{definition}
    Let $A$ be a commutative algebra over a field $\F$. A sequence $\{ a_1, \dots , a_n \}$ in $A$ is said to be a \textit{regular sequence} if $a_1$ is not a zero divisor in $A$ and for all $2\le i \le n$, $a_i$ is not a zero divisor in $A/(a_1, \dots, a_{i-1})$.
    
    If $\{ F_i \}_{i\in I}$ is a filtration of $A$, with $I$ being a totally ordered set, then we define the \textit{associated graded algebra} of $A$ as \[ \gr (A):= \bigoplus_{i\in I} F_i / F_{i-1} . \]
\end{definition}

\begin{lemma}[\cite{Eisenbud99}] \label{lema:graded} 
    Let $A$ be an algebra over a field $\F$. Let $\mathcal{F}$ be a multiplicative filtration of $A$ bounded below. For $i=1,2,\dots , n $ let $x_i \in \mathcal{F}_{j_i}/\mathcal{F}_{j_i -1} \subseteq \gr_{\mathcal{F}} (A)$ be graded elements in the associated graded algebra lifting to $\widetilde{x}_i \in \mathcal{F}_{j_i} \subseteq A$. If $\{ x_1 , \dots x_n \} $ is a regular sequence in $\gr_{\mathcal{F}} (A)$, then $\{ \widetilde{x}_1 , \dots \widetilde{x}_n \}$ is a regular sequence in $A$.   
\end{lemma}

\begin{lemma} \label{lema:associated-graded}
    Let $A = H^* (BN(\infty); \F_p) $, and $\mathcal{F}_{*}$ be the rank filtration. The associated graded algebra is isomorphic to the polynomial algebra \begin{equation}
        \gr_{\mathcal{F}} (A) = \mathcal{F}_0 [c_{[1]} \odot 1_{\infty} , c_{[2]} \odot 1_{\infty}, \dots , c_{[k]} \odot 1_{\infty}, \dots  ]
    \end{equation}
\end{lemma}

\begin{proof}
    Every stabilized Hopf monomial $x = c_{[k_1]}^{m_1} \odot \cdots \odot c_{[k_l]}^{m_l} \odot 1_{\infty}$ is uniquely determined by an eventually zero sequence of non-negative integers of the form \[ \underline{s} (x) = (\underbrace{m_1,\dots,m_1}_{k_1 \text{ times}},\dots, \underbrace{m_l, \dots, m_l}_{k_l \text{ times}}, 0,0,\dots ).\] Let $A'$ be the subspace of $A$ generated by the linear span of the Hopf monomials $c_{[k_1]}^{m_1} \odot \cdots \odot c_{[k_l]}^{m_l} \odot 1_{\infty}$ with $m_1 > m_2 > \cdots >m_l$ and $k_1,\dots,k_l \ge 1$. By Hopf ring distributivity, the product of any such two Hopf monomials is again of this form and hence an element of $A'$. This makes $A'$ a subalgebra of $A$. We define a total order on the set of these Hopf monomials by letting $x<y$ if and only if $\underline{s} (x) < \underline{s} (y)$ in the lexicographic order. The set of these monomials forms an ordered basis for $A'$ as a $\F_p$-vector space. Given an eventually zero sequence of non-negative integers $s$, we denote the unique basis element $x_s$ such that $\underline{s} (x_s) = s$. \\
    \textbf{Claim.} If $x= c_{[k_1]}^{m_1} \odot \cdots \odot c_{[k_l]}^{m_l} \odot 1_{\infty}$, then for $m\ge 1$ the leading term of $x\cdot (c_{[k]} \odot 1_{\infty})$ in the above ordered basis is $x_{\underline{s} (x)+ \underline{s} (c_{[k]}\odot 1_{\infty})}$.\\
    Recall from Theorem~\ref{thm:cohomology DX mod 2} that $\Delta c_{[k]} = \sum_{j=0}^k c_{[j]} \otimes c_{[k-j]}$. To prove this claim, we proceed by induction. If $l=1$, the Hopf ring distributivity implies \[ x\cdot (c_{[k]} \odot 1_{\infty}) = \sum_{j=0}^{\mathrm{min} \{ k_1, k\}} c_{[j]}^{m_1 +1} \odot c_{[k_1 -j]}^{m_1} \odot c_{[k-j]} \odot 1_{\infty} . \]
    The leading term is the one for which $j$ is maximal and this proves the claim in the case when $l=1$.\\
    To prove the general case, let $x' = c_{[k_2]}^{m_2} \odot \cdots \odot c_{[k_l]}^{m_l} \odot 1_{\infty}$. Again using the Hopf ring distributivity, we have \[ x\cdot (c_{[k]} \odot 1_{\infty}) = \sum_{j=0}^{\mathrm{min} \{ k_1, k \}}  c_{[j]}^{m_1 +1} \odot c_{[k_1 -j]}^{m_1} \odot \big ( x' \cdot ( c_{[k-j]} \odot 1_{\infty} ) \big ) .\] The leading term is the one for which $j$ is maximal and this proves the induction step. \\
    As a consequence of the Claim $A' = \F_p [c_{[1]}\odot 1_{\infty}, \dots ,c_{[k]} \odot 1_{\infty}, \dots]$ as an algebra. Therefore, to prove the lemma it is enough to check that for all $x = \overline{x} \odot 1_{\infty} \in \mathcal{F}_0$ and for all $y=\overline{y} \odot 1_{\infty} \in A'$, the cup product $x\cdot y$ is equal to $\overline{x} \odot \overline{y} \odot 1_{\infty} $ plus elements of lower rank. Unpacking the Hopf ring distributivity $x \cdot y$ is obtained combinatorially by splitting the columns of the skyline diagram of $\overline{y}$, and stacking each column either onto a column of $\overline{x}$ with the same width, or onto the $1_{\infty}$ part. Stacking a column of $\overline{y}$ onto a column of $\overline{x}$ lowers rank. Therefore the term of maximal rank is obtained by stacking $\overline{y}$ onto the $1_{\infty}$ part, which yields $\overline{x} \odot \overline{y} \odot 1_{\infty}$.
\end{proof}

\begin{theorem} \label{theo:regular}
 % Let $ f: BN(\infty) \rightarrow BU(\infty)$ be the limit of the canonical maps $f_n : BN(n) \rightarrow BU(n)$ and $f^* : H^* (BU(\infty); \F_p) \rightarrow H^* (BN(\infty); \F_p )$ be the induced map in cohomology. Then 
 The sequence $\{ f^* (c_1), f^* (c_2), \dots  \}$ in $A_{\infty} (BU(1))$ is a regular sequence.
\end{theorem}
\begin{proof}
    From Lemma~\ref{lema:associated-graded}, we have that the associated graded algebra $\gr_{\mathcal{F}} (A)$ corresponding to the rank filtration $\mathcal{F}_*$ of $A = H^* (BN (\infty); \F_p)$ is isomorphic to $\mathcal{F}_0 [c_{[1]} \odot 1_{\infty} ,   c_{[2]} \odot 1_{\infty}, \dots  ]$. Also, note that the sequence $\{ c_{[1]} \odot 1_{\infty}, c_{[2]} \odot 1_{\infty}, \dots  \}$ is a regular sequence in $\gr_{\mathcal{F}} (A)$. From Lemma~\ref{lema:chernformula}, we have $f_n^* (c_k) = c_{[k]} \odot 1_{n-k}$ modulo terms of lower rank and therefore $c_{[k]} \odot 1_{\infty} \in \gr_{\mathcal{F}} (A)$ lifts to $f^* (c_k) \in A$. Hence by Lemma~\ref{lema:graded}, $\{ f^* (c_1), f^* (c_2), \dots \}$ is a regular sequence in $A_{\infty} (BU(1))$.
\end{proof}

As a direct consequence of Theorem~\ref{theo:regular}, we have the following.
\begin{corollary} \label{cor:stable-cohomology}
    For all prime $p$, there are isomorphisms of algebras \[ H^* \Big ( \frac{U(\infty)}{N(\infty)} ; \F_p \Big ) \cong \frac{A_{\infty} (BU(1))}{(f^* (c_1), f^* (c_2), \dots )} .\]
\end{corollary}
\begin{proof}
    We consider the Serre spectral sequence associated with the fiber sequence \begin{equation} \label{eq3}
    U(\infty) \longrightarrow \frac{U(\infty)}{N(\infty)} \longrightarrow BN(\infty) \end{equation} and compare it with the Serre spectral sequence associated with the universal fibration \begin{equation} \label{eq4} U(\infty) \longrightarrow EU(\infty) \longrightarrow BU(\infty).\end{equation} Recall that $H^* (U(\infty); \F_p)$ is the exterior algebra generated by  $z_{2i-1}$ for $i=1,2, \dots$ and $z_{2i-1}$ transgresses to $c_i$ in the Serre spectral sequence associated with (\ref{eq4}). Comparing the two spectral sequences, we see that $z_{2i-1}$ transgresses to $f^* (c_i)$ in the Serre spectral sequence associated with (\ref{eq3}). Hence, the differentials are given as follows: for all $i \ge 1$, $d_{2i}\colon z_{2i-1} \mapsto f^* (c_i)$ and $d_{2i-1} \equiv 0$. By Theorem~\ref{theo:regular}, the sequence $\{ f^* (c_1), f^* (c_2), \dots \}$ is a regular sequence in $A_{\infty} (BU(1))$ and hence by \cite[Theorem~11.22]{atiyahMac} is algebraically independent over $\F_2$. Therefore, for all $i \ge 1$, $d_{2i}$ is injective on the ideal generated by $z_{2i-1}$, otherwise, there will be algebraic relations between the $f^* (c_i)$'s. The $E_{\infty}$-page is thus given by
    \[ E_{\infty}^{*,*} = \frac{A_{\infty} (BU(1))}{(f^* (c_1), f^* (c_2), \dots )} \] and the result follows.
\end{proof}

\begin{corollary}
    $H^* (\UFl_{\infty} (\C); \F_p)$ is the free graded commutative algebra generated by the stabilization of decorated gathered blocks $b\odot 1_{\infty}$ such that $\rk (b) = 0$ and satisfies the conditions of Corollary~\ref{cor:polynomial generators}.
\end{corollary}
\begin{proof}
    By Lemma~\ref{lema:associated-graded} and Corollary~\ref{cor:stable-cohomology}, the stable cohomology of the unordered flag manifold is isomorphic to $\mathcal{F}_0$ the rank zero subalgebra of $A_{\infty} (X)$. 
\end{proof}

With rational coefficients, $ A_{\infty} (BU(1)) $ is the polynomial algebra generated by $ c_{[n]} \odot 1_\infty $ for $ n \geq 1 $. Therefore, our analysis recovers the following well-known result.
\begin{corollary}
    $ H^*(\UFl_{\infty} (\C); \mathbb{Q}) \cong \mathbb{Q} $.
\end{corollary}

\begin{remark}
Fix a dimension $ d \geq 0 $. The count of Shubert classes implies that the family $ \{ H^d(\Fl_n(\mathbb{C}) \}_n $ is a finitely generated FI-module. Consequently, it exhibits representation stability. This property was first proved by Church and Farb in \cite[Theorem 7.1]{Church-Farb}.

For cohomology over fields of characteristic $ 0 $, by passing to symmetric invariants, this implies the homological stability of the cohomology of unordered flag manifolds. However, it fails to do so over fields of positive characteristics, as illustrated more generally in \cite{Nagpal}.
\end{remark}

\subsection{The real case} \label{section:4.2}

For the real case, we deal with the mod $2$, mod $p$ for $p>2$, and rational cohomologies separately.

\subsubsection{Mod $2$ cohomology}

To compute the cohomology of $ \Fl_n(\mathbb{R}) $, we chose to use the alternating subgroup $ \Bgrouppos{n} $ instead of $ \Bgroup_n $. Therefore, to compute the cohomology of the limit $ \UFl_{\infty} (\mathbb{R}) $, we need to determine the stable cohomology of $ \Bgrouppos{n} $.

\begin{corollary} \label{cor:homological stability alternating subgroups}
The sequence of spaces $ \{ B\Bgrouppos{n} \}_{n \in \mathbb{N}} $ exhibits homological stability with stabilization maps $ B\Bgrouppos{n} \to B\Bgrouppos{n+1} $ induced by the standard group inclusions $ \Bgrouppos{n} \hookrightarrow \Bgrouppos{n+1} $. Moreover, for all primes $ p $, the stable cohomology ring is
\[
H^*(B\Bgrouppos{\infty}; \mathbb{F}_p) \cong \varprojlim_n H^*(B\Bgrouppos{n}; \mathbb{F}_p) \cong \left\{ \begin{array}{ll}
\frac{ A_{\infty}(\mathbb{P}^\infty(\mathbb{R})) }{(e_\infty)} & \mbox{if } p = 2 \\
A_{\infty}(\mathbb{P}^\infty(\mathbb{R})) & \mbox{if } p > 2
\end{array} \right.,
\]
where $ e_\infty = w \odot 1_\infty + \gamma_{1,1} \odot 1_\infty $.
\end{corollary}
\begin{proof}
If $ p > 2 $, the result follows immediately from Theorem \ref{thm:cohomology alternating group mod p} and the homological stability of $ B(\Bgroup_n) $.

For $ p = 2 $, we observe that, with the notation of Definition \ref{def:basis B+}, every element of $ \mathcal{G}_{ann} \cap H^*(\Bgroup_n; \mathbb{F}_2) $ has degree at least $ 3 \lfloor n/4 \rfloor $, which goes to $ \infty $ as $ n \to \infty $.

We deduce that, for all $ d \geq 0 $ and for $ n $ large enough, in the mod $ 2 $ cohomological Gysin sequence of $ B\Bgrouppos{n} \to B\Bgroup_n $ the connecting homomorphisms is injective in degree $ d-1 $.
Therefore there is a short exact sequence
\[
0 \longrightarrow H^{d-1}(\Bgroup_n; \mathbb{F}_2) \stackrel{e_n}{\longrightarrow} H^d(\Bgroup_n; \mathbb{F}_2) \longrightarrow H^d (\Bgrouppos{n}; \mathbb{F}_2) \longrightarrow 0.
\]
The stabilization maps induce homomorphisms of short exact sequences, hence the homological stability of $ B\Bgroup_n $ implies that of $ B\Bgrouppos{n} $. Moreover, the stable cohomology of $ B\Bgrouppos{n} $ is the quotient of the stable cohomology of $ B\Bgroup_n $ by the limit of the classes $ e_n $, which is $ e_\infty $.
\end{proof}

\begin{definition}
Let $ p $ be a prime number. We define $ {A}_{\infty}' (\Bgrouppos{}) $ as the ring $ H^*(\Bgrouppos{\infty}; \mathbb{F}_p) $ computed in the previous corollary.
\end{definition}

The inclusions $ \Bgrouppos{n} \hookrightarrow SO(n) $ induce maps $ g_n \colon B\Bgrouppos{n} \rightarrow BSO(n) $. We call the limiting map $ g \colon B\Bgrouppos{\infty} \rightarrow BSO(\infty) $.
Similarly, we have maps $ h_n \colon B\Bgroup_n \rightarrow BO(n) $ and $ h \colon B \Bgroup_{\infty} \to BO(\infty) $

Note that, in analogy with the complex case, the homomorphism $ \delta_{n,m} \colon SO(n) \times SO(m) \to SO(n+m) $ given by the direct sum of matrices induces a coassociative and cocommutative coproduct $ \Delta $ on $ \bigoplus_{n \geq 0} H^*(BSO(n); \mathbb{F}_p) $. Since it is compatible with cup product, this object is a bialgebra. The stabilization maps of $ BSO(n) $ preserve $ \Delta $ and hence the limit $ H^*(BSO(\infty)) $ is also a bialgebra. The morphisms 
$ \bigoplus_n g_n^* \colon \bigoplus_n H^*(BSO(n); \mathbb{F}_p) \to \AB $ and $ g^* \colon H^*(BSO(\infty)) \to {A}_{\infty}' (\Bgrouppos{}) $ are bialgebra morphisms.

Similarly to the complex case, there is a universal oriented real vector bundle $ \eta_n $ on $ BSO(n) $. Pullbacks of $ \eta_n $ along (homotopy classes of) maps $ X \to BSO(n) $ classify the isomorphism classes of oriented real vector bundles over $ X $. Stiefel-Whitney classes of such a bundle correspond to pullbacks of the universal Stiefel-Whitney classes along its classifying map. Since $ \delta_{n,m}^*(\eta_{n+m}) $ is isomorphic to $ \eta_n \oplus \eta_m $, the formula for characteristic classes of direct sums imply that
\[
\forall k \in \mathbb{N}: \quad \Delta(w_k) = \sum_{n+m=k} w_n \otimes w_m.
\]

Similarly, $ \bigoplus_n H^*(BO(n); \mathbb{F}_p) $ and $ H^*(BO(\infty); \mathbb{F}_p) $ are bialgebras, and 
\begin{align*}
    \bigoplus_n h_n^* \colon \bigoplus_n H^*(BO(n); \mathbb{F}_p) \longrightarrow A_{\mathbb{P}^\infty(\mathbb{R})}  \\ 
\tag*{and} h^* \colon H^*(BO(\infty); \mathbb{F}_p) \longrightarrow A_\infty(\mathbb{P}^\infty(\mathbb{R})) 
\end{align*}  
are bialgebra homomorphisms. The same formula for the coproduct of universal characteristic classes holds in $ BO(n) $.

Let us define a charged version of the rank filtration when $ p = 2 $.
\begin{definition} \label{dfn:rank charged}
    Assume that $ p = 2 $. The \textit{rank of a charged Hopf monomial} in $\AB$ denoted as $\rk$ is defined as the rank of the corresponding non-charged decorated Hopf monomial in $ A_{\mathbb{P}^\infty(\mathbb{R})} $.
    We also define $\rk (x\otimes y) := \rk (x) + \rk (y)$ in $\AB \otimes \AB$.
\end{definition}
\begin{definition} \label{dfn: rank-filtation charged}
     Assume that $ p = 2 $ We define the \textit{rank filtration} $\mathcal{F}_{*}$ for $\AB$ by setting $\mathcal{F}_n$ as the linear span of charged Hopf monomials $x$ with $\rk (x) \le n$. Moreover, we define a rank filtration on ${A}_{\infty}' (\Bgrouppos{})$, by defining $\mathcal{F}_n$ as the linear span of stabilized Hopf monomials $x\odot 1_{\infty}$ where $\rk (x)\leq n$. With a slight abuse of notation, we also denote it as $\mathcal{F}_{*}$.
\end{definition}
The rank filtration is exhaustive, increasing, and bounded from below. Moreover, due to the almost-Hopf ring distributivity law, it is multiplicative with respect to the cup product: $\mathcal{F}_m \cdot \mathcal{F}_n \subseteq \mathcal{F}_{m+n}$. 

We can now compute the pullbacks $ g^*(w_k) $ of Stiefel-Whitney classes with essentially the same argument used for Chern classes in the previous subsection.

\begin{lemma} \label{lema:swformula}
\begin{enumerate}
    \item For $ 1 \leq k \leq n$ let $w_k \in H^{k} (BO(n); \F_2)$ be the $k$-th universal Stiefel-Whitney class. Then $h_n^* (w_k) = w_{[k]} \odot 1_{n-k} \mod \mathcal{F}_{k-1}$. 
    \item For $ 2 \leq k \leq n $, let $ w_k \in H^k(BSO(n); \F_2)$ be the $ k $-th universal Stiefel-Whitney class. Then $ g_n^*(w_k) = (w_{[k]} \odot 1_{n-k})^0 \mod \mathcal{F}_{k-1} $.
\end{enumerate}
\end{lemma}
\begin{proof}
The proof of (1) follows from the same argument used for the proof of Lemma \ref{lema:chernformula}, by replacing Chern classes with Stiefel-Whitney classes.

There is a commutative diagram
\begin{center}
\begin{tikzcd}
B\Bgrouppos{n} \arrow{r}{g_n} \arrow{d} & BSO(n) \arrow{d} \\
B\Bgroup_n \arrow{r}{h_n} & BO(n).
\end{tikzcd}
\end{center}
The left vertical map induces $ \res_n $ in cohomology, while pullback along the right vertical map sends $ w_k $ to $ w_k $ if $ k \geq 2 $ and $ w_1 $ to $ 0 $.
(2) follows from the induced commutative diagram in cohomology, by observing that $ \res_n $ preserves the rank filtration.
\end{proof}

There is a version of Kochman's formulae for $ BO(\infty) $.
\begin{theorem}[Kochman's formulae for $ BO $, \cite{Kochman}] \label{kof real}
Let $ Q_*^r \colon H^d(BO(\infty); \mathbb{F}_2) \to H^{d-r}(BO(\infty); \mathbb{F}_2) $ be the dual of the $ r $-th order Dyer-Lashof operation in homology.
Then, for all $ r \leq k $, the following equality holds:
\[
Q_*^r(w_k) = \left( \begin{array}{c} r-1 \\ k-r-1
\end{array} \right) w_{k-r}.
\]
Moreover, if $ [n] \in H^0(BO(\infty) \times \mathbb{Z}; \mathbb{F}_2) \cong \mathbb{Z} $ corresponds to $ n \in \mathbb{Z} $, then 
\[
Q^r([1]) = (w_1^r)^\vee * [2].
\]
\end{theorem}
    
Using this theorem, Lemma \ref{lema:swformula} and Proposition \ref{indec} with $ X = BO(1) $ and the monomial basis on $ w $, we can prove the pullback formula for $ h^*(w_k) $. The proof is very similar to the complex case: we use Lemma~\ref{lema:chernformula} to reduce to primitives, and compute the primitive part of the pullback by applying Proposition~\ref{indec} and Kochman's formulae via Lemma \ref{pairing real}. For this reason, we only provide the statements and we omit proofs.

\begin{lemma} \label{pairing real}
    Let $k\ge 1$ and $\langle \cdot , \cdot \rangle \colon H^* (B\Bgroup_{2^k}; \F_2) \otimes H_* (B\Bgroup_{2^k}; \F_2) \rightarrow \F_2$ be evaluation pairing between cohomology and homology of the space $B\Bgroup_{2^k}$. Then, \[ \big \langle h_{2^k}^* (w_{2^k-1}) , Q^{2^{k}} Q^{2^{k-1}} \cdots Q^4 Q^2 (1) \big \rangle =1 , \] where $1 \in H_0 (B\Bgroup_1; \F_2)$.
\end{lemma}
We will again state the pullback formula in terms of the total Stiefel-Whitney class \[ w_* = 1+w_1 +w_2 +\cdots \in H^* (BO(\infty); \F). \] To obtain the pullback formula for a specific Stiefel class $w_k$, we only need to consider the classes of cohomological dimension $k$.

\begin{theorem} \label{pullbackformula real} (Pullback Formula)
    Let $\mathcal{X} = \{ \underline{a} = (a_0, a_1, a_2, \dots ) \in \N^{\N^*} \mid a_n = 0 \text{ for } n \gg 0 \}$ be the set of infinite sequences of natural numbers that have finite support (i.e., that are eventually zero). Then, \begin{equation} \label{pullback-real}
        h^* (w_*) = \sum_{\underline{a} \in \mathcal{X}} w_{[a_0]} \odot \big ( \bigodot_{i,a_i \neq 0} \gamma_{i, a_i} \big ) \odot 1_{\infty} .
    \end{equation} 
    Note that requiring that every sequence in $\mathcal{X}$ has finite support guarantees that the transfer product in the formula (\ref{pullback-real}) is iterated a finite number of times and well-defined.
\end{theorem}

Modulo $ 2 $, the real counterpart of Lemma \ref{lema:associated-graded} is proved with essentially the same argument.
\begin{lemma} \label{lema:associated-graded real}
le $ A = H^*(B\Bgroup_\infty; \mathbb{F}_2) $ and $ \mathcal{F}_* $ be the rank filtration. The associated graded algebra is isomorphic to the polynomial algebra
\[
\gr_{\mathcal{F}}(A) = \mathcal{F}_0[w_{[1]} \odot 1_\infty, w_{[2]} \odot 1_\infty, \dots, w_{[k]} \odot 1_\infty, \dots].
\]
\end{lemma}

Consequently, the description of the mod $ 2 $ stable cohomology of unordered real flag manifold is similar to the complex one.
\begin{theorem} \label{theo:regular real}
Let $ h^* \colon H^*(BO(\infty); \mathbb{F}_2) \to H^*(B\Bgroup_\infty; \mathbb{F}_2) $ be the map induced by $ h $ in cohomology. Then the sequence $ \{h^*(w_1),\dots, h^*(w_k),\dots \} $ is a regular sequence in $ H^*(B\Bgroup_\infty; \F_2) $.
\end{theorem}
\begin{proof}
The proof is analogous to that of Theorem \ref{theo:regular}.
\end{proof}

\begin{corollary}
There is an isomorphism of algebras
\[
H^* (\overline{\Fl}_\infty(\mathbb{R}); \mathbb{F}_2) \cong \frac{A_\infty(BO(1))}{(h^*(w_1),h^*(w_2),\dots)}.
\]
\end{corollary}
\begin{proof}
We compare the Serre spectral sequences associated with the fiber sequences
\begin{gather*}
SO(\infty) \longrightarrow \frac{SO(\infty)}{\Bgrouppos{\infty}} \longrightarrow B\Bgrouppos{\infty} \\
\tag*{and} SO(\infty) \longrightarrow ESO(\infty) \longrightarrow BSO(\infty).
\end{gather*}
Recall that $ H^*(SO(\infty); \mathbb{F}_2) $ is the exterior algebra generated by classes $ a_{i-1} $ for $ i = 2,3,\dots $ and $ a_{i-1} $ transgresses to $ w_i $ in the spectral sequence associated to the bottom fibration. Therefore, $ d_i(a_{i-1}) = g^*(w_i) $ in the Serre spectral sequence associated to the top fibration. By Corollary \ref{cor:homological stability alternating subgroups}, the mod $ 2 $ cohomology of $ \Bgrouppos{\infty} $ is isomorphic to $ A_\infty(BO(1))/(h^*(w_1)) $. As $ \res_\infty \circ h = g $, under this isomorphism $ g^*(w_k) $ is identified with the image of $ h^*(w_k) $ in the quotient.
By Theorem \ref{theo:regular real} $ \{h^*(w_1), h^*(w_2), \dots \} $ is a regular sequence, and hence we must have for all $ i \ge 1 $, $ d_i $ injective on the ideal generated by $a_{i-1}$.
\end{proof}

\begin{corollary}
$ H^*(\overline{\Fl}_\infty(\mathbb{R}); \mathbb{F}_2) $ is the free commutative algebra generated by the stabilization of decorated gathered blocks $ b \odot 1_\infty $ in $ A_\infty(BO(1)) $ such that $ \rk(b) = 0 $.
\end{corollary}

\subsubsection{Cohomology at odd primes}

With coefficients in $ \F_p $, $ p > 2 $, the cohomology of $ \mathbb{P}^\infty(\mathbb{R}) $ is trivial. Thus, every Hopf monomial in $ A_{\mathbb{P}^\infty(\mathbb{R})} $ has rank $ 0 $ and the rank filtration becomes useless.
Consequently, we must fully exploit the mod $ p $ pullback formula to perform our stable computations.

We recall that there are Pontrjagin classes $ \mathcal{p}_k \in H^{4k}(BSO(\infty); \mathbb{Z}) $ that are pullbacks of even Chern classes via the complexification map $ cpx \colon BO(\infty) \to BU(\infty) $.
We state our pullback formula for these by means of the total Pontrjagin class
\[
\mathcal{p}_* = \sum_{n=0}^\infty \mathcal{p}_n = 1 + \mathcal{p}_1 + \mathcal{p}_2 + \dots \in H^*(BO(\infty)),
\]
where we let, by convention, $ \mathcal{p}_0 = 1 $.

\begin{proposition} \label{pullback pontrjagin}
Let $\mathcal{Y} = \{ \underline{a} = (a_1, a_2, \dots ) \in \N^{\N^*} \mid a_n = 0 \text{ for } n \gg 0 \}$ be the set of infinite sequences of natural numbers that have finite support (i.e., that are eventually zero). Then, \begin{equation} \label{pullback-pon}
    g^* (\mathcal{p}_*) = \sum_{\underline{a} \in \mathcal{Y}} \big ( \bigodot_{i,a_i \neq 0} \gamma_{i, a_i} \big ) \odot 1_{\infty} .
\end{equation} 
Note that requiring that every sequence in $\mathcal{Y}$ has finite support guarantees that the transfer product in the formula (\ref{pullback-pon}) is iterated a finite number of times and well-defined.
\end{proposition}
\begin{proof}
We first prove that the desired identity holds for $ h^*(\mathcal{p}_*) \in H^*(B\Bgroup_\infty; \mathbb{F}_p) $.
There is a commutative diagram
\begin{center}
\begin{tikzcd}
D_\infty(\{*\}_+) \arrow{r}{j_1} \arrow{d}{j_2} & D_\infty(BO(1)_+) \arrow{r}{h} & BO(\infty) \arrow{d}{cpx} \\
BN(\infty) \arrow{rr}{f} & & BU(\infty)
\end{tikzcd}
\end{center}
where $j_1$ and $j_2$ are the obvious inclusions.
By passing to cohomology we deduce that
\[
j_1^* h^*(\mathcal{p}_k) = j_1^*h^* cpx^*(c_{2k}) = j_2^*f^*(c_k).
\]
By functoriality, $ j_2^* $ is the stabilization of the Hopf ring morphism $ A_{BU(1)} \to A_{\{*\}} $ induced by the inclusion $ \{*\} \to BU(1) $. In particular, $ j_2^*(\gamma_{k,l}) = \gamma_{k,l} $ and $ j_2^*(c_{[l]}) = 0 $ for all $ k,l \geq 1 $.
As $ j_1^* $ is an isomorphism in mod $ p $ cohomology, the result follows from Theorem \ref{pullbackformula}.

To pass from $ h^*(\mathcal{p}_*) $ to $ g^*(\mathcal{p}_*) $, it is enough to recall that the restriction map $ H^*(\Bgroup_\infty; \mathbb{F}_p) \to H^*(\Bgrouppos{\infty}; \mathbb{F}_p) $ is an isomoprhism by Corollary \ref{cor:res mod p}.
\end{proof}

\begin{corollary} \label{cor:regular real odd}
Let $ p > 2 $ be an odd prime. Then, in mod $ p $ cohomology, the following statements hold:
\begin{enumerate}
\item $ g^*(\mathcal{p}_k) = 0 $ if $ k $ is not a multiple of $ (p-1)/2 $
\item $ \{g^*(\mathcal{p}_{\frac{p-1}{2}}),g^*(\mathcal{p}_{2\frac{p-1}{2}}), g^*(\mathcal{p}_{3\frac{p-1}{2}}),\dots\} $ is a regular sequence in $ H^*(B\Bgrouppos{\infty}; \mathbb{F}_p) $.
\end{enumerate}
\end{corollary}
\begin{proof}
Statement $ 1 $ is immediate from Lemma \ref{pullback pontrjagin}, because the right-hand side of the formula for $ g^*(\mathcal{p}_*) $ only have terms in degrees that are multiple of $ 2(p-1) $.

To prove the second claim, we first define an alternative rank function on stabilized Hopf monomials in $ H^*(B\Bgrouppos{\infty}; \mathbb{F}_p) $, which is identified with $ H^*(B(\Sigma_\infty); \mathbb{F}_p) = A_\infty(\{*\}) $ by Theorem \ref{thm:cohomology alternating group mod p}.
Given a gathered block $ b $, we define
\[
\rk(b) = \left\{ \begin{array}{ll}
nm & \mbox{if } b = \gamma_{1,n}^m \\
0 & \mbox{if } b \not= \gamma_{1,n}^m \forall n,m \in \mathbb{N}
\end{array} \right. .
\]
Given a Hopf monomial $ x = b_1 \odot \dots \odot b_r $, we define $ \rk(x) = \sum_{i=1}^r \rk(b_i) $. Finally, for all stabilized Hopf monomials $ x \odot 1_\infty $, with $ x $ pure, we let $ \rk(x \odot 1_\infty) = \rk(x) $.
This produces an increasing filtration $ \mathcal{F} = \{ \mathcal{F}_n \}_{n \geq 0} $ such that $ \mathcal{F}_n $ is the subspace linearly spanned by basis element of rank at most $ n $.

The filtration $ \mathcal{F}_n $ is multiplicative, and the same arguments used in our proof of Lemma \ref{lema:associated-graded} shows that there is an isomorphism
\[
\gr_{\mathcal{F}}(A_\infty(\{*\})) \cong \mathcal{F}_0 [\gamma_{1,1} \odot 1_\infty, \dots, \gamma_{1,n} \odot 1_\infty,\dots].
\]
As an immediate consequence of Proposition \ref{pullback pontrjagin}, $ g^*(\mathcal{p}_{k \frac{p-1}{2}}) = \gamma_{1,k} \odot 1_\infty $ modulo $ \mathcal{F}_{k-1} $
and that the sequence $ (\gamma_{1,1} \odot 1_\infty, \dots, \gamma_{1,n} \odot 1_\infty,\dots) $ is a regular sequence in $ \gr_{\mathcal{F}}(A_\infty(\{*\})) $ by the same argument used in the proof of Theorem \ref{theo:regular}.
\end{proof}

\begin{corollary}
If $ p > 2 $ is an odd prime, the quotient algebra with mod $ p $ coefficients
\[
\frac{A_\infty(\{*\})}{(g^*(\mathcal{p}_1),g^*(\mathcal{p}_2),\dots)}
\]
is the free graded commutative algebra generated by stabilized gathered blocks $ b \odot 1_\infty $ satisfying the conditions of Corollary \ref{cor:polynomial generators} such that $ \rk(b) = 0 $.
\end{corollary}

Finally, we are ready to complete the computation of the stable mod $ p $ cohomology of $ O(\infty)/\Bgroup_{\infty} = SO(\infty)/\Bgrouppos{\infty}$.
\begin{theorem} \label{thm:stable cohomology real mod p}
Let $ p > 2 $ be an odd prime. Then there is an isomorphism of $ \F_p $-algebras
\[
H^* \left( \frac{O(\infty)}{\Bgroup_{\infty}}; \F_p \right) \cong \Lambda(\{a_{4n-1}: \frac{p-1}{2} \not| n\}) \otimes \frac{A_\infty(\{*\})}{(g^*(\mathcal{p}_1),g^*(\mathcal{p}_2),\dots)},
\]
where $ \Lambda $ is the exterior algebra functor, and the classes $ a_k $ are indexed by their degree.
\end{theorem}
\begin{proof}
We consider the cohomological Serre spectral sequence associated with the fiber sequence $ SO(\infty) \to \overline{\Fl}_\infty(\mathbb{R}) \to B\Bgrouppos{\infty} $. Recall that $ H^*(SO(\infty); \F_p) $ is the exterior algebra generated by classes $ a_3,a_7,\dots, a_{4k-1},\dots $. Note that the action of the fundamental group of the base on the fiber is homotopically trivial, so the $ E_2^{*,*} $ page involves cohomology with constant coefficients.
By comparing it with the fibration $ SO(\infty) \to ESO(\infty) \to BSO(\infty) $ we deduce, as in the complex case, that $ a_{4k-1} $ is transgressive and transgresses to $ g^*(\mathcal{p}_{k}) $.
This completely determines the differentials on each page.

In the cohomological Serre spectral sequence of the bottom fibration, the class $ a_{4n-1} \in H^*(SO(\infty)) $ is transgressive and transgresses to the Pontryagin class $ \mathcal{p}_n $.
By comparing the Serre spectral sequences of the two fibrations above, we see that, for the top fibration, $ a_{4n-1} $ is also transgressive and transgresses to $ g^*(\mathcal{p}_n) $. This spectral sequence is multiplicative with respect to the cup product, hence the remark above completely determines all its differentials.

By Corollary \ref{cor:regular real odd}, the pullbacks of Pontryagin classes $ g^*(\mathcal{p}_n) $ for $ \frac{p-1}{2} | n $ form a regular sequence in the cohomology of the base and that the differentials of the remaining classes are $ 0 $.
Consequently, one can check inductively that the classes $ a_{4n-1} $ with $ \frac{p-1}{2} \not| n $ survive to the limit page and the differential $ d_{4n} $ with $ \frac{p-1}{2} | n $ is injective on the ideal generated by $ a_{4n-1} $.
Therefore
\[
E_\infty^{*,*} \cong \Lambda(\{a_{4n-1}: \frac{p-1}{2} \not| n\}) \otimes \frac{A_\infty(\{*\})}{(g^*(p_1),g^*(p_2),\dots)}.
\]

Since we are using field coefficients, $ H^*(\overline{\Fl}_\infty(\mathbb{R}); \F_p) \cong E_\infty^{*,*} $ as a graded vector space. A priori, this isomorphic might not preserve the product. In this particular case, however, the generators $ a_{4k-1} $, being odd-dimensional, must actually span an exterior algebra inside $ H^*(\overline{\Fl}_\infty(\mathbb{R}); \mathbb{F}_p) $, and thus this is an isomorphism of algebras.
\end{proof}

\subsubsection{Rational cohomology}

The rational cohomology of unordered flag manifolds is known. However, we stress that we can retrieve it as a simple consequence of our machinery.

\begin{proposition}
\[
H^*(\overline{\Fl}_\infty(\mathbb{R}); \mathbb{Q}) \cong H^*(SO(\infty); \mathbb{Q}) \cong \Lambda(\{a_{4k-1}:k\in \mathbb{N}\}).
\]
\end{proposition}
\begin{proof}
The rational cohomology of the finite group $ \Bgrouppos{\infty} $ is trivial by Serre's theorem and the action of $ \Bgrouppos{\infty} $ on $ SO(\infty) $ is homotopicallly trivial.
Hence, the rational Serre spectral sequence of the fiber sequence $ SO(\infty) \to \overline{\Fl}_\infty(\mathbb{R}) \to B\Bgrouppos{\infty} $ collapses at the second page, which reduces to the cohomology of the fiber.
\end{proof}

\subsection{Poincar\'e series of stable cohomology}

We now determine the Poincar\'e series formula for stable cohomology, specifically $H^* (\UFl_{\infty} (\C) ; \F_2)$. For this, we need to obtain the Poincar\'e series of the stabilization $D_{\infty} (BU(1)_+) \cong BN(\infty)$.

\subsubsection{Poincar\'e series of $H^* (BN(\infty); \F_2)$} 
To obtain $H^* (BN(\infty); \F_2)$, we take the stabilization of full-width Hopf monomials in $\bigoplus_{n\ge 0} H^* (D_n (BU(1)_+); \F_2 )$. This means we exclude Hopf monomials that have a $\odot$-factor of $1_m$ in their expression. Let us denote the bigraded Poincar\'e series of $\bigoplus_{n\ge 0} H^* (D_n (BU(1)_+); \F_2 )$ by $\Pi (t,s)$. To obtain the Poincar\'e series $\Pi_{stab} (t)$ of $H^* (BN(\infty); \F_2)$, we need to \begin{itemize}
    \item forget the second degree (component), 
    \item keep only terms corresponding to full-width Hopf monomials.
\end{itemize} 
Therefore $\Pi_{stab} (t) = [(1-s) \cdot \Pi (t,s)] \mid_{s=1}$.

\subsubsection{Poincar\'e series of $H^* (\UFl_{\infty} (\C) ; \F_2)$} 
We know from Theorem~\ref{theo:regular} that the sequence $\big (f^* (c_1), f^* (c_2), \dots \big )$ is regular in $H^* ( BN (\infty); \F_2 ) $ and from Corollary~\ref{cor:stable-cohomology} \[ H^* ( \UFl_{\infty} (\C) ; \F_2 ) \cong \frac{H^* (BN (\infty); \F_2)}{\big (f^* (c_1), f^* (c_2), \dots \big )} . \] From the additivity of Poincar\'e series and the short exact sequence of the graded vector spaces \begin{align*}
    0 \longrightarrow \frac{H^* (BN (\infty); \F_2)}{\big (f^* (c_1),\dots, f^* (c_n) \big )} [2(n+1)] &\longrightarrow \frac{H^* (BN (\infty); \F_2)}{\big ( f^* (c_1), \dots, f^* (c_n) \big )} \\ &\longrightarrow \frac{H^* (BN (\infty); \F_2)}{\big 
( f^* (c_1),\dots , f^* (c_{n+1}) \big )} \longrightarrow 0 
\end{align*}  one can inductively compute the Poincar\'e series of $\frac{H^* (BN (\infty); \F_2)}{\big ( f^* (c_1), \dots, f^* (c_n) \big )}$ and pass to the limit to get the Poincar\'e series of $H^* (\UFl_{\infty} (\C); \F_2)$: \[  \Pi_{\frac{U(\infty)}{N(\infty)}} (t) = \Pi_{stab} (t) \cdot \big ( 1 - \sum_{n=1}^{\infty} t^{2n} \big ) . \] 
It still remains to compute $\Pi (t,s)$. This can be done by first determining the bigraded Poincar\'e series of the primitive component and subsequently, by using the Hopf algebra structure, extending the result to obtain the Poincar\'e series $\Pi(t, s)$.

\begin{remark}
The homological stability of unordered complete flag manifolds is linked to the representation stability of their ordered version. 
\end{remark}

%%% Local Variables:
%%% TeX-master: "Unordered flag varieties.tex"
%%% End:

%% file: Section5.tex
% !TEX root = Unordered flag varieties.tex

\section{The spectral sequence computing \texorpdfstring{$ H^*(\overline{\Fl}_n(\mathbb{C});\F) $}{UFl-C} and \texorpdfstring{$ H^*(\overline{\Fl}_n(\mathbb{R});\F) $}{UFl-R}}

\subsection{The complex case}

In this section, we prepare the ground for the general spectral sequence argument that we will use to compute $ H^*(\UFl_n (\C); \F_p)$. For $ G = N(n) $ and $ X = U(n) $ in Theorem~\ref{generalSSS}, we obtain the fiber sequence \begin{equation} \label{spseq}
    U(n) \longrightarrow \UFl_n (\C) \longrightarrow BN(n) 
\end{equation} The $E_2$-page of the spectral sequence associated with (\ref{spseq}) is given by \[ H^* (BN(n); H^* (U(n); \F_p)) \cong H^* (D_n (BU(1)_+); \F_p) \otimes H^* (U(n); \F_p). \]

Recall from \cite[Proposition~3D.4]{Hatcher} that \begin{equation} \label{eq:unitary}
    H^* (U(n); \F_p) \cong \Lambda_{\F_p} [ z_{2k-1} \mid k=1,2,\dots , n],
\end{equation}  where $z_{2k-1} \in H^{2k-1} (U(n); \F_p)$. We will express the differentials in the spectral sequence associated with (\ref{spseq}) using the pullbacks of the Chern classes $c_k $ under the map $f_n \colon D_n (BU(1)_+) \rightarrow BU(n) $.
\begin{lemma} \label{differential}
    In the spectral sequence associated with the fiber sequence (\ref{spseq}) the $z_{2k-1}$ transgresses to $f_n^* (c_k)$.
\end{lemma}
\begin{proof}
    Recall that a map between the fiber sequences induces a map between the corresponding spectral sequences which commute with the differentials. The lemma follows from this by comparing the Serre spectral sequence associated with the fiber sequence $U(n) \rightarrow \UFl_n (\C) \rightarrow BN(n)$ and the universal fibration $U(n) \rightarrow EU(n) \rightarrow BU(n)$.
    % Consider the universal fibration $U(n) \rightarrow EU(n) \rightarrow BU(n)$ and the following diagram\[ 
    % \begin{tikzcd} 
    % U(n) \ar[equal]{d} \ar[r] & \UFl_n (\C) \ar[d] \ar[r] & BN(n) \ar[d, "f_n"] \\
    % U(n) \ar[r] & EU(n) \ar[r] & BU(n) 
    % \end{tikzcd}\]
    % The map between the fiber sequences induces a map between the corresponding spectral sequences which commute with the differentials. If we denote the spectral sequence associated with (\ref{spseq}) by $E$ and the one associated with the universal fibration by $E'$, then we have the following diagram \[ \begin{tikzcd}   
    % E_r^{p,q} \ar[r, "d_r"] \ar[d, "f_n^*"] & E_r^{p+r, q-r+1} \ar[d, "f_n^*"] \\
    % {E'}_{r}^{p,q} \ar[r, "d'_r"] & {E'}_{r}^{p+r, q-r+1}
    % \end{tikzcd}\] 
    % that commutes. Recall that in the spectral sequence $E'$, the $z_{2k-1}$ transgresses to the Chern classes $c_k$. Therefore, by the above commutative diagram the $z_{2k-1}$ transgresses to $f_n^* (c_k)$ in the spectral sequence $E$. 
\end{proof}

The pullback formulas in Theorem~\ref{pullbackformula} have a limitation in that they are expressed using the total Chern class $c_*$ and hence do not immediately give us $f_n^* (c_k)$ for specific $n$ and $k$, which is needed for the unstable calculations. However, this issue can be addressed by imposing certain constraints on the sequences that contribute to the sum for $f_n^*(c_k)$ in the original formula. The following proposition outlines the conditions that need to be met.
\begin{proposition} \label{prop:finite-pullback}
    The following statements are true:
    \begin{enumerate}
        \item Let $\mathcal{X}_n$ denote the set of sequences $\underline{a} = (a_0, a_1, a_2 , \dots) \in \mathcal{X}$ such that $\sum_{i \ge 0} a_i 2^i \le n$ and $\mathcal{X}_{n,k}$ denote the set of sequences $\underline{a} \in \mathcal{X}_n$ such that $a_0 + \sum_{i\ge 1} a_i (2^i -1) = k$. Then, in mod $ 2 $ cohomology, \[ f_n^* (c_k) = \sum_{\underline{a} \in \mathcal{X}_{n,k}} c_{[a_0]} \odot \big ( \bigodot_{i,a_i \neq 0} \gamma_{i, a_i}^2 \big ) \odot 1_{m} , \] where $m \in \N$ is such that the addend is in component $n$. 
        \item Let $ p $ be an odd prime. Let $ \mathcal{Y}_n $ denote the set of sequences $\underline{a} = (a_0, a_1 , a_2 , \dots) \in \mathcal{Y}$ such that $\sum_{i \ge 0} a_i p^i \le n$ and $\mathcal{Y}_{n,k}$ denote the set of sequences $\underline{a} \in \mathcal{Y}_n$ such that $a_0 + \sum_{i\ge 1} a_i (p^i -1) = k$.
        Then, in mod $ p $ cohomology, \[
        f_n^*(c_k) = \sum_{\underline{a} \in \mathcal{Y}_{n,k}} c_{[a_0]} \odot \big ( \bigodot_{i, a_i \neq 0} \gamma_{i,a_i} \big ) \odot 1_{m},
        \]
        where $ m \in \N $ is such that the addend is in component $ n $.
    \end{enumerate}
\end{proposition}
\begin{proof}
    The proof is straightforward. Recall that the component to which $\gamma_{i, a_i}^2$ belongs is $a_i 2^i$ and the component to which $c_{[a_0]}$ belongs is $a_0$. Therefore, the component to which $c_{[a_0]} \odot \big ( \bigodot_{i,a_i \neq 0} \gamma_{i, a_i}^2 \big )$ belongs is $a_0+ \sum_{i\ge 1} a_i 2^i = \sum_{i \ge 0} a_i 2^i$. For the addend to be in $\mathrm{Im} (f_n^* ) \subset H^* (BN(n); \F_2)$, this value must be less than or equal to $n$. Moreover the cohomological dimension of $\gamma_{i, a_i}^2$ is $2a_i (2^i -1) $ and the cohomological dimension of $c_{[a_0]}$ is $2a_0$. Note that the pullback $f_n^*$ preserves the cohomological dimension. So, a sequence $\underline{a} \in \mathcal{X}_n$ is an addend in $f_n^* (c_k)$ if and only if $2a_0 + \sum_{i\ge 1} 2a_i (2^i -1) = 2k$.  
\end{proof}

The following example demonstrates the use of this pullback formula for $n=5$ in mod $2$ cohomology.

\begin{example} \label{n5}
    First, we compute $\mathcal{X}_{5,k}$ for $1 \le k \le 5$. \begin{align*}
        \mathcal{X}_{5,1} &= \{ (1,0,0,\dots) , (0,1,0,\dots ) \} , \\
        \mathcal{X}_{5,2} &= \{ (2,0,0,\dots), (1,1,0,\dots) , (0,2,0,\dots) \} , \\
        \mathcal{X}_{5,3} &= \{ (3,0,0,\dots), (2,1,0,\dots), (1,2,0,\dots), (0,0,1,0,\dots) \} , \\
        \mathcal{X}_{5,4} &= \{ (4,0,0,\dots), (3,1,0,\dots), (1,0,1,0,\dots) \} , \\
        \mathcal{X}_{5,5} &= \{ (5,0,0,\dots) \}.
    \end{align*}
    Hence by Proposition~\ref{prop:finite-pullback} we have \begin{align*}
        f_5^* (c_1) &= c \odot 1_4 + \gamma_{1,1}^2 \odot 1_3, \\
        f_5^* (c_2) &= c_{[2]} \odot 1_3 + c \odot \gamma_{1,1}^2 \odot 1_2 + \gamma_{1,2}^2 \odot 1_1 ,\\
        f_5^* (c_3) &= c_{[3]} \odot 1_2 + c_{[2]} \odot \gamma_{1,1}^2 \odot 1_1 + c\odot \gamma_{1,2}^2 + \gamma_{2,1}^2 \odot 1_1, \\
        f_5^* (c_4) &= c_{[4]} \odot 1_1 + c_{[3]} \odot \gamma_{1,1}^2 + c \odot \gamma_{2,1}^2 , \\
        f_5^* (c_5) &= c_{[5]}.  
    \end{align*}
\end{example}

As $\UFl_1 \cong \{*\}$ is just a point, $H^* (\UFl_1 ; \F_2) \cong \F_2$ is trivial. For $n=2$, we have an isomorphism $\UFl_2 (\C) \cong \mathbb{P}^{2} (\mathbb{R})$. For $n \ge 3$, we will be using the Serre spectral sequence $E$ associated with the fiber sequence $U(n) \rightarrow \UFl_n (\C) \rightarrow BN(n)$ to determine the cohomology of $\UFl_n (\C)$. The $E_2$-page is given by $E_2^{k,l} = H^k (BN(n); H^l (U(n); \F_2)) $. Recall the transgressive differentials $d_{2k}$ described in Lemma~\ref{differential} as follows \[ d_{2k} : z_{2k-1} \longmapsto f_n^* (c_k) .\] 
The differential $d_{2k}$ can be fully described from the multiplicative structure of the spectral sequence. Also, note that odd differentials are zero. The pullbacks $f_n^* (c_k)$ are computed using the formula from Proposition~\ref{prop:finite-pullback}. 

On the $E_2$-page of the spectral sequence, the $l=0$ row is given by $H^* (BN(n); \F_2)$ and the higher rows are $H^* (BN(n); \F_2)$ multiplied by the product of some finite combination of the $z_{2k-1}$'s. A major step in the computation is checking whether the differentials $d_{2k}$ are injective or not, i.e. whether $d_{2k} (b\cdot z_{2k-1})$ is zero or not for $b$ a generator of $H^* (BN(n); \F_2)$. 
% From our previous discussions in \S \ref{section:3.3}, the generators of the cohomology ring $H^* (BN(n); \F_2)$ are of the form $ b_1 \odot \cdots \odot b_r$ where the sum of the widths of $b_i$ (in the skyline diagram) equals $n$ and  each $b_i$ is either $c_{[m]}$, $\gamma_{k,l}$, $c_{[m]} \gamma_{k,l}$ for some suitable $m,k,l$ or a unit class $1_{j}$ for some $j \le n$.

\begin{lemma} \label{cdiff}
    In the spectral sequence $E$, $d_{2k} ((c_{[l]} \odot 1_{n-l}) \cdot z_{2k-1}) $ is non zero for all $1 \le k,l \le n$.
\end{lemma}
\begin{proof}
    From the multiplicative structure of the spectral sequence $E$, we have \[  d_{2k} ((c_{[l]} \odot 1_{n-l}) \cdot z_{2k-1}) = (c_{[l]} \odot 1_{n-l}) \cdot d_{2k} (z_{2k-1}) = (c_{[l]} \odot 1_{n-l}) \cdot f_n^* (c_k).  \] 
    Recall from Proposition~\ref{prop:finite-pullback} that $f_n^* (c_k) = c_{[k]} \odot 1_{n-k} + \cdots$ in $H^* (BN(n); \F_2)$. Note that this also holds in \[ \frac{H^* (BN(n); \F_2)}{ (f_n^* (c_1), \dots, f_n^* (c_{k-1}))} .\] Without loss of any generality we assume $k \ge l$. The other case is similar. Using Hopf ring distributivity,  \[ (c_{[l]} \odot 1_{n-l}) \cdot (c_{[k]} \odot 1_{n-k})  = c_{[l]}^2 \odot c_{[k-l]} \odot 1_{n-k} + \cdots \] which is always non-zero. We illustrate this using a skyline diagram with $k=4$, $l=3$, and $n=6$ in Figure~\ref{fig:cdot}. \begin{figure}
        \centering
        \begin{tikzpicture}
        \draw[]{(-4,0.8) rectangle (-2.8,0.0)};
        \draw[]{(-2.8,0.0) -- (-1.6,0.0)};
        \draw[dashed]{(-3.6,0.8) -- (-3.6,0.0)}; 
        \draw[dashed]{(-3.2,0.8) -- (-3.2,0.0)};
        \node[color=black] (A) at (-1.3,0.6) {$\cdot$};
        \draw[]{(-1.0, 0.8) rectangle (0.6, 0.0)};
        \draw[]{(0.6,0) -- (1.4,0)};
        \draw[dashed]{(-0.6,0.8) -- (-0.6,0.0)};
        \draw[dashed]{(-0.2,0.8) -- (-0.2,0.0)}; 
        \draw[dashed]{(0.2,0.8) -- (0.2,0.0)}; 
        \node[color=black] (B) at (1.8, 0.6) {$=$};
        \draw[]{(2.2,1.6) rectangle (3.4,0.0)};
        \draw[]{(3.4,0.8) rectangle (3.8,0.0)};
        \draw[]{(3.8,0) -- (4.6,0)};
        \draw[]{(2.2,0.8) -- (3.4,0.8)};
        \draw[dashed]{(2.6,1.6) -- (2.6,0.0)}; 
        \draw[dashed]{(3.0,1.6) -- (3.0,0.0)};
        \node[color=black] (C) at (5.2, 0.6) {$+$};
        \node[color=black] (G) at (5.9,0.6) {$\cdots$};
        \node[color=black] (D) at (-3.0, -0.6) {$(c_{[3]} \odot 1_3)$};
        \node[color=black] (E) at (0.0, -0.6) {$(c_{[4]} \odot 1_2)$};
        \node[color=black] (F) at (3.4, -0.6) {$(c_{[3]}^2 \odot c \odot 1_2)$};
    \end{tikzpicture}
        \caption{Skyline diagram for the dot product of $c_{[3]} \odot 1_3$ and $c_{[4]} \odot 1_2$}
        \label{fig:cdot}
    \end{figure}
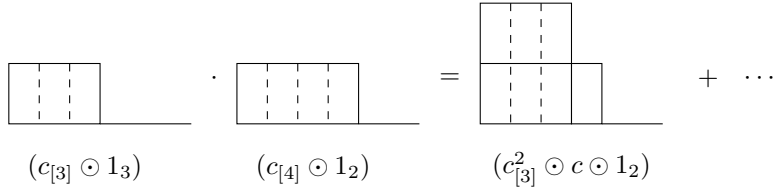
Therefore $(c_{[l]} \odot 1_{n-l}) \cdot f_n^* (c_k) \neq 0$ for all $1 \le k,l \le n$.
\end{proof}

Due to the previous lemma, the only classes that can be mapped to zero by the differentials $d_{2k}$ are the rank zero decorated gathered blocks. Here the rank function is as defined in Definition~\ref{dfn:rank}. Hence, we only need to consider these classes when we are checking the injectivity of the differentials. This simplifies our computations in \S \ref{section:6}.

\subsection{The real case}

A similar machinery can be used to compute $ H^*(\overline{\Fl}_n; \F_p)$ for all $ n \in \mathbb{N} $ and for all primes $ p $.

Recall from \cite[Corollary~4D.3]{Hatcher} that
\[
H^*(SO(n); \mathbb{F}_2) \cong \left\{ \begin{array}{ll} \Lambda(\{a_1,\dots,a_{n-1}\}) & \mbox{if } p = 2 \\
\Lambda(\{a_3,\dots,a_{4\lfloor \frac{n}{2} \rfloor -1}\}) & \mbox{if } p > 2, n \mbox{ odd} \\
\Lambda(\{a_3,\dots,a_{4\lfloor \frac{n}{2} \rfloor -5}, a'_{n-1}\}) & \mbox{if } p > 2, n \mbox{ even}
\end{array} \right. ,
\]
where generators are indexed by their degree.

The generators of the mod $ p $ cohomology of $ SO(n) $ are transgressive in the Serre spectral sequence associated with the fiber sequence $ SO(n) \to ESO(n) \to BSO(n) $. Moreover, $ d_i(a_{i-1}) = w_i $ if $ p = 2 $, $ d_i(a_{i-1}) = \mathcal{p}_i $ if $ p > 2 $, and $ d_n(a'_{n-1}) $ is the universal Euler class $ X_n \in H^n(BSO(n); \mathbb{F}_p) $ if $ p > 2 $ and $ n $ is even.

By comparison with the Serre spectral sequence of the fibration $ SO(n) \to \overline{\Fl}_n(\mathbb{R}) \to BSO(n) $ we deduce that the generators of the cohomology of $ SO(n) $ are transgressive and transgresses to the pullback of the corresponding characteristic class via $ g_n $.

The pullbacks of the Stiefel-Whitney and Pontrjagin classes can be computed from Theorem \ref{pullbackformula real} and Proposition \ref{pullback pontrjagin} by restriction from the stable cohomology as done for Chern classes in Proposition \ref{prop:finite-pullback}.
\begin{proposition} \label{prop:finite-pullback real}
The following statements are true:
\begin{enumerate}
\item Let $\mathcal{X}_n$ denote the set of sequences $\underline{a} = (a_0, a_1, a_2 , \dots) \in \mathcal{X}$ such that $\sum_{i \ge 0} a_i 2^i \le n$ and $\mathcal{X}_{n,k}$ denote the set of sequences $\underline{a} \in \mathcal{X}_n$ such that $a_0 + \sum_{i\ge 1} a_i (2^i -1) = k$. Then, in mod $ 2 $ cohomology, \[ g_n^* (w_k) = \sum_{\underline{a} \in \mathcal{X}_{n,k}} \res_n \left( w_{[a_0]} \odot \big ( \bigodot_{i,a_i \neq 0} \gamma_{i, a_i} \big ) \odot 1_{m}\right) , \] where $m \in \N$ is such that the addend is in component $n$. 
\item Let $ p $ be an odd prime. Let $ \mathcal{Y}_n $ denote the set of sequences $\underline{a} = (a_1, a_2 , \dots) \in \mathcal{Y}$ such that $\sum_{i \ge 1} a_i p^i \le n$ and $\mathcal{Y}_{n,k}$ denote the set of sequences $\underline{a} \in \mathcal{Y}_n$ such that $\sum_{i\ge 1} a_i (p^i -1) = 2k$.
Then, in mod $ p $ cohomology, \[
g_n^*(\mathcal{p}_k) = \sum_{\underline{a} \in \mathcal{Y}_{n,k}} \big ( \bigodot_{i, a_i \not= 0} \gamma_{i,a_i} \big ) \odot 1_{m},
\]
where $ m \in \N $ is such that the addend is in component $ n $.
\end{enumerate}
\end{proposition}

On the contrary, the calculation of $ g_n^*(X_n) $ when $ n $ is even requires an additional analysis.

\begin{proposition} \label{prop:finite-pullback Euler}
Let $ p > 2 $ be a prime. Then $ g_n^*(X_n) = 0 $ for all $ n > 0 $ even.
\end{proposition}
\begin{proof}
We preliminarly recall that $ H^*(B\Bgrouppos{n}; \mathbb{F}_p) \cong H^* (B\Sigma_n; \mathbb{F}_p) $ by Theorem \ref{thm:cohomology alternating group mod p}. 
We prove this identity by induction on $ n $. If $ n = 2 $ (base of the induction), then $ X_2 = 0 $ because $ H^*(B\Sigma_2; \mathbb{F}_p) $ is trivial.
We now assume that $ n > 2 $. Then, by the formula for the Euler class of direct sums (see for instance \cite{Brown}), $ \Delta(X_n) = \sum_{m=0}^{\frac{n}{2}} X_{2m} \otimes X_{n-2m} $. Hence, by induction hypothesis, $ X_n $ is primitive. As $ H^*(B\Sigma_n; \mathbb{F}_p) $ does not contain any non-zero primitive, this implies $ X_n = 0 $.
\end{proof}

\begin{example} \label{n5 real}
$ \mathcal{X}_{4,3} = \{(3,0,0,\dots), (2,1,0,\dots), (0,0,1,0,\dots)\} $.
Therefore, by Proposition \ref{prop:finite-pullback real} we have
\[
g_4^*(w_3) = (w_{[3]} \odot 1_2)^0 + (w_{[2]} \odot \gamma_{1,1})^0 + \gamma_{2,1}^+ + \gamma_{2,1}^-.
\]

Similarly, from Proposition \ref{prop:finite-pullback real} and the calculations of $ \mathcal{X}_{5,k} $ of Example \ref{n5} we have
\begin{align*}
g_5^*(w_2) &= (w_{[2]} \odot 1_3)^0 + \res(w \odot \gamma_{1,1} \odot 1_2) = (w_{[2]} \odot 1_3)^0 + (\gamma_{1,1}^2 \odot 1_3)^0, \\
g_5^*(w_3) &= (w_{[3]} \odot 1_2)^0 + (w_{[2]} \odot \gamma_{1,1} \odot 1_1)^0 + \res_n(w \odot \gamma_{1,2}) + (\gamma_{2,1} \odot 1_1)^0 \\
&= (w_{[3]} \odot 1_2)^0 + (w_{[2]} \odot \gamma_{1,1} \odot 1_1)^0 + (\gamma_{1,1}^2 \odot \gamma_{1,1} \odot 1_1)^0 + (\gamma_{2,1} \odot 1_1)^0, \\
g_5^*(w_4) &= (w_{[4]} \odot 1_1)^0 + (w_{[3]} \odot \gamma_{1,1})^0 + \res_n(w\odot \gamma_{2,1}) = (w_{[4]} \odot 1_1)^0 + (w_{[3]} \odot \gamma_{1,1})^0, \\
g_5^*(w_5) &= (w_{[5]})^0.
\end{align*}
\end{example}

Propositions \ref{prop:finite-pullback real} and \ref{prop:finite-pullback Euler} and multiplicativity completely determine the differentials on the Serre spectral sequence associated to $ SO(n) \to \overline{\Fl}_n(\mathbb{R}) \to B\Bgrouppos{n} $.
With our method, we can compute them systematically, and thus retrieving in a simpler way the results about $ H^*(\overline{\Fl}_n(\mathbb{R}); \mathbb{F}_2) $ for $ n \leq 5 $ proved in \cite{G-J-M}.
In that article, the authors used a complicated geometric argument to fully determine the spectral sequences, that we can completely avoid with our pullback formulas.

%%% Local Variables:
%%% TeX-master: "Unordered flag varieties.tex"
%%% End:

%% file: Section6.tex
% !TEX root = Unordered flag varieties.tex

\section{Unstable cohomology of complete unordered flag varieties} \label{section:6}

\subsection{Mod \texorpdfstring{$2$}{2} Cohomology of \texorpdfstring{$\UFl_3 (\C)$}{UFl3}}

Recall from (\ref{eq:unitary}) that \[ H^* (U(3); \F_2) \cong \Lambda_{\F_2} [z_1, z_3, z_5].\]
From Lemma~\ref{differential} and the pullback formula \begin{align*}
    d_2 (z_1) &= c \odot 1_2 + \gamma_{1,1}^2 \odot 1_1 , \\
    d_4 (z_3) &= c_{[2]} \odot 1_1 + c \odot \gamma_{1,1}^2  , \\
    d_6 (z_5) &= c_{[3]}.
\end{align*} 
Also, from the Hopf ring structure of $A_{BU(1)}$ in Theorem~\ref{thm:cohomology DX mod 2}, \[ H^* (BN(3); \F_2) \cong \frac{\F_2 [c\odot 1_2, c_{[2]} \odot 1_1 , c_{[3]} , \gamma_{1,1} \odot 1_1 ]}{\big ( (c\odot 1_2)\cdot (c_{[2]} \odot 1_1)\cdot (\gamma_{1,1} \odot 1_1) +  c_{[3]} \cdot (\gamma_{1,1} \odot 1_1) \big )}  . \]
On the $E_2$-page the differential $d_2$ is non-zero by Lemma~\ref{cdiff} and the following: \begin{align*}
    d_2 ((\gamma_{1,1} \odot 1_1) \cdot z_1 ) &= (\gamma_{1,1} \odot 1_1) \cdot (c \odot 1_2 + \gamma_{1,1}^2 \odot 1_1) = c\odot \gamma_{1,1}^2 + \gamma_{1,1}^3 \odot 1_1 .
\end{align*}
This shows that $d_2$ is injective on the ideal generated by $z_1$. As before $E_3$-page is the same as the $E_4$-page as $d_3 \equiv 0$. Hence, the $E_4$-page is given as follows: \begin{align*}
    E_4^{*,*} & \cong \frac{H^* (BN(3); \F_2)}{\big ( f_3^* (c_1) \big )} \otimes \Lambda_{\F_2} [z_3, z_5] \\
    & \cong \frac{\F_2 [c_{[2]} \odot 1_1 , c_{[3]} , \gamma_{1,1} \odot 1_1 ]}{\big ( (c_{[2]} \odot 1_1)\cdot (\gamma_{1,1} \odot 1_1)^3 +  c_{[3]} \cdot (\gamma_{1,1} \odot 1_1) \big )} \otimes \Lambda_{\F_2} [z_3, z_5] .
\end{align*}  
On the $E_4$-page the the differential $d_4$ is also injective as \begin{align*}
    d_4 ((\gamma_{1,1} \odot 1_3) \cdot z_3 ) &= (\gamma_{1,1} \odot 1_1) \cdot (c_{[2]} \odot 1_1 + c \odot \gamma_{1,1}^2) =  c_{[2]} \gamma_{1,1} \odot 1_1 + c \odot \gamma_{1,1}^3 .
\end{align*}
The $E_6$-page is therefore given by \begin{align*}
    E_6^{*,*} & \cong \frac{H^* (BN(3); \F_2)}{\big ( f_3^* (c_1), f_3^* (c_2) \big )} \otimes \Lambda_{\F_2} [z_5] \\
    & \cong \frac{\F_2 [c_{[3]} , \gamma_{1,1} \odot 1_1 ]}{\big ( (\gamma_{1,1} \odot 1_1)^7 +  c_{[3]} \cdot (\gamma_{1,1} \odot 1_1) \big )} \otimes \Lambda_{\F_2} [z_5] .
\end{align*}
Note that in $H^* (BN(3); \F_2)/ (f_3^* (c_1))$, we have the relation $c\odot 1_2 = \gamma_{1,1}^2 \odot 1_1$. Using this identification we have \[ c\odot \gamma_{1,1}^2 = (c\odot 1_2) \cdot (\gamma_{1,1}^2 \odot 1_1) = (\gamma_{1,1} \odot 1_1)^3 \] and hence $c_{[2]} \odot 1_1 = c \odot \gamma_{1,1}^2 = (\gamma_{1,1} \odot 1_1)^3$ in $H^* (BN(3); \F_2)/ (f_3^* (c_1), f_3^* (c_2))$. On the $E_6$-page the differential $d_6$ is again injective as \[  d_6 ((\gamma_{1,1} \odot 1_1) \cdot z_5) = c_{[2]}  \gamma_{1,1} \odot c . \]
Again all higher differentials are zero, and $E_7^{*,*} \cong E_{\infty}^{*,*}$, which gives us the following: 
\begin{theorem}
    The mod $2$ cohomology ring of $\UFl_3 (\C)$ is given by 
    \[ H^* (\UFl_3 (\C); \F_2) \cong \frac{H^* (BN(3); \F_2)}{\big ( f_3^* (c_1), f_3^* (c_2), f_3^* (c_3) \big )} \cong \frac{\F_2 [\gamma_{1,1} \odot 1_1 ]}{((\gamma_{1,1} \odot 1_1)^7 )} ,  \]
    where $|\gamma_{1,1} \odot 1_1 | = 1$.
\end{theorem}

\begin{corollary}
    The Poincar\'e series of the mod $2$ cohomology ring of $\UFl_3 (\C)$ is \[ \Pi_{\UFl_3 (\C)} (t) = 1+t +t^2 + t^3 +t^4 + t^5 +t^6 . \]
\end{corollary}

Note that this agrees with the previous results obtained in Theorem~5.3 and Corollary~5.4 of \cite{G-J-M} using a different spectral sequence.

\subsection{Mod \texorpdfstring{$2$}{2} Cohomology of \texorpdfstring{$\UFl_4 (\C) $}{UFl4}}

From (\ref{eq:unitary}), $H^* (U(4); \F_2) \cong \Lambda_{\F_2} [z_1, z_3, z_5, z_7]$. Also, from Lemma~\ref{differential} and the pullback formula the differentials in the spectral sequence $E$ associated with $U(4) \rightarrow \UFl_4 (\C) \rightarrow BN(4)$ are given by \begin{align*}
    d_2 (z_1) &= c\odot 1_3 + \gamma_{1,1}^2 \odot 1_2 , \\
    d_4 (z_3) &= c_{[2]} \odot 1_2 + c\odot \gamma_{1,1}^2 \odot 1_1 + \gamma_{1,2}^2 , \\
    d_6 (z_5) &= c_{[3]} \odot 1_1 + c_{[2]} \odot \gamma_{1,1}^2 + \gamma_{2,1}^2, \\
    d_8 (z_7) &= c_{[4]}.
\end{align*} 
From the Hopf ring structure of $A_{BU(1)}$, we have that $H^* (BN(4); \F_2)$ is generated by  \begin{equation*} \label{bn4}
       \{ c\odot 1_3, c_{[2]} \odot 1_3 , c_{[3]} \odot 1_1 , c_{[4]}, c_{[2]} \odot \gamma_{1,1} , \gamma_{1,1} \odot 1_2, \gamma_{1,2}, \gamma_{2,1} \} , \end{equation*}
with all the relations in Theorem~\ref{thm:cohomology DX mod 2} and Hopf ring distributivity. By Lemma~\ref{cdiff}, the differential $d_2$ non-zero on $\{ c\odot 1_3, c_{[2]}\odot 1_2 , c_{[3]}\odot 1_1, c_{[4]} \} $. On the rank zero decorated gathered blocks $d_2$ is given by \begin{align*}
    d_2 ((\gamma_{1,1} \odot 1_2) \cdot z_1 ) &= (\gamma_{1,1} \odot 1_2) \cdot (c \odot 1_3 + \gamma_{1,1}^2 \odot 1_2) = c\odot \gamma_{1,1} \odot 1_1 + \gamma_{1,1}^3 \odot 1_2 , \\
    d_2 ((c_{[2]} \odot \gamma_{1,1} ) \cdot z_1 ) &= (c_{[2]} \odot \gamma_{1,1} ) \cdot (c \odot 1_3 + \gamma_{1,1}^2 \odot 1_2) =  c^2 \odot c \odot \gamma_{1,1} +  \cdots , \\
    d_2 (\gamma_{1,2} \cdot z_1 ) &= \gamma_{1,2} \cdot (c \odot 1_3 + \gamma_{1,1}^2 \odot 1_2) = \gamma_{1,1}^3 \odot \gamma_{1,1} , \\ 
    d_2 (\gamma_{2,1} \cdot z_1 ) &= \gamma_{2,1} \cdot (c \odot 1_3 + \gamma_{1,1}^2 \odot 1_2) = 0. 
\end{align*}
As $d_2 (\gamma_{2,1} \cdot z_1) = 0$, we have that $d_2$ is not injective on the ideal generated by $z_1$ and $\mathrm{ker} (d_2)$ is the ideal generated by $(\gamma_{2,1} \cdot z_1)$. So, the $E_3 \equiv E_4$-page is given by \begin{align*}
    & E_4^{*,*} \cong \Big ( \frac{H^* (BN(4); \F_2)}{\big ( f_4^* (c_1) \big ) } \oplus ( \langle \gamma_{2,1} \cdot z_1 \rangle ) \Big ) \otimes \Lambda_{\F_2} [z_3, z_5, z_7] ,
\end{align*}
where $\langle \gamma_{2,1} \cdot z_1 \rangle$ denotes the ideal generated by $\gamma_{2,1} \cdot z_1$. As $c\odot 1_3 = \gamma_{1,1}^2 \odot 1_2 $ on $E_4^{*,*}$, we can rewrite the differential $d_4$ as $d_4 (z_3) = c_{[2]} \odot 1_2 + \gamma_{1,1}^4 \odot 1_1 + \gamma_{1,2}^2$. The injectivity of $d_4$, $d_6$, and $d_8$ on the ideals generated by $z_3$, $z_5$, and $z_7$ respectively, can be verified by directly computing these differentials on the rank zero gathered blocks multiplied by $z_3$, $z_5$, and $z_7$ respectively.

On the first row of the spectral sequence, $\gamma_{1,2}^2 \gamma_{2,1} \cdot z_1$ is killed by $d_4 (\gamma_{2,1} \cdot z_1 z_3)$, $\gamma_{2,1}^3 \cdot z_1$ is killed by $d_6 (\gamma_{2,1} \cdot z_1 z_5)$, and $c_{[4]} \gamma_{2,1} \cdot z_1$ is killed by $d_8 (\gamma_{2,1} \cdot z_1 z_7)$ since 
\begin{align*}
    d_4 (\gamma_{2,1} \cdot z_3) &= \gamma_{2,1} \cdot (c_{[2]} \odot 1_2 + \gamma_{1,1}^4 \odot 1_1 + \gamma_{1,2}^2) =  \gamma_{1,2}^2 \gamma_{2,1}, \\
    d_6 (\gamma_{2,1} \cdot z_5) &= \gamma_{2,1} \cdot (c_{[3]} \odot 1_1 + c_{[2]} \odot \gamma_{1,1}^2 + \gamma_{2,1}^2) =  \gamma_{2,1}^3, \\
    d_8 (\gamma_{2,1} \cdot z_7) &= \gamma_{2,1} \cdot c_{[4]} =  c_{[4]} \gamma_{2,1}.
\end{align*} All other higher differentials are zero and we have
\begin{align*}
    E_{\infty}^{*,*} & \cong \frac{H^* (BN(4); \F_2)}{\big ( f_4^* (c_1) , f_4^* (c_2), f_4^* (c_3), f_4^* (c_4) \big )} \oplus \F_2 \{ \gamma_{2,1} \cdot z_1, \gamma_{1,2} \gamma_{2,1} \cdot z_1 ,\gamma_{2,1}^2 \cdot z_1,  \gamma_{1,2} \gamma_{2,1}^2 \cdot z_1 \}.
\end{align*}
We record the result of our above computation as the following Theorem.
\begin{theorem} \label{unst4}
    Let $u_4$ be the cohomology class $\gamma_{2,1} \cdot z_1$. Then, we have an isomorphism \[ H^* (\UFl_4 (\C) ; \F_2) \cong \frac{H^* (BN(4); \F_2)}{\big ( f_4^* (c_1) , f_4^* (c_2), f_4^* (c_3), f_4^* (c_4) \big )} \oplus \F_2 \{ u_4, u_4 \gamma_{1,2}, u_4 \gamma_{2,1}, u_4 \gamma_{1,2} \gamma_{2,1} \} . \]
\end{theorem}
On the $E_{\infty}$ page of the spectral sequence, the classes $u_4, u_4 \gamma_{1,2}, u_4 \gamma_{2,1}$, and $u_4 \gamma_{1,2} \gamma_{2,1}$ survive due to the non-injectivity of the differential $d_2$. However, in the corresponding spectral sequence for stable cohomology, all differentials are injective on the corresponding ideals generated by the generators $z_{2k-1}$, and therefore these four classes are not restrictions of any stable class. We refer to such classes as ``unstable classes''.

% We can find the relations between the generators of $H^* (\UFl_4 (\C); \F_2)$ by rewriting the pullback of Chern classes in the quotients and invoking the relations $R_1, R_2, R_3$ from (\ref{bn4}). Let us denote $H^* (BN(4); \F_2)$ by $A$. As, previously remarked $c \odot 1_3 = \gamma_{1,1}^2 \odot 1_2$ in $A/(f^* (c_1))$. So, from $R_1$ and $R_2$ we have \[ (\gamma_{1,1} \odot 1_2 )^2 \cdot \gamma_{1,2} = 0 \quad \text{and} \quad (\gamma_{1,1} \odot 1_2 )^2 \cdot \gamma_{2,1} = 0 \] in $A/(f^* (c_1))$. In $A/(f^* (c_1), f^* (c_2))$, we have $c_{[2]} \odot 1_2 = (\gamma_{1,1} \odot 1_2)^4 + \gamma_{1,2}^2$. Again using the relation $(c_{[2]} \odot 1_2)\cdot \gamma_{2,1}$ from $R_2$, we have $\gamma_{1,2}^2 \gamma_{2,1} = 0$ in $A/(f^* (c_1), f^* (c_2))$. Also, using the previous relations we can rewrite  $f_4^* (c_3)$ as \[ f_4^* (c_3) \] in $A/(f^* (c_1), f^* (c_2))$.  
\begin{corollary}
    The Poincar\'e series of the mod $ 2 $ cohomology ring of $\UFl_4 (\C)$ is \[ \Pi_{\UFl_4 (\C)} (t) = 1+t+2t^2 +3t^3 +4t^4 +4t^5 + 5t^6 + 4t^7 + 4t^8 + 3t^9 + 2t^{10} + t^{11} + t^{12} . \]
\end{corollary}
\begin{proof}
    We can check this directly. Recall that the cohomological dimensions of the classes that generate $H^* (BN(4); \F_2)$ are as follows: \begin{align*}
        |\gamma_{1,1} \odot 1_2| &= 1, \quad |\gamma_{1,2}| = 2, \quad |\gamma_{2,1}| = 3, \\ |c_{[2]} \odot \gamma_{1,1}| &= 5, \quad \text{and} \quad 
         |c_{[k]} \odot 1_{4-k}| = 2k \quad (1 \le k \le 4).
    \end{align*}
    We also have the following relations between the generators:\begin{align*}
    \big ( (\gamma_{1,1} \odot 1_2)\cdot (c_{[2]}+1_2) + (c_{[2]} \odot \gamma_{1,1}) \big )\cdot (c\odot 1_3) + (\gamma_{1,2}\odot 1_2)\cdot (c_{[3]} \odot 1_1) &= 0, \\
    \text{for all } x\in \{ \gamma_{1,1}\odot 1_2, c_{[2]} \odot \gamma_{1,1}, c \odot 1_{3}, c_{[2]} \odot 1_{2}, c_{[3]} \odot 1_{1}  \}, \quad  x \cdot \gamma_{2,1} &= 0 , \\
    (\gamma_{1,1}\odot 1_2) \cdot (c_{[3]} \odot 1_1) = 0, \quad  (c_{[2]} \odot \gamma_{1,1}) \cdot (c_{[3]} \odot 1_1) &= 0, \\
    \gamma_{1,2} \cdot (c \odot 1_3) = 0, \quad \gamma_{1,2} \cdot (c_{[3]} \odot 1_1) &= 0.
    \end{align*}
    Note that these relations between the generators are not exhaustive. From the description of $H^* (\UFl_4; \F_2)$ in Theorem~\ref{unst4} and the above relations, we get the following relations in $H^* (\UFl_4; \F_2)$ between the Hopf ring generators. By identifying $(\gamma_{1,1} \odot 1_2)^2 = c\odot 1_3$, we have \begin{align*}
        (\gamma_{1,1} \odot 1_2)^2 \cdot \gamma_{1,2} = 0 , \quad (\gamma_{1,1} \odot 1_2)^2 \cdot \gamma_{2,1} &= 0,\\
         (\gamma_{1,1} \odot 1_2)^3 \cdot (c_{[2]}+1_2) + (c_{[2]} \odot \gamma_{1,1})\cdot (\gamma_{1,1} \odot 1_2)^2 &+ (\gamma_{1,2}\odot 1_2)\cdot (c_{[3]} \odot 1_1) = 0.
    \end{align*}
    By identifying $c_{[2]} \odot 1_2 = c\odot \gamma_{1,1}^2 \odot 1_1 + \gamma_{1,2}^2 = (\gamma_{1,1}\odot 1_2)^4 + \gamma_{1,2}^2$ and $c_{[3]} \odot 1_1 = \gamma_{1,2}^2 + \gamma_{2,1}^2 + (c_{[2]}\odot \gamma_{1,1})\cdot (\gamma_{1,1} \odot 1_2)$ we have\begin{align*}
        \gamma_{1,2}^2 \cdot \gamma_{2,1}  = 0 ,\quad (\gamma_{1,1} \odot 1_2)^7 + (\gamma_{1,1} \odot 1_2) \cdot \gamma_{1,2}^3 = 0.
    \end{align*}
    With the information we have so far, we can explicitly write down all generators of the cohomology groups $H^k (\UFl_4; \F_2)$ for $0\le k \le 6$.
    
    \begin{align*}
        H^0 (\UFl_4; \F_2) &\cong \F_2 , \\
        H^1 (\UFl_4; \F_2) &\cong \F_2 \{ (\gamma_{1,1} \odot 1_2) \}, \\
        H^2 (\UFl_4; \F_2) &\cong \F_2 \{ (\gamma_{1,1} \odot 1_2)^2, \gamma_{1,2} \}, \\
        H^3 (\UFl_4; \F_2) &\cong \F_2 \{ (\gamma_{1,1} \odot 1_2)^3, (\gamma_{1,1} \odot 1_2)\gamma_{1,2}, \gamma_{2,1} \}, \\
        H^4 (\UFl_4; \F_2) &\cong \F_2 \{ (\gamma_{1,1} \odot 1_2)^4, (\gamma_{1,1} \odot 1_2)\gamma_{2,1}, \gamma_{1,2}^2, u_4 \}, \\
        H^5 (\UFl_4; \F_2) &\cong \F_2 \{ (\gamma_{1,1} \odot 1_2)^5, (\gamma_{1,1} \odot 1_2)\gamma_{1,2}^2, \gamma_{1,2} \gamma_{2,1}, c_{[2]} \odot \gamma_{1,1} \}, \\
        H^6 (\UFl_4; \F_2) &\cong \F_2 \{ (\gamma_{1,1} \odot 1_2)^6, (\gamma_{1,1} \odot 1_2)(c_{[2]} \odot \gamma_{1,1}), \gamma_{1,2}^3, \gamma_{2,1}^2, u_4 \gamma_{1,2} \}.
    \end{align*}
    Although we haven't identified the specific generators of $H^k(\UFl_4; \F_2)$ for $k>6$, we can still determine its Poincar{'e} series. Note that $\UFl_4$ is a $12$-dimensional manifold and that with mod $2$ coefficients, Poincar{'e} duality holds without needing the orientability assumption. Hence, the corollary follows.
\end{proof}

\subsection{Mod \texorpdfstring{$2$}{2} Cohomology of \texorpdfstring{$\UFl_5 (\C) $}{UFl5}}

From (\ref{eq:unitary}), $H^* (U(5); \F_2) \cong \Lambda_{\F_2} [z_1, z_3, z_5, z_7, z_9]$. Also, from Lemma~\ref{differential} and Example~\ref{n5} the differentials in the spectral sequence $E$ associated with $U(5) \rightarrow \UFl_5 (\C) \rightarrow BN(5)$ are given by \begin{align*}
    d_2 (z_1) &= c \odot 1_4 + \gamma_{1,1}^2 \odot 1_3, \\
    d_4 (z_3) &= c_{[2]} \odot 1_3 + c \odot \gamma_{1,1}^2 \odot 1_2 + \gamma_{1,2}^2 \odot 1_1 ,\\
    d_6 (z_5) &= c_{[3]} \odot 1_2 + c_{[2]} \odot \gamma_{1,1}^2 \odot 1_1 + c\odot \gamma_{1,2}^2 + \gamma_{2,1}^2 \odot 1_1, \\
    d_8 (z_7) &= c_{[4]} \odot 1_1 + c_{[3]} \odot \gamma_{1,1}^2 + c \odot \gamma_{2,1}^2 , \\
    d_{10} (z_9) &= c_{[5]}.
\end{align*}
From the Hopf ring structure of $A_{BU(1)}$, we deduce that $H^* (BN(5); \F_2)$ is generated by  \[ \{ c\odot 1_4 , c_{[2]} \odot 1_3 , c_{[3]} \odot 1_2 , c_{[4]} \odot 1_1 , c_{[5]} , \gamma_{1,1} \odot 1_3 , c_{[2]} \odot \gamma_{1,1} \odot 1_1 ,\gamma_{1,2} \odot 1_1 , \gamma_{2,1} \odot 1_1 \} \] with the relations determined by Theorem~\ref{thm:cohomology DX mod 2} along with Hopf ring distributivity. Again, we check the injectivity of the differentials on the rank zero Hopf monomials multiplied by the corresponding $z_{2k-1}$'s. On the $E_2$-page we have \begin{align*}
    d_2 ((\gamma_{1,1} \odot 1_3) \cdot z_1 ) &= (\gamma_{1,1} \odot 1_3) \cdot d_2 (z_1) = c\odot \gamma_{1,1} \odot 1_2 + \gamma_{1,1}^3 \odot 1_3 + \cdots , \\
    d_2 ((c_{[2]} \odot \gamma_{1,1} \odot 1_1) \cdot z_1 ) &= (c_{[2]} \odot \gamma_{1,1} \odot 1_1 ) \cdot d_2 (z_1) =  c_{[3]} \odot \gamma_{1,1} +  \cdots , \\
    d_2 ((\gamma_{1,2} \odot 1_1) \cdot z_1 ) &= (\gamma_{1,2} \odot 1_1) \cdot d_2 (z_1) = c \odot \gamma_{1,2} + \gamma_{1,1}^3 \odot \gamma_{1,1} \odot 1_1 , \\ 
    d_2 ((\gamma_{2,1} \odot 1_1) \cdot z_1 ) &= (\gamma_{2,1} \odot 1_1) \cdot (d_2 (z_1) = c\odot \gamma_{2,1} . 
\end{align*}
The above formulas along with Lemma~\ref{cdiff} show that $d_2$ is injective on the ideal generated by $z_1$. So, $E_4 \equiv E_3$ is given by \[ E_4^{*,*} \cong \frac{H^* (BN(5); \F_2)}{\big ( f_5^* (c_1) \big )}  . \] 
Similarly as before it can be checked that the differentials $d_4$, $d_6$, and $d_8$ are all injective on the ideals generated by $z_3$, $z_5$, and $z_7$ respectively. Thus the $E_{10}$-page is given by 
\[ E_{10}^{*,*} \cong \frac{H^* (BN(5); \F_2)}{\big ( f_5^* (c_1), f_5^* (c_2), f_5^* (c_3) , f_5^* (c_4) \big )}  \]
Finally, the differential $d_{10}$ on the rank zero Hopf monomial on the $E_{10}$-page is given by \begin{align*}
    d_{10} ((\gamma_{1,1} \odot 1_3) \cdot z_9 ) &= (\gamma_{1,1} \odot 1_3) \cdot c_{[5]} = c_{[2]}\gamma_{1,1} \odot c_{[3]} , \\
    d_{10} ((c_{[2]} \odot \gamma_{1,1} \odot 1_1) \cdot z_9 ) &= (c_{[2]} \odot \gamma_{1,1} \odot 1_1 ) \cdot c_{[5]} =  c_{[2]}^2 \odot c_{[2]}\gamma_{1,1} \odot c +  \cdots , \\
    d_{10} ((\gamma_{1,2} \odot 1_1) \cdot z_9 ) &= (\gamma_{1,2} \odot 1_1) \cdot c_{[5]} = c_{[4]}\gamma_{1,2} \odot c , \\ 
    d_{10} ((\gamma_{2,1} \odot 1_1) \cdot z_9 ) &= (\gamma_{2,1} \odot 1_1) \cdot c_{[5]} = c_{[4]} \gamma_{2,1} \odot c  \\ 
    &= (c_{[4]} \gamma_{2,1} \odot 1_1) \cdot (c \cdot 1_4 + \gamma_{1,1}^2 \odot 1_3) = 0 .
\end{align*}
We see that $d_{10}$ is not injective with $\mathrm{ker}(d_{10}) = \langle (\gamma_{2,1} \odot 1_1) \cdot z_9  \rangle$. Note that \begin{align*}
      d_{4} ((\gamma_{2,1} \odot 1_1) \cdot z_3 z_9) &= (\gamma_{1,2}^2 \gamma_{2,1} \odot 1_1) \cdot z_9 , \\ d_{6} ((\gamma_{2,1} \odot 1_1) \cdot z_5 z_9) &= (\gamma_{2,1}^3 \odot 1_1) \cdot z_9 , \\ d_{8} ((\gamma_{2,1} \odot 1_1) \cdot z_7 z_9) &= (c_{[4]} \gamma_{2,1} \odot 1_1) \cdot z_9 .
\end{align*}
All higher differentials are zero and therefore the spectral sequence collapses at the $E_{11}$-page. The $E_{\infty} \equiv E_{11}$-page is given by \[ E_{\infty}^{*,*} \cong \frac{H^* (BN(5); \F_2)}{\big ( f_5^* (c_1), f_5^* (c_2), f_5^* (c_3) , f_5^* (c_4), f_5^* (c_5) \big )} \oplus \mathcal{U}  \] where $\mathcal{U} = \F_2 \{ (\gamma_{2,1} \odot 1_1) \cdot z_9 , (\gamma_{1,2} \gamma_{2,1} \odot 1_1) \cdot z_9 , (\gamma_{2,1}^2 \odot 1_1) \cdot z_9 , (\gamma_{1,2} \gamma_{2,1}^2 \odot 1_1) \cdot z_9  \}$. We summarize the result of our computations in this subsection as the following theorem.
\begin{theorem} \label{unst5}
    Let us denote the cohomology class $(\gamma_{2,1} \odot 1_1) \cdot z_9$ by $u_{12}$. Then, we have an isomorphism \[ H^* (\UFl_5 (\C) ; \F_2) \cong \frac{H^* (BN(5); \F_2)}{\big ( f_5^* (c_1) ,\dots , f_5^* (c_5) \big )} \oplus \F_2 \{ u_{12}, u_{12} \gamma_{1,2}, u_{12} \gamma_{2,1}, u_{12} \gamma_{1,2} \gamma_{2,1} \} . \]
\end{theorem}

\subsection{Mod \texorpdfstring{$p$}{p} Cohomology of \texorpdfstring{$\UFl_p (\C)$}{UFl-C} for \texorpdfstring{$p>2$}{p>2}} \label{section:6.4}

% In this subsection, we present a computation for $H^* (\UFl_p (\C); \F_p)$ for all odd prime $p$.
From Theorem~\ref{thm:cohomology DX mod p}, we have that $H^* (BN(p); \F_p)$ is generated by $\gamma_{1,1}$, $\alpha_{1,1}$ and $c_{[k]} \odot 1_{p-k}$ for $k=1,2,\dots ,p$. Our strategy will be to study the Serre spectral sequence associated with (\ref{spseq}) with $n=p$.

We introduce some notations that will be helpful later in describing the cohomology of $\UFl_p (\C)$. Recall from (\ref{eq:unitary}) that, 
\begin{equation} \label{eq5}
    H^* (U(p); \F_p) = \Lambda_{\F_p} [z_1, z_3,\dots , z_{2p-1}]
\end{equation}  
\begin{definition}\label{dfn:unitary-gen}
    Let $S=\{s_1, \dots, s_k\}$ be a subset of $\{ 1,2,\dots, p\}$ such that $s_1 < \cdots <s_k$. We define $z_S$ to be the cohomology class \[ z_S := \begin{cases} z_{2s_1 -1} \cdots z_{2s_k -1} & \text{if } S\neq \emptyset \\ 1_{U(p)} & \text{if } S = \emptyset \end{cases} \]
\end{definition}

The following are immediate from Theorem~\ref{pullbackformula} and Proposition~\ref{prop:finite-pullback}.
\begin{corollary}\label{cor:diff-modp}
    Let $f_p : BN(p) \rightarrow BU(p)$ as before. Then $f_p^* (c_k) = c_{[k]} \odot 1_{p-k}$ for $k=1,\dots, p-2$, $f_p^* (c_{p-1}) = c_{[p-1]} \odot 1_1 + \gamma_{1,1}$, and $f_p^* (c_p) = c_{[p]}$.
\end{corollary}
Note that $\alpha_{1,1}$ never shows up in the pullback formulas for $f_p^*$ as it has an odd cohomological dimension whereas $f_p^* (c_k)$ are even-dimensional for all $1 \le k \le p$.

\begin{theorem}
    Let $\gamma_S := z_S \cdot \gamma_{1,1}$ and $\alpha_S := z_S \cdot \alpha_{1,1}$, where $z_S$ is as in Definition~\ref{dfn:unitary-gen}. Then,
\[ H^* (\UFl_p (\C) ; \F_p ) = \frac{\F_p [\gamma_S, \alpha_S| S \subset \{ 1,2,\dots, p-2\} ]}{(\gamma_S^2, \alpha_S^2, \gamma_S \cdot \alpha_S)} \]
\end{theorem}
\begin{proof}
    From the Hopf ring structure of $A_{BU(1)}$ in Theorem~\ref{thm:cohomology DX mod 2}, we deduce that from  \[ H^* (BN(p) ; \F_p) \cong \frac{\F_p [\gamma_{1,1}, \alpha_{1,1}, c_{[k]} \odot 1_{p-k} \mid k=1,\dots , p]}{\big ( \alpha_{1,1}^2 , \gamma_{1,1} \cdot (c_{[k]} \odot 1_{p-k}) , \alpha_{1,1} \cdot (c_{[k]} \odot 1_{p-k}) \mid k=1,\dots ,p-1 \big )} \]   
    From Lemma~\ref{differential}, the differential $d_{2k}$ is given by \[ d_{2k} : z_{2k-1} \longmapsto f_p^* (c_k) . \]
    Using the product structure on the $E_{2i}$-page of the spectral sequence, we deduce the following: \begin{align*}
    \gamma_{1,1} \cdot (c_{[k]}\odot 1_{[p-k]}) = 0 \quad &\text{and} \quad \alpha_{1,1} \cdot (c_{[k]}\odot 1_{[p-k]}) = 0 \quad \text {for all } k=1,\dots,p-1 \\  \gamma_{1,1} \cdot c_{[p]} \neq 0 \quad &\text{and} \quad \alpha_{1,1} \cdot c_{[p]} \neq 0.
    \end{align*}
    Hence, $d_{2k} (z_{2k-1}\cdot \gamma_{1,1}) = 0$ and $d_{2k} (z_{2k-1} \cdot \alpha_{1,1}) = 0$ for all $k=1,2,\dots, p-2$. Also, using the multiplicative structure of the spectral sequence and the formulas from Corollary~\ref{cor:diff-modp}, we have \begin{align*}
    d_{2(p-1)} (z_{2p-3} \cdot \gamma_{1,1}) = \gamma_{1,1}^2 &, \quad  d_{2(p-1)} (z_{2p-3} \cdot \alpha_{1,1}) = \gamma_{1,1} \cdot \alpha_{1,1} \\ d_{2p} (z_{2p-1} \cdot \gamma_{1,1}) = \gamma_{1,1} \cdot c_{[p]} &, \quad d_{2p} (z_{2p-1} \cdot  \alpha_{1,1}) = \alpha_{1,1} \cdot c_{[p]}.
    \end{align*}
    As all higher differentials $d_k$ for $k > 2p$ are zero, the spectral sequence collapses at the $E_{2p}$-page and the $E_{\infty}$-page is given by \[ E_{\infty}^{*,*} \cong \frac{H^* (BN(p) ; \F_p)}{\big ( f_p^* (c_1) ,\dots ,f_p^* (c_p) \big )} \oplus \big ( \F_p \{ \gamma_{1,1} , \alpha_{1,1} \} \otimes \Lambda_{\F_p} [z_1, \dots , z_{2p-5}] \big ) . \]  
    This proves the theorem.
\end{proof}

\begin{corollary}
    The mod $p$ Poincar{\'e} series of $\UFl_p (\C)$ is given by \[ \Pi_{\UFl_p}^p (t) = 1 + (t^{2p-3} + t^{2p-2}) \prod_{k=1}^{p-2} (1+t^{2k-1}) . \]
\end{corollary}

\subsection{Mod \texorpdfstring{$p$}{p} cohomology of \texorpdfstring{$\overline{\Fl}_p(\mathbb{R})$}{UFl-R} for \texorpdfstring{$p>2$}{p>2}}

We conclude this section by computing the mod $ p $ cohomology of $ \overline{\Fl}_p(\mathbb{R}) $ for all odd prime $ p $. From Theorem \ref{thm:cohomology alternating group mod p} and Corollary \ref{cor:res mod p}, there is an isomorphism
\[
H^*(B\Bgrouppos{p}; \mathbb{F}_p) \cong \frac{\F_p[\gamma_{1,1},\alpha_{1,1}]}{(\alpha_{1,1}^2)},
\]
where the degrees of $ \gamma_{1,1} $ and $ \alpha_{1,1} $ are $ 2p-2 $ and $ 2p-3 $, respectively.
Similarly
\[
H^*(SO(p); \mathbb{F}_p) \cong \Lambda(\{a_1,\dots,a_{2p-3}).
\]

The differential $ d_{4r} $ of the $ 4r $-th page of the associated spectral sequence maps $ a_{4r-1} $ to the pullback of the Pontrjagin class $ \mathcal{p}_r $, and all the other differentials are $ 0 $. By Proposition \ref{prop:finite-pullback real}, $ d_{2p-2}(a_{2p-3}) = \gamma_{1,1} $, and $ d_r = 0 $ if $ r \not= 2p-2 $.
We conclude that $ E_2^{*,*} \cong E_{2p-2}^{*,*} $ and that
\[
E_{\infty}^{*,*} \cong E_{2p-1}^{*,*} \cong \Lambda(\{a_3,\dots,2p-7\}) \otimes \frac{\F_p[\alpha_{1,1}]}{(\alpha_{1,1}^2)}.
\]
We deduce the following results.
\begin{theorem}
Let $ p > 2 $ be an odd prime. Then there is an isomorphism of graded algebras
\[
H^*(\overline{\Fl}_p(\mathbb{R}); \F_p) \cong \Lambda(\{a_3,\dots,a_{2p-7}\}) \otimes \frac{\mathbb{F}_p[\alpha_{1,1}]}{(\alpha_{1,1}^2)}.
\]
\end{theorem}

\begin{corollary}
The mod $ p $ Poincar\`e series of $ \overline{\Fl}_p(\mathbb{R}) $ is given by
\[
\Pi^p_{\overline{\Fl}_p(\mathbb{R})}(t) = (1+t^3)(1+t^7)\cdots (1+t^{2p-7})(1+t^{2p-3}).
\]
\end{corollary}

%%% Local Variables:
%%% TeX-master: "Unordered flag varieties.tex"
%%% End:

%% file: Unordered_flag_varieties.bbl
\begin{thebibliography}{10}

\bibitem{Adem-Gomez}
A.~Adem and J.~G{\'o}mez.
\newblock {A classifying space for commutativity in {Lie} groups}.
\newblock {\em Algebraic \& Geometric Topology}, 15(1):493 -- 535, 2015.

\bibitem{Atiyah2002}
M.~F. Atiyah.
\newblock The geometry of classical particles.
\newblock {\em Surveys in differential geometry}, 7:1--15, 2002.

\bibitem{atiyahMac}
M.~F. Atiyah and I.~G. MacDonald.
\newblock {\em Introduction to commutative algebra.}
\newblock Addison-Wesley-Longman, 1969.

\bibitem{Brown}
E.~H. Brown, Jr.
\newblock The cohomology of {$B{\rm SO}_{n}$} and {$B{\rm O}_{n}$} with integer
  coefficients.
\newblock {\em Proc. Amer. Math. Soc.}, 85(2):283--288, 1982.

\bibitem{Church-Farb}
Thomas Church and Benson Farb.
\newblock Representation theory and homological stability.
\newblock {\em Adv. Math.}, 245:250--314, 2013.

\bibitem{May-Cohen}
F.~R. Cohen, T.~J. Lada, and J.~P. May.
\newblock {\em The homology of iterated loop spaces}.
\newblock Lecture Notes in Mathematics, Vol. 533. Springer-Verlag, Berlin-New
  York, 1976.

\bibitem{Eisenbud99}
D.~Eisenbud.
\newblock {\em Commutative algebra. With a view toward algebraic geometry.}
\newblock Berlin: Springer-Verlag, 0 edition, 2 1999.

\bibitem{Fox1939}
R.~H. Fox.
\newblock On the {Lusternik--Schnirelmann} category.
\newblock {\em Annals of Mathematics}, 42(2):333--370, 1941.

\bibitem{Sinha:12}
C.~Giusti, P.~Salvatore, and D.~Sinha.
\newblock The mod-2 cohomology rings of symmetric groups.
\newblock {\em J. Topol.}, 5(1):169--198, 2012.

\bibitem{Sinha:17}
C.~Giusti and D.~Sinha.
\newblock Mod-two cohomology rings of alternating groups.
\newblock {\em J. Reine Angew. Math.}, 772:1--51, 2021.

\bibitem{Guerra:17}
L.~Guerra.
\newblock Hopf ring structure on the mod $p$ cohomology of symmetric groups.
\newblock {\em Algebr. Geom. Topol.}, 17(2):957--982, 2017.

\bibitem{Guerra:21}
L.~Guerra.
\newblock The mod 2 cohomology of the infinite families of coxeter groups of
  type $ \mathrm{B} $ and $ \mathrm{D} $ as almost hopf rings.
\newblock {\em Algebr. Geom. Topol.}, 2023.

\bibitem{Guerra-Santanil}
L.~Guerra and S.~Jana.
\newblock The mod-2 cohomology of the alternating subgroups of the coxeter
  groups of type $ \mathrm{B} $, (in preparation).

\bibitem{G-J-M}
L.~Guerra, S.~Jana, and A.~Maiti.
\newblock Mod-2 cohomology of unordered flag varieties in lower orders and
  {Auerbach} bases, 2023.

\bibitem{Guerra-Salvatore-Sinha}
L.~Guerra, P.~Salvatore, and D.~Sinha.
\newblock Cohomology rings of extended powers and free infinite loop spaces,
  2023.

\bibitem{Hatcher}
A.~Hatcher.
\newblock {\em Algebraic topology}.
\newblock Cambridge University Press, Cambridge, 2002.

\bibitem{Humphreys}
J.~E. Humphreys.
\newblock {\em Reflection groups and {C}oxeter groups}, volume~29 of {\em
  Cambridge Studies in Advanced Mathematics}.
\newblock Cambridge University Press, Cambridge, 1990.

\bibitem{Kochman}
S.~O. Kochman.
\newblock Homology of the classical groups over the {D}yer-{L}ashof algebra.
\newblock {\em Trans. Amer. Math. Soc.}, 185:83--136, 1973.

\bibitem{mcc}
J.~McCleary.
\newblock {\em A User's Guide to Spectral Sequences}.
\newblock Cambridge University Press, Cambridge, 2001.

\bibitem{Milnor-Stasheff}
J.~W. Milnor and J.~D. Stasheff.
\newblock {\em Characteristic classes}.
\newblock Annals of Mathematics Studies, No. 76. Princeton University Press,
  Princeton, N. J.; University of Tokyo Press, Tokyo, 1974.

\bibitem{Nagpal}
Rohit Nagpal.
\newblock {\em F{I}-modules and the cohomology of modular representations of
  symmetric groups}.
\newblock ProQuest LLC, Ann Arbor, MI, 2015.
\newblock Thesis (Ph.D.)--The University of Wisconsin - Madison.

\bibitem{Weber2016}
A.~Weber and M.~Wojciechowski.
\newblock On the {Pe{\l{}}czy{\'n}ski} conjecture on {Auerbach} bases.
\newblock {\em Communications in Contemporary Mathematics}, 19(06):1750016,
  2017.

\end{thebibliography}
